\documentclass{gtpart}

\usepackage{pinlabel}
\usepackage{plain}
\usepackage{booktabs,amsmath}
\usepackage{amssymb}
\usepackage{amsthm}
\usepackage{textcomp}
\usepackage[utf8]{inputenc}
\usepackage{graphicx}
\graphicspath{{pics/}}
\usepackage{import}
\usepackage{curves}
\usepackage{xcolor}
\usepackage{float}
\usepackage{pdfpages}
\usepackage{changepage}
\usepackage{cleveref}
\usepackage{nicefrac}
\usepackage[toc,page]{appendix}
\usepackage{tikz}
\usetikzlibrary{calc,angles,patterns}
\usepackage{csquotes}

\usepackage{chngcntr}

\usepackage{fancyhdr}

\usepackage{plain}

\usepackage{mathptmx}
\usepackage[english]{babel}
\usepackage{amsthm}
\usepackage{array}
\usepackage{mathptmx}
\usepackage{array}
\usepackage{datetime}

\usepackage{enumerate}
\usepackage{graphicx}
\usepackage{lmodern}
\usepackage{mathrsfs}

\usepackage{cite}

\usepackage[T2A,T1]{fontenc} 
\DeclareSymbolFont{cyrillic}{T2A}{cmr}{m}{n}
\DeclareMathSymbol{\loba}{\mathalpha}{cyrillic}{203} 

\newcommand{\HH}{\mathbb{H}}
\newcommand{\CC}{\mathbb{C}}
\newcommand{\QQ}{\mathbb{Q}}

\newcommand{\NN}{\mathbb N}

\newcommand*{\ee}{\mathrm{e}}

\DeclareMathOperator{\PSL}{PSL}

\newtheorem*{theorem*}{Theorem}
\newtheorem*{main*}{Theorem}
\newtheorem*{proposition*}{Proposition}
\newtheorem*{properties*}{Properties}
\newtheorem{corollary}{Corollary}
\newtheorem*{corollary*}{Corollary}
\newtheorem*{lemma*}{Lemma}

\newtheorem*{remark*}{Remark}
\newtheorem*{conj*}{Conjectures}
\newtheorem{proposition}{Proposition}

\newtheorem*{example*}{Example}

\newtheorem*{examples*}{Examples}
\newtheorem{remark}{Remark}

\newtheorem*{definition*}{Definition}
\def\loba{\loba}

\usepackage{chngcntr}
\counterwithout*{table}{subsection}

\title[The non-arithmetic cusped hyperbolic 3-orbifold of minimal volume]{The non-arithmetic cusped hyperbolic\\ 3-orbifold of minimal volume
}

\author{Simon T. Drewitz}
\givenname{Simon T.}
\surname{Drewitz}
\address{Department of Mathematics\\
University of Fribourg\\
CH-1700 Fribourg\\Switzerland}
\email{s.t.drewitz@dunelm.org.uk}

\author{Ruth Kellerhals}
\givenname{Ruth}
\surname{Kellerhals}
\address{Department of Mathematics\\
University of Fribourg\\
CH-1700 Fribourg\\Switzerland}
\email{Ruth.Kellerhals@unifr.ch}

\subject{primary}{msc2010}{51M10, 20F55, 51M25}
\subject{secondary}{msc2010}{11F06, 20G20, 11M06}

\newtheorem{lem}{Lemma}

\theoremstyle{definition}

\makeop{Homo}
\numberwithin{equation}{section}

\def\loba{\hbox{\rm J\kern-1pt I}}
\def\vol{\hbox{vol}}

\usepackage{tikz}
\usetikzlibrary{calc, angles}
\tikzset{orthobullet/.style={circle, fill=black, draw=black,
                              minimum size=5pt, inner sep=0pt}}

\newcommand{\oneoverdalignedthree}{%
  \begin{tikzpicture}[scale=1]
    \pgfmathsetmacro{\d}{1}
    \pgfmathsetmacro{\height}{sqrt(3)/2}
    \pgfmathsetmacro{\heightb}{1/sqrt(3)}
    \draw (210:\heightb) -- (330:\heightb) -- (90:\heightb) -- cycle;
    \draw (210:\heightb) circle[radius=0.5];
    \draw (330:\heightb) circle[radius=0.5];
    \draw (90:\heightb) circle[radius=0.5];
    \draw[dotted] (210:\heightb) -- +(30:\height);
    \draw[dotted] (330:\heightb) -- +(30+120:\height);
    \draw[dotted] (90:\heightb) -- +(30-120:\height);
  \end{tikzpicture}
}

\newcommand{\oneoverdaligned}[1]{%
  \begin{tikzpicture}[scale=1]
    \pgfmathsetmacro{\d}{2*cos(180/#1)}
    \pgfmathsetmacro{\height}{\d*sqrt(3)/2}
    \pgfmathsetmacro{\heightb}{\d/sqrt(3)}
    \pgfmathsetmacro{\w}{\d-1/\d}
    \draw (210:\heightb) -- (330:\heightb) -- (90:\heightb) -- cycle;
    \draw (210:\heightb) circle[radius=0.5];
    \draw (330:\heightb) circle[radius=0.5];
    \draw (90:\heightb) circle[radius=0.5];
    \pgfmathsetmacro{\dkglobal}{\d}
    \pgfmathsetmacro{\rkglobal}{1/\d}
    \foreach \k in {1,...,\numexpr (#1-2)/2\relax}
    {
      \foreach \x in {210, 330, 90}
      {
        \foreach \y in {-150, 150}
        {
          \draw ($(\x:\heightb) + (\x+\y:1/\dkglobal)$)
            circle[radius=0.5*\rkglobal^2];
        }
      }
      \pgfmathparse{\d-1/\dkglobal}
      \global\let\dkglobal=\pgfmathresult
      \pgfmathparse{\rkglobal/\dkglobal}
      \global\let\rkglobal=\pgfmathresult
    }
    \draw[dotted] (210:\heightb) -- +(30:\height);
    \draw[dotted] (330:\heightb) -- +(30+120:\height);
    \draw[dotted] (90:\heightb) -- +(30-120:\height);
  \end{tikzpicture}
}

\begin{document}
\begin{abstract}    
We show that the 1-cusped quotient of the hyperbolic space $\HH ^3$ by the
tetrahedral Coxeter group $\Gamma_*=[5,3,6]$ has minimal volume among all non-arithmetic cusped hyperbolic 3-orbifolds,
and as such it is uniquely determined. Furthermore, the lattice $\Gamma_*$ is incommensurable to any Gromov-Piatetski-Shapiro type lattice. Our methods have their origin in the work of C. Adams \cite{Adams1, Adams2}. We extend considerably this approach via the geometry of the underlying horoball configuration induced by a cusp.


\smallskip
\noindent \textbf{Keywords.} Hyperbolic space, non-arithmetic lattice,
Coxeter group, horoball diagram, volume.
\end{abstract}

\maketitle

\qquad\small\today


\font\titel=cmbx10 at 20 truept
\font\titelsl=cmsl10 at 20 truept
\font\titelrm=cmr10 at 20 truept
\font\mittel=cmr10 at 14 truept
\font\bfmittel=cmbx8 at 14 truept
\font\medklein=cmr10 at 12 truept
\font\klein=cmr8
\font\kklein=cmr9
\font\bfklein=cmbx8
\font\bfkklein=cmbx9
\font\itklein=cmsl10 at 10 truept

\font\refklein=cmr12 at 11 truept
\font\titel=cmbx10 scaled\magstep2

\def\bar{\overline}
\def\vol{\hbox{vol}}
\def\t{\theta}
\def\vol{\hbox{vol}}

\def\<{\langle}
\def\>{\rangle}

\def\Li{\mathop{\rm Li}\nolimits}
\def\IM{\mathop{\rm Im}\nolimits}

\def\gd{\hbox{$\bullet\kern-6pt$ --\kern-5pt---\kern-5pt---
\kern-6pt}}
\def\gr #1{\hbox{$\bullet\kern-6pt $ ---\kern-5pt\raise
6pt\hbox{$\scriptstyle #1$}\kern-8pt ---\kern-7pt ---\kern-6pt --\kern-0pt}}
\def\grs #1{\hbox{$\bullet\kern-4pt $ --\kern-3pt\raise
9pt\hbox{$#1$}\kern-13pt ---\kern-6pt ---\kern-5pt --\kern0pt}}

\def\grastrich #1{\hbox{$\bullet\kern-2pt $ ---\kern-3pt\raise
8pt\hbox{$\alpha'_#1$}\kern-7pt --\kern-6pt ---\kern-5pt ---\kern+1pt}}
\def\grinf{\hbox{$\bullet\kern-2pt $ ---\kern-4pt\raise
6pt\hbox{$\infty$}\kern-11pt --\kern-6pt ---\kern-5pt ---\kern+1pt}}
\def\grh #1{\hbox{$\bullet\kern-2pt $ ---\kern-3pt\raise
8pt\hbox{$#1$}\kern-7pt --\kern-6pt ---\kern-5pt ---\kern+1pt}}
\def\gra #1{\hbox{$\bullet\kern -2pt$ --\raise
8pt\hbox{$\alpha_{#1}$}\kern-13pt
---\kern-4pt\kern-4pt ---\kern-5pt ---\kern +2pt}}
\def\graa #1{\hbox{$\bullet\kern -2pt$ ---\kern-6pt\raise
8pt\hbox{$\alpha_{#1}$}\kern-14pt
---\kern-4pt\kern-4pt ---\kern-5pt ---\kern +2pt}}
\def\graablack #1{\hbox{$\bullet\kern -2pt$ ---\kern-6pt\raise
8pt\hbox{$\alpha_{#1}$}\kern-14pt
---\kern-4pt\kern-4pt ---\kern-5pt ---\kern +2pt$\bullet$}}
\def\diamondgra #1{\hbox{$\diamond\kern -2pt$ --\raise
8pt\hbox{$\alpha_{#1}$}\kern-15pt
---\kern-4pt\kern-4pt ---\kern-5pt ---\kern +2pt}}
\def\gradiamond #1{\hbox{$\circ\kern -2pt$ --\raise
8pt\hbox{$\alpha_{#1}$}\kern-13pt
---\kern-4pt\kern-4pt ---\kern-5pt ---$\diamond$}}
\def\diamondgr #1{\hbox{$\diamond\kern -2pt$ ---\raise
8pt\hbox{${#1}$}\kern-13pt
---\kern-4pt\kern-4pt ---\kern-5pt ---\kern +2pt}}
\def\grdiamond #1{\hbox{$\circ\kern -2pt$ --\raise
8pt\hbox{${#1}$}\kern-13pt
---\kern-4pt\kern-4pt ---\kern-5pt ---\kern+2pt$\diamond$}}
\def\blackgra #1{\hbox{$\bullet\kern -2pt$ --\raise
8pt\hbox{$\alpha_{#1}$}\kern-13pt
---\kern-4pt\kern-4pt ---\kern-5pt ---\kern +2pt}}
\def\grablack #1{\hbox{$\circ\kern -2pt$ --\raise
8pt\hbox{$\alpha_{#1}$}\kern-13pt
---\kern-4pt\kern-4pt ---\kern-5pt ---\kern +1pt$\bullet$}}
\def\blackgr #1{\hbox{$\bullet\kern-2pt $ ---\kern-3pt\raise
8pt\hbox{$#1$}\kern-7pt --\kern-6pt ---\kern-5pt ---\kern+1pt}}
\def\grblack #1{\hbox{$\circ\kern-2pt $ ---\kern-3pt\raise
8pt\hbox{$#1$}\kern-7pt --\kern-6pt ---\kern-5pt ---\kern+1pt$\bullet$}}
\def\grslash{\hbox{$\bullet$\kern-1pt\raise8pt\hbox{$\diagup$}\raise14pt\hbox{$\bullet$}\raise8pt\hbox{$\diagdown$} \kern-3pt$\bullet$\kern0pt---\kern-3pt---\kern0pt$\bullet$}}
\def\gra3tilde{\hbox{$\bullet$\kern-2pt\raise8pt\hbox{$\diagup$}\kern-1pt\raise14pt\hbox{$\bullet$}\raise8pt\hbox{$\diagdown$} \kern-5pt$\bullet$\kern-28pt---\kern-3pt---\kern-2pt-}}

\def\l#1{\displaystyle{\tilde{}\,}\hbox{$#1$}}
\def\r#1{\hbox{$#1$}\raise-8pt\hbox{${\,\tilde{}}$}}

\def\Vor{\parindent=17pt\par\hang\textindent}
\def\ad#1{#1}
\def\add#1{[#1]}


\section{Introduction}\label{sec:intro}
Let $\mathbb H^3$ be the hyperbolic $3$-space viewed in the
upper half space $U^3$ of Poincar\'e, and let $\hbox{Isom}\mathbb H^3$
be its isometry group. A cusped hyperbolic $3$-orbifold $V$ is
the quotient of $\mathbb H^3$ by a non-cocompact lattice $\Gamma\subset\hbox{Isom}\mathbb H^3$, that is,
by a discrete group of hyperbolic isometries with a non-compact fundamental polyhedron of finite volume. In particular, $\Gamma$ contains a non-trivial parabolic subgroup whose elements fix a point $q$ on the boundary $\partial\mathbb H^3$. Without loss of generality, assume that $q=\infty$ with stabiliser $\Gamma_{\infty}<\Gamma$. The group
$\Gamma_{\infty}$ gives rise to precisely invariant subsets which are horoballs
centred at $\infty$. They can be arranged in such a way that there is a maximal horoball $B_{\infty}$
touching some of its $\Gamma$-images, called full-sized horoballs, and which project to a maximal cusp $C\subset V$ of finite volume in $V$. Hence, the number of non-conjugate parabolic subgroups of $\Gamma$
equals the number of maximal cusps in $V$.
The orthogonal projection of the full-sized horoballs to the horosphere $\partial B_{\infty}=H_{\infty}$ yields a horoball-packing with a characteristic horoball diagram on $H_{\infty}$. In this way, a combination of results from crystallography and a density result of K. B\"or\"oczky \cite[Theorem 4]{Boer} allows one to deduce lower volume bounds.

Using this picture, R. Meyerhoff \cite{Mey2} identified the 1-cusped quotient space $\HH^3/[3,3,6]$ as the minimal volume cusped hyperbolic $3$-orbifold. The Coxeter group with symbol
$[3,3,6]$ generates the group of symmetries of an ideal regular tetrahedron $S_{reg}^{\infty}$ of dihedral angle $\frac{\pi}{3}$. Furthermore,
it
is an arithmetic group commensurable to the Eisenstein modular group PSL$(2,\mathbb Z[\omega])$ where $\omega= (-1+\sqrt{-3})/2$
is a primitive cubic root of unity.
The orbifold $\HH^3/[3,3,6]$ is covered by the well-known (non-orientable) $1$-cusped Gieseking manifold $G$ which, by a result of Adams \cite{Adams4}, is of minimal volume among all cusped hyperbolic $3$-manifolds.

The result of Meyerhoff was considerably extended by Adams \cite[Theorem 6.1, Corollary 6.2]{Adams1} who identified the
six cusped (orientable and non-orientable) hyperbolic $3$-orbifolds of smallest volume. Crucial in Adams' proof was the assumption of the strict upper volume bound of $\frac{1}{4}\,\vol(S_{reg}^{\infty})$ in the case of orientable orbifolds, implying that such an orbifold has only one cusp and a single orbit of full-sized horoballs modulo the action of the corresponding stabiliser.
In \cite{NR}, W. Neumann and A. Reid characterised Adams' spaces and showed that they are all arithmetically defined.

In this work, we consider {\it non-arithmetic} cusped hyperbolic $3$-orbifolds and prove the following result.
\begin{theorem*}\label{thm:main}
Among all non-arithmetic cusped hyperbolic 3-orbifolds, the 1-cusped quotient space $V_*$
of $\HH^3$ by the tetrahedral Coxeter group $[5,3,6]$ has minimal volume. As such the orbifold $V_*$ is unique, and its
volume $v_*$ is given explicitly by \eqref{eq:vol536}.
\end{theorem*}

The Coxeter group $[5,3,6]$ gives rise to the group of symmetries of
an ideal regular dodecahedron $D_{reg}^{\infty}\subset\HH^3$ of dihedral angle
$\frac{\pi}{3}$. By applying different face identifications to
$D_{reg}^{\infty}$, the orbifold $V_*$ admits several
non-isometric non-arithmetic cover manifolds; see \cite{Everitt}.

Notice that the non-arithmetic orbifold $V_*=\HH^3/[5,3,6]$ does not relate to a Gromov--Piatetski-Shapiro construction.
In fact, the Coxeter tetrahedron associated to $[5,3,6]$ is not {\it splittable} in the
sense of \cite[Section 6.2, Example 6.10]{Stover}. Therefore, by \cite[Lemma 6.9]{Stover},
the group $\Gamma_*$ is incommensurable to any Gromov--Piatetski-Shapiro type lattice.

By considering the smooth case of non-arithmetic cusped hyperbolic 3-manifolds, the above theorem allows us to deduce the following result
in terms of the non-arithmetic tetrahedral Coxeter group.

\begin{proposition*}
The fundamental group of a non-arithmetic cusped hyperbolic 3-manifold of minimal volume is incommensurable to the Coxeter group $[5,3,6]$; its volume is smaller than or equal to $24\cdot \hbox{\rm covol}([(3^3,6)])\approx 8.738570$.
\end{proposition*}

Furthermore and as a by-product of the geometric methods used to prove the Theorem, 
we obtain the following two-dimensional analogue which has possibly been overlooked so far.
\begin{proposition*}\label{prop:main-2dim}
Among all non-arithmetic cusped hyperbolic $2$-orbifolds, the 1-cusped quotient space $V_*$
of $\HH^2$ by the triangle Coxeter group $[5,\infty]$ has mi\-nimal area. As such the orbifold $V_*$ is unique, and its
area is given by $\frac{3\pi}{10}$.
\end{proposition*}
For fixed dimension $n$ with $2\le n\le 9$, the cusped hyperbolic $n$-orbifold of minimal volume
is known and intimately related to an arithmetic Coxeter simplex group (see \cite{HK} and \cite {Hild1, Hild2}).
In comparison to this, it is
much more difficult to identify the non-arithmetic
cusped hyperbolic $n$-orbifolds of minimal volume for $n\ge4$ by providing a presentation of their fundamental groups. In view of the Gromov--Piatetski-Shapiro construction,
observe that non-arithmetic non-cocompact hyperbolic Coxeter simplex groups do not exist anymore for $n\ge4$.

This work is structured as follows. In Section \ref{sec:chapter2} we present the necessary background about cusped hyperbolic orbifolds, their (non-)arithmeticity and the realisation as quotients by hyperbolic Coxeter groups. We provide a short overview about volume computations for hyperbolic (truncated) tetrahedra and finish by presenting the necessary
information about horoball packings and cusp densities.
Section \ref{sec:chapter3} contains the proof of our Theorem which consists of several steps. We first show that a non-arithmetic cusped hyperbolic 3-orbifold $\HH^3/\Gamma$ of minimal volume has exactly one cusp $C$, and then, that $C$ is a rigid cusp of type $\{2,3,6\}$ or $\{2,4,4\}$. We study these two cases separately and have also to distinguish -- in contrast to Adams' work \cite{Adams1} --
whether there is one or more equivalence classes of a full-sized horoball $B$ covering $C$
with respect to $\Gamma_{\infty}$. An essential aspect is the view of the $\Gamma_{\infty}$-periodic horoball packing induced by $B$ by means of its horoball diagram. The possible
configurations of full-sized horoballs and their hierarchy allow us to estimate mi\-nimal distances of their centres and
to derive lower volume bounds. In all cases, we are able to identify fundamental polyhedra for $\Gamma$ or to exclude groups when they are arithmetic or of a too big covolume.
In Section \ref{sec:chapter4}, we discuss briefly the related results given by the two propositions above, both due to the first author \cite{Drewitz}. In the parts \ref{subsec:appendix1} and \ref{subsec:appendix2} of the Appendix, we describe certain small volume orientable orbifolds with precisely one cusp of type $\{2,3,6\}$ and $\{2,4,4\}$, respectively, whose associated parabolic groups give rise to only one equivalence class of full-sized horoballs. In the case of $\{2,3,6\}$, we correct the corresponding construction and result of Adams \cite[p. 10]{Adams1}.
\subsection*{Acknowledgements}
The authors would like to thank Colin Adams for helpful discussions.
The second author was partially
supported by Schweizerischer Nationalfonds 200021--172583.

\section{Non-arithmetic cusped hyperbolic 3-orbifolds}\label{sec:chapter2}
\subsection{Cusps of hyperbolic orbifolds}\label{cusp}
Consider the hyperbolic space $\mathbb H^3$ in the upper half space $U^3=\{\,(x,y,t)\in\mathbb R^3\mid t>0\,\}$ equipped
with the metric $\,ds^2=\frac{dx^2+dy^2+dt^2}{t^2}\,$. Points
on the boundary $\partial U^3=\{(x,y,0)=(u,0)\in\mathbb R^3\}\cup\{\infty\}=\mathbb R^2\cup\{\infty\}$ are called {\it ideal} points, and points in $\mathbb R^3$ with $t<0$ are called {\it ultraideal}.

Let $\,\Gamma\subset\hbox{Isom}\mathbb H^3$ be a non-cocompact lattice, that is, $\Gamma$ is a discrete group with a non-compact fundamental polyhedron $P\subset\HH^3$
of finite volume.
Then, the quotient space $V=\HH^3/\Gamma$ is a {\it cusped} hyperbolic 3-orbifold of finite volume
which is a smooth
manifold if the group $\Gamma$ has no torsion elements.
By Selberg's lemma (see \cite[Theorem 7.5.7]{Rat-book}, for example),
$\Gamma$ always has a finite index
subgroup $\Lambda$ which is torsion-free. In particular, $\Gamma$ and $\Lambda$ are
{\it commensurable} groups in $\hbox{Isom}\mathbb H^3$, that is, the
intersection of $\Gamma$ with some conjugate of $\Lambda$ in $\hbox{Isom}\mathbb H^3$
is of finite index in both groups. Recall that commensurability is an equivalence relation
preserving properties such as cocompactness, finite covolume and arithmeticity.

Each cusp $C$ of $V$
is of the form $\,B_q/\Gamma_q$ where $B_q\subset\HH^3$ is a horoball based at an ideal point
$\,q\in\partial\HH^3$ where $\Gamma_q\subset\Gamma$ is the (non-trivial) stabiliser of $q$ in
$\Gamma$. Enlarge $C$ so that it touches either itself or another cusp of $V$. Such a cusp
is called a {\it maximal} cusp of $V$.
Without loss of generality, we will assume that a maximal cusp of $V$ is covered by the horoball
$\,B_{\infty}=\{(x,y,t)\in\mathbb R^3\mid t>1\}$
based at $\infty$, with distance 1 from the ground space $\{t=0\}$, and bounded by the
horosphere
$H_{\infty}=\{(x,y,t)\in\mathbb R^3\mid t=1\,\}$. Recall that the induced metric $ds^2\mid_{t=1}\,$
on $H_{\infty}$ coincides with the Euclidean
metric with distance function denoted by $d_0$.
Since the cusp $C$ is maximal, there are horoball images $\gamma(B_{\infty})$
with $\gamma\in\Gamma$ not fixing $\infty$ whose closures touch $B_{\infty}$. These images are called {\it
full-sized} horoballs. By looking at their orthogonal projections onto the horosphere
$H_{\infty}$, we get horodisks whose centres coincide with the touching points of the corresponding full-sized horoballs. It will be convenient to identify full-sized horoballs and their base points with their horodisks and centres and vice versa.
We always suppose, without loss of generality, that one full-sized horoball is based at the origin $0$ of $\{t=0\}$.

The group $\Gamma_{\infty}$ is a {\it crystallographic} group acting cocompactly
by Euclidean isometries on $\{t=0\}$ (and on $H_{\infty}$). As such it contains a {\it translation lattice} $L\subset\Gamma_{\infty}$
of rank two, with minimal translation length $\tau\ge1$, and a finite subgroup $\phi\subset O(2)$, called the {\it point group},
which is a subgroup of the automorphism group $\hbox{Aut}(L)$. The latter group consists of all Euclidean
isometries fixing the origin and mapping
$L$ onto itself.

The orthogonal projection of the full-sized horoballs onto the horosphere
$H_{\infty}$ provides a sphere packing by balls of diameter 1 of the Euclidean plane $\{t=1\}$.
Recall that the densest packing of the Euclidean plane is achieved by the hexagonal lattice packing
where each ball is surrounded by six balls.
In the sequel, planar Euclidean and spatial horoball packings of large local densities will play an important role.

The ideal fixed point $\infty$ can be seen as an ideal vertex of
a fundamental polyhedron $P$ of $\Gamma$ whose vertex link $P\cap H_{\infty}$
is related either to a triangle $\Delta=\{p,q,r\}$ with angles $\frac{\pi}{p},\frac{\pi}{q},\frac{\pi}{r}$ satisfying $\,(p,q,r)=(2,3,6)\,,\,(2,4,4)$ or $(3,3,3)$
or to a parallelogram (see \cite[Section 2]{Adams1}).

In this context, a cusp in $V$ is called {\it rigid} ~if Dehn filling cannot be performed, and otherwise it is called {\it non-rigid}.
In \cite{Adams2}, Adams showed that a cusp is rigid if and only if there are singular curves of order different from $2$ going directly out of the cusp.

The {\it type}
of a rigid cusp is denoted by $\{p,q,r\}$ according to the orders
$p,q$ and $r$ of its singular axes and
the triangular description mentioned above.
Assuming that there is at
least one non-rigid cusp, Adams found the three (uniquely determined) orientable
orbifolds of smallest, second smallest and
third smallest (limit) volumes; they all have arithmetic fundamental groups (see \cite[Chapter 4, Chapter 7]{Adams2} and the
proof of Proposition \ref{prop:case1} below).

The (non-)arithmeticity of a discrete group $\,\Gamma\subset\hbox{Isom}\mathbb H^3$ of finite covolume can be characterised by
the following fundamental property due to Margulis (see \cite[Theorem 10.3.5]{MR1}, for example).
Consider the {\it commensurator}
\[
\textrm{Comm}(\Gamma)=\{\,\gamma\in \textrm{Isom}\mathbb H^3\,\mid\,\Gamma\hbox{ and }\gamma \Gamma\gamma^{-1}
\hbox{ are commensurable}\,\}
\]
of $\Gamma$ in $\textrm{Isom}\mathbb H^3$. Then, the group
$\textrm{Comm}(\Gamma)$ is a discrete subgroup in $\textrm{Isom}\mathbb H^3$ containing $\Gamma$ with finite index if and only
if $\Gamma$ is non-arithmetic.
In particular, for $\Gamma$ non-arithmetic, the commensurator $\textrm{Comm}(\Gamma)$ is the maximal element
in the commensurability class of $\Gamma$ so that all non-arithmetic hyperbolic orbifolds and manifolds with fundamental groups
commensurable to $\Gamma$ cover a smallest common quotient.
\subsection{Hyperbolic Coxeter orbifolds}\label{Coxeter}
Arithmeticity can be described in a nice way for the class of {\it hyperbolic Coxeter groups} and the associated quotients called {\it Coxeter orbifolds}. Consider first a
{\it Coxeter polyhedron} in $\mathbb H^3$ which is a convex polyhedron
$P_C\subset\mathbb H^3$ of finite volume all of whose dihedral angles are integral submultiples of $\pi$.
A hyperbolic Coxeter group
$\Gamma_C\subset \textrm{Isom}\mathbb H^3$ (with fundamental polyhedron $P_C$)
is the discrete group generated by the (finitely many)
reflections in the facets of $P_C$.
Since the dihedral angles of $P_C$ are non-obtuse, Andreev's theorem (see \cite{Andreev1}, \cite{Andreev2}, \cite{Roeder}) provides
necessary and sufficient
conditions for its existence. In particular, there are infinitely many non-isometric Coxeter polyhedra in $\mathbb H^3$ but only
finitely many Coxeter tetrahedra. In fact, there are precisely 9 compact Coxeter tetrahedra and 23 non-compact ones (for a list with
their volumes, see \cite[pp. 347-348]{JKRT1}).
By Vinberg's seminal work \cite{V1}, \cite{V2} about hyperbolic Coxeter polyhedra in any dimension, many of their properties can be read off
from their Coxeter graphs and Gram matrices. Let $P_C\subset\mathbb H^3$ be a hyperbolic Coxeter polyhedron bounded
by $N\ge4$ geodesic planes $H_1,\ldots,H_N$.
Consider the $N\times N$ Gram matrix $G=G(P_C)=(g_{ij})$ of $P_C$ which is a real symmetric matrix with $g_{ii}=1$ and, for $i\not=j$,
\begin{equation}
\label{eq:Gram}
-g_{ij}=\left\{
\begin{array}{ll}
\cos\frac{\pi}{m_{ij}}&\textrm{if $H_i,H_j$ intersect at the angle $\frac{\pi}{m_{ij}}$ in $\mathbb H^3$},\\
1&\textrm{if $H_i,H_j$ meet at $\partial\mathbb H^3$},\\
\cosh l_{ij}&\textrm{if $H_i,H_j$ are at distance $l_{ij}$ in $\mathbb H^3$}.\\
\end{array}
\right.
\end{equation}

In case of many orthogonal bounding planes and small $N$ it is
convenient to represent $P_C$ by means of its {\it Coxeter graph} which is
a (weighted) graph
$\Sigma=\Sigma(P_C)$ of order $N$ defined as follows.
To each bounding plane $H$ of $P_C$
we associate a node $\nu$ in $\Sigma$. Two different nodes $\nu_i,\nu_j$ are
connected by an edge with a weight if the planes $H_i,H_j$ are not orthogonal. The weight equals $m_{ij}$
if $g_{ij}=-\cos\frac{\pi}{m_{ij}}$. In the special (and frequent) case $m_{ij}=3$, however, the edge carries no label.
An edge will be decorated by the symbol $\infty$ if $g_{ij}=-1$.
Edges related to disjoint planes with $g_{ij}<-1$ are replaced by dotted edges, and the
weights are usually omitted.
In order to describe a
Coxeter graph in an abbreviated way, we use the
{\it Coxeter symbol}. In particular, $[p,q,r]$
with integral components is associated to a linear Coxeter graph
describing a {\it hyperbolic Coxeter orthoscheme} with dihedral angles $\frac{\pi}{p},\frac{\pi}{q},\frac{\pi}{r}$, and the
Coxeter symbol $[(p^r,q^s)]$
describes a polyhedron with cyclic Coxeter graph having $r\ge1$
consecutive weights $p$ followed by $s$ consecutive weights $q$
(see \cite[Appendix]{JKRT1}, for example). In the sequel, we often
represent a Coxeter group by quoting the Coxeter symbol of its Coxeter polyhedron.

For the arithmeticity of hyperbolic Coxeter groups, there is a very efficient criterion due to Vinberg (see \cite[pp. 226-227]{V2}).
We quote it in the special case of a non-compact Coxeter polyhedron $P_C\subset\mathbb H^3$
with Gram matrix $G=(g_{ij})$ and with associated reflection
group
$\Gamma_C\subset \textrm{Isom}\mathbb H^3$. Write $2G=:(h_{ij})$ and form {\it cycles (of length $k$)} of the form
\begin{equation}\label{cycles}
h_{i_1i_2}h_{i_2i_3}\cdot\ldots\cdot h_{i_{k-1}i_k}h_{i_ki_1}\,,
\end{equation}
with distinct indices $i_j$ in $2\,G$. Then, $\Gamma_C$ is arithmetic with field of definition $\mathbb Q$ if and only if
all the cycles of $2\,G$ are rational integers.
Furthermore, if the Coxeter graph $\Sigma(P_C)$ contains no dotted edges,
then, by a result of Guglielmetti \cite[Proposition 1.13]{Gugliel}, $\Gamma_C$
is arithmetic if and only if all weights of $\Sigma(P_C)$ lie in $\{\infty,2,3,4,6\}$, and
each cycle of length at least 3 in $2\,G$ lies in $\mathbb Z$.

\begin{examples*}\label{ex-nonarithm}\hfill\\
1. Among all non-compact Coxeter orthoschemes $[p,q,r]$ in $\HH^3$, only $[5,3,6]$ defines a non-arithmetic reflection group, which we
denote by $\Gamma_*$. The associated Coxeter orbifold $V_*$ has one cusp.

2. The Coxeter tetrahedron $P_{\circ}$ with the cyclic graph $[(3^3,6)]$ yields a non-arithmetic reflection group, denoted $\Gamma_{\circ}$.
The associated Coxeter orbifold $V_{\circ}$ has 2 cusps.
\end{examples*}

\begin{remark}\label{incomm}
By \cite[Theorem 3]{JKRT2}, the groups $[5,3,6]$ and $[(3^3,6)]$ are not commensurable; in particular, their invariant trace fields are
different (see \cite[Section 13.2]{MR1}). Due to the graph
symmetry $[(3^3,6)]$, the Coxeter tetrahedron $P_{\circ}$ has a symmetry plane $H_r$ along which it can be dissected
into 2 isometric tetrahedra (of non-Coxeter type), each with one cusp. The group extension
$\Gamma_{\circ}^r:=[(3^3,6)]\ast C_r$ by the cyclic group $C_r$ generated by the half-turn $r$
with respect to $H_r$ is a non-arithmetic discrete group containing $[(3^3,6)]$ with index 2 and giving rise to the 1-cusped
quotient space $V_{\circ}^r$.
\end{remark}

\subsection{Volumes of non-compact hyperbolic tetrahedra}\label{volume}
Consider a finite volume orthoscheme $R=R(\alpha,\beta)\subset \mathbb H^3$ with one ideal vertex $q$
and dihedral angles $\alpha,\beta,\beta'=\frac{\pi}{2}-\beta$
such that $\beta'\le\alpha<\frac{\pi}{2}$ (see Figure \ref{fig:ortho}). More precisely, $R$ is a tetrahedron bounded by geodesic planes $H_1,\ldots,H_4$ with opposite vertices $p_1=q,p_2,p_3,p_4$
such that the Gram matrix $G(R)=(g_{ij})$ is given by
\[
 G(R)=\begin{pmatrix}
 1 & -\cos\alpha & 0 & 0 \\
 -\cos\alpha & 1 & -\cos\beta & 0 \\
 0 & -\cos\beta & 1 & -\cos\beta' \\
 0 & 0 & -\cos\beta' & 1 \\
\end{pmatrix}\,\,.
\]
The matrix $G(R)$ is of signature $(3,1)$ and has -- beside positive definite principal submatrices --
exactly one positive semi-definite principal submatrix of rank 2
(at the lower right) characterising the ideal vertex $q$ of $R$.

\begin{figure}[H]
  \hspace*{-2.5cm}
    \def\svgwidth{\textwidth}
    \import{pics/}{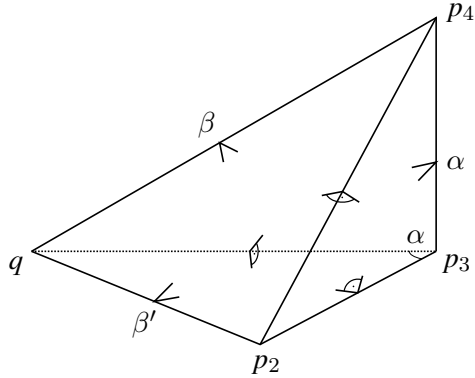}
  \caption{An orthoscheme $R(\alpha,\beta)$ with ideal vertex $q$}
  \label{fig:ortho}
\end{figure}

It is well-known that the scissors congruence group ${\cal P}(\mathbb H^3)$ is generated by the classes of orthoschemes $R(\alpha,\beta)$
(see \cite[Chapter 3]{K2}, for example). Hence, the volume of any polyhedron $P\subset\HH^3$ is a linear combination of
volumes of orthoschemes of type $R(\alpha,\beta)$.

The volume of $R(\alpha,\beta)$ is given by the expression
(see \cite{K0}, for example)
\begin{equation}\label{eq:vol-ortho}
    \vol(R(\alpha,\beta))=\frac{1}{4}\,
       \left\{\loba(\frac{\pi}{2}+\alpha-\beta)-\loba(\frac{\pi}{2}+\alpha+\beta)+2\,\loba(\beta)\,\right\}\,\,,
\end{equation}
    where $\loba(x)=\frac{1}{2}~\sum\limits_{r=1}^{\infty}\frac{\sin(2rx)}{r^2}=
 -\int\limits_0^{x}\,\log\vert2~\sin t\vert\,dt\,,\,x\in\mathbb R,\,$ is Lobachevsky's function. The function
 $\loba(x)$ is odd, $\pi$-periodic and satisfies the distribution formula (see \cite[Appendix]{JKRT1}, for example)
 \[
\frac{1}{k}\,\loba(kx)=\sum_{r=0}^{k-1}\loba\left(x+
\frac{r\pi}{k}\right)\,\,,\,\,k\in\NN\,\,.
 \]
As a consequence, for the volume of an ideal tetrahedron $T(\alpha,\beta,\gamma)$ with dihedral angles $\alpha,\beta,\gamma$ along edges emanating from an ideal vertex such that $\alpha+\beta+\gamma=\pi$, one deduces that
\begin{equation}\label{eq:vol-ideal}
\vol(T(\alpha,\beta,\gamma))=\loba(\alpha)+\loba(\beta)+\loba(\gamma)\,\,.
\end{equation}
As a special case,
an ideal regular tetrahedron $S_{reg}^{\infty}$ (of dihedral angle $\frac{\pi}{3}$) is of volume $\mu_3=\vol(S_{reg}^{\infty})=24\,\vol(R(\frac{\pi}{3},\frac{\pi}{3}))$. An ideal regular octahedron $O_{reg}^\infty$ (of dihedral angle $\frac{\pi}{2}$) can be dissected into $48$ orthoschemes
 $R(\frac{\pi}{3},\frac{\pi}{4})$ so that the volume $\omega_3$ of $O_{reg}^\infty$ is given by
 $\,\omega_3=48\,\vol(R(\frac{\pi}{3},\frac{\pi}{4}))$.

\begin{examples*}\label{volumes}\hfill\\
The volume of the (arithmetic) Coxeter orthoscheme $R(\frac{\pi}{3},\frac{\pi}{3})=[3,3,6]$ equals
 \begin{equation}\label{eq:vol336}
   \frac{1}{8}\,\loba(\frac{\pi}{3}) \approx 0.042289\,\,.
 \end{equation}
 In particular, we deduce that $\mu_3=\hbox{\rm vol}(S^{\infty}_{reg})=3\,\loba(\frac{\pi}{3})\approx 1.014942$.

The volume of the (arithmetic) Coxeter orthoscheme $R(\frac{\pi}{3},\frac{\pi}{4})=[3,4,4]$ equals
 \begin{equation}\label{eq:vol344}
   \frac{1}{6}\,\loba(\frac{\pi}{4}) \approx 0.076330\,\,.
 \end{equation}
 As a consequence, we have that $\omega_3=\vol(O_{reg}^\infty)=8\,\loba(\frac{\pi}{4})\approx 3.663862$.

The volume of the (non-arithmetic) Coxeter orthoscheme $R(\frac{\pi}{5},\frac{\pi}{3})=[5,3,6]$ equals
 \begin{equation}\label{eq:vol536}
   \frac{1}{2}\,\loba(\frac{\pi}{3})+ \frac{1}{4}\,
    \left\{\,\loba(\frac{\pi}{6}+\frac{\pi}{5})+\loba(\frac{\pi}{6}-\frac{\pi}{5})\right\} \approx 0.171502\,\,.
 \end{equation}

By dissection, the volume of the four Coxeter tetrahedra with cyclic Coxeter symbol $[(3,6,3,m)]\,,\,3\le m\le6\,,$ can be determined quite easily. In particular,
the volume of the non-arithmetic Coxeter tetrahedron $[(3^3,6)]$ equals
 \begin{equation}\label{eq:volcycle}
   \frac{5}{8}\,\loba(\frac{\pi}{3})+\frac{1}{3}\,\loba(\frac{\pi}{4})\approx 0.364107\,\,,
 \end{equation}
\end{examples*}
which is the smallest one among the four.
\newline
For the volumes of {\it all} hyperbolic Coxeter tetrahedra, we refer to \cite[Appendix]{JKRT1}.

A generalisation of the volume formula \eqref{eq:vol-ortho}
for orthoschemes $R(\alpha,\beta)$ with one ideal vertex can be obtained by allowing that $\alpha+\beta<\frac{\pi}{2}$. This condition is equivalent to the assumption that the vertex $p_4$ as depicted in Figure \ref{fig:ortho} is ultraideal, that is, the face planes $H_q, H_2$ and $H_3$ opposite to $q$, $p_2$ and $p_3$ intersect (in $p_4$) outside of $\mathbb H^3\cup\partial \mathbb H^3$. In this situation, $R(\alpha,\beta)$ is of infinite volume. However, there is a unique hyperbolic plane $H$ intersecting orthogonally $H_q, H_2$ and $H_3$. In fact, in the projective model of $\mathbb H^3$, the plane $H$ is the {\it polar plane} associated to $p_4$. By truncating $R(\alpha,\beta)$ by means of $H$, we obtain a {\it simply truncated orthoscheme} denoted by $R_t(\alpha,\beta)$. By \cite[(34)]{K0}, the volume of $R_t(\alpha,\beta)$ is given analytically by the same formula \eqref{eq:vol-ortho}, however under the constraint $\alpha+\beta<\frac{\pi}{2}$. This fact combined with suitable dissection procedures can be applied to determine the volumes of various families of truncated polyhedra in $\mathbb H^3$.
\begin{example*}\label{truncated-ortho}\hfill\\
For $\,k,l\in\mathbb N$ with $\frac{1}{k}+\frac{1}{l}<\frac{1}{2}$,
consider a simply truncated Coxeter orthoscheme $R_t(\frac{\pi}{k},\frac{\pi}{l})$ with one ideal vertex $q$. Its Coxeter graph is given by
\[
\gr{k}\gr{l}\grs{\frac{2l}{l-2}}\bullet\cdot\cdots\bullet\,\,\,\,.
\]
In particular, the volume of the simply truncated Coxeter orthoscheme
$R_t(\frac{\pi}{k},\frac{\pi}{3})\,,\,k\ge7\,,$ is given by
 \begin{equation}\label{eq:trunc-36}
 \frac{1}{2}\,\loba(\frac{\pi}{3})+\frac{1}{4}\,\left\{\,\loba(\frac{\pi}{6}+\frac{\pi}{k})+\loba(\frac{\pi}{6}-\frac{\pi}{k})\,\right\}\,\,,
 \end{equation}
while the volumes of the
Coxeter family
$R_t(\frac{\pi}{k},\frac{\pi}{6})\,,\,k\ge4\,,$ are equal to
 \begin{equation}\label{eq:trunc-63}
 \frac{1}{2}\,\loba(\frac{\pi}{6})+\frac{1}{4}\,\left\{\,\loba(\frac{\pi}{3}+\frac{\pi}{k})+\loba(\frac{\pi}{3}-\frac{\pi}{k})\,\right\}\,\,.
 \end{equation}


In both cases, for $\,k\rightarrow\infty$, the limiting Coxeter polyhedron is a pyramid with two ideal vertices and Coxeter symbol $[\infty,3,6,\infty]$ and has
volume $\,\frac{5}{4}\,\loba(\frac{\pi}{3})\approx0.42289$.
\end{example*}

\begin{remark}\label{Schlaefli}
Notice that for the infinite families of polyhedra $R(\alpha,\beta)$ and $R_t(\alpha,\beta)$ with {\rm fixed} angle $\beta$, the volume is strictly increasing when the dihedral angle $\alpha$ decreases.
This is a direct consequence of Schl\"afli's formula for the volume differential; see \cite[Section 2]{K0}.
\end{remark}
\subsection{Some horoball geometry}\label{horo}
Let $V=\HH^3/\Gamma$ be a finite volume hyperbolic 3-orbifold  with a set of disjoint
cusps ${\cal C}=\{C_1,\ldots,C_m\}\,,\,m\ge1$. Let $C\in{\cal C}$ be a maximal cusp such that
$C=B_{\infty}/\Gamma_{\infty}$. Consider the image horoballs $\gamma(B_{\infty})$
with $\gamma\not\in\Gamma_{\infty}$ and project them
orthogonally to the horosphere $H_{\infty}=\{t=1\}$.
The full-sized horoballs
project onto Euclidean balls of diameter 1 and yield a periodic packing of the
Euclidean plane by equal balls. Any minimal configuration of
them providing the entire information (about translational and finite order symmetries) of
the crystallographic group $\Gamma_{\infty}$
yields a
{\it horoball diagram} or {\it cusp diagram} $D\subset H_{\infty}$
(see \cite[Figure 1 or Figure 6]{Adams1}, for example, and Section \ref{notations}).

Denote by $F\subset\mathbb R^2$ a fundamental polygon for the action of $\Gamma_{\infty}$
on the horosphere $H_{\infty}$, and let $\hbox{vol}_0(F)$ be its Euclidean area. Then, the volume
of $C$ can be expressed by
(see \cite[Section 5]{Cox2})
\begin{equation}\label{eq:cusp-volume}
 \vol(C)=\frac{1}{2}\,\hbox{vol}_0(F)\,\,.
\end{equation}
Notice that if $\Gamma_{\infty}$ is a reflection group, its index two
subgroup $\Gamma^+_{\infty}$ of orientation preserving rotations yields twice the volume of \eqref{eq:cusp-volume}.

The concept of local density of a horoball $B$ covering $C$
with respect to its Dirichlet-\hbox{Vorono\v\i}
cell $D(B)$ leads to the following lower volume bound for $V$ in comparison with the total cusp volume $\vol({\cal C})=\sum_{r=1}^m\,\vol(C_r)$
(see \cite{Mey1} and \cite[Lemma 3.2]{K3}).
\begin{equation}\label{eq:density-bound}
 \vol(V)\ge\frac{\vol({\cal C})}{d_3(\infty)}\,\,,
\end{equation}
where $d_3(\infty)=\frac{\sqrt{3}}{2\,\mu_3}\approx0.853276$ is the simplicial horoball density
related to an ideal regular tetrahedron $S_{reg}^{\infty}$ of volume $\,3\,\loba(\frac{\pi}{3})$.
\newline
Observe that the bound \eqref{eq:density-bound} is sharp if the lift of each element of ${\cal C}$
to $\HH^3$ induces a {\it regular} horoball packing (see \cite[Section 2]{K3}).

In the case of a $1$-cusped orbifold with maximal cusp $C$, the cusp density $\delta(C)=\vol(C)/\vol(V)<1$ is
bounded from above by (see \eqref{eq:density-bound})
\begin{equation}\label{eq:cuspdensity-bound}
 \delta(C)\le d_3(\infty)\,\,.
\end{equation}

The following facts will be useful when studying horoball diagrams. 






\begin{lem}{\rm \cite[Lemma 4.4]{Adams2}}\label{horo0}
The centres of two tangent horoballs of radii $r_1$ and $r_2$ are at Euclidean distance $2\sqrt{r_1r_2}$.
\end{lem}

\begin{lem}{\rm \cite[Lemma 4.3]{Adams2}, \cite[Lemma 1]{HK}}\label{horo1}
Consider the horoball $B_{\infty}$ and a full-sized horoball $B_u$ based at $\,u\in\mathbb R^2$.
Denote by $l$ a geodesic with endpoints $u$ and $v\in \mathbb R^2\setminus\{u\}$, and let $\delta_0=d_0(u,v)$ be the
Euclidean distance from $u$ to $v$. Put $a=(u,1)$, and let $p$ be the intersection point of $l$ with $H_u=\partial B_u$.
Then, the induced distance $d(a,p)$ from $a$ to $p$ on $H_u$ is given by $\,d(a,p)=\frac{1}{\delta_0}$.
\end{lem}


\begin{corollary}{\rm \cite[Lemma 2]{HK}}\label{horo2}
Consider a horoball $B_u(h)$ of diameter $h$ in $\HH^3$. Then, the interior
of its upper hemisphere is an open disk of radius 1 with respect to the induced metric
on the boundary $H_u(h)$.
\end{corollary}

\begin{corollary}{\rm \cite[Corollary 7]{Hild2}}\label{horo3}
Let $B_u(h)$ and $B_v(k)$ be two horoballs with common touching point $p$ and of Euclidean diameter $h$ and $k$, respectively.
Denote by $\delta_0$ the Euclidean distance $d_0(u,v)$ and by $\delta=\delta_0/h$ the induced distance $d((u,h),(v,h))$
on the horosphere based at $\infty$ and at Euclidean height $h$.
Then, the following identity holds.
\[
 \frac{\sqrt{h}}{\sqrt{k}}=\frac{\delta_0}{k}=\frac{h}{\delta_0}=\frac{h\,\delta}{k}=\frac{1}{\delta}\,\,.
\]
\end{corollary}

\begin{corollary}\label{horo4}
Consider a full-sized horoball $B_u$ and a horoball $B_v(k)$ of diameter $k$ based at $\,v\in\mathbb R^2$ which touches $B_u$.
Denote by $\delta_0$ the Euclidean distance $d_0(u,v)$.
Then, the diameter $k$ of $B_v$ equals $\delta_0^2$.
\end{corollary}

\begin{lem}{\rm \cite[Lemma 4.6]{Adams2}, \cite[Lemma 1.2]{Adams1}}\label{horo5}
Consider two horoballs $B_u(h)$ and $B_v(k)$ based at $\,u,v\in\mathbb R^2$ of diameter $h$ and $k$, respectively,
which cover the cusp $C$ of $V$. Suppose that $B_u(h)$ and $B_v(k)$ are not tangent, and denote by $d_0(u,v)=:r$
the Euclidean distance between their base points $u$ and $v$.
Then, there exists a horoball of diameter $\frac{hk}{r^2}$ covering $C$.
\end{lem}

Consider the action of the crystallographic subgroup $\Gamma_{\infty}\subset\Gamma$
on the set of all full-sized horoballs. The following facts are fairly obvious (see \cite[Lemma 3.1 and Lemma 2.1]{Adams1}).
\begin{lem}\label{3balls}
Suppose that $\Gamma_{\infty}$ identifies all full-sized horoballs. If the shortest translation length satisfies
$\tau>1$,
then there cannot be a set of three full-sized horoballs which are pairwise tangent.
\end{lem}

\begin{lem}\label{elliptic}
Suppose that $\Gamma_{\infty}$ identifies all full-sized horoballs.
Then, every point of tangency between two $\Gamma$-equivalent horoballs lies on the axis of an order two
elliptic isometry in $\Gamma$ such that its axis is tangent to both horoballs.
\end{lem}

\section{Proof of the Theorem}\label{sec:chapter3}
In this section, we provide a proof in several steps of our main result stated as follows.
\begin{main*}\label{main-thm}
Among all non-arithmetic cusped hyperbolic 3-orbifolds, the 1-cusped quotient space $V_*$
of $\HH^3$ by the tetrahedral Coxeter group $[5,3,6]$ has minimal volume. As such the orbifold $V_*$ is unique, and its
volume $v_*$ is given explicitly by \eqref{eq:vol536}.
\end{main*}
For the proof, we adapt and generalise the strategies and results of Adams as developed in \cite{Adams1}, \cite{Adams2} and \cite{Adams3}.
\subsection{The non-rigid and multiply cusped cases}\label{step1}
Let $\mu_3=\vol(S_{reg}^{\infty})=3\,\loba(\frac{\pi}{3})\approx 1.014942$ and
$\omega_3=\vol(O_{reg}^{\infty})=8\,\loba(\frac{\pi}{4})\approx 3.663862$
be the volumes of an ideal regular tetrahedron
and of an ideal regular octahedron, respectively (see  \eqref{eq:vol336} and \eqref{eq:vol344}).

Denote by $V=\HH^3/\Gamma$ a non-arithmetic cusped orbifold of minimal volume.
By a result of Meyerhoff \cite{Mey2}, the smallest volume of {\it any} cusped hyperbolic 3-orbifold equals $\frac{\mu_3}{24}$ and
is realised in a unique way by the volume of the quotient of $\HH^3$ by the arithmetic tetrahedral Coxeter group $[3,3,6]$.
Therefore, the volume of $V$ has to satisfy the inequalities
\begin{equation}\label{eq:minvol-bound}
0.042289\approx\frac{\mu_3}{24}< \vol (V)\le\vol (V_*)=v_*
\approx 0.171502\,\,.
\end{equation}
\begin{proposition}\label{prop:case1}
A non-arithmetic cusped hyperbolic 3-orbifold $V$ of minimal volume has precisely one cusp, and this cusp is a rigid one.
\end{proposition}
\begin{proof}
In \cite{Adams3}, Adams considered the small volume spectrum of hyperbolic 3-orbifolds with $m\ge2$ cusps.
He proved that the four smallest (non-orientable) orbifolds
of this sort have volumes equal to
\begin{equation}\label{eq:minvol-mult}
 \frac{\mu_3}{6}\,\,,\,\,\frac{5\,\mu_3}{24}\,\,,\,\,\frac{\omega_3}{16}\quad\hbox{and}\quad\frac{\mu_3}{4}
\end{equation}
(see \cite[Corollary 2.6]{Adams3}). Comparing the four values in \eqref{eq:minvol-mult} with
the volume bound in \eqref{eq:minvol-bound} shows that all values except the first one
given by $\frac{\mu_3}{6}$
are strictly bigger than $\vol (V_*)$. Furthermore, in \cite[Lemma 2.2]{Adams3}, Adams identified the individual
multiply-cusped orbifolds with small volumes and proved their uniqueness. According to his proof \cite[pp. 155-156]{Adams3}, one easily deduces
that the value $\frac{\mu_3}{6}$
is the volume of the quotient space of $\HH^3$ by the {\it arithmetic} tetrahedral Coxeter group $[3,6,3]$.
Therefore, $V$ can have only one cusp.

Suppose that a hyperbolic 3-orbifold $V$ has a single cusp $C$, and that $C$ is non-rigid. By
the results of Adams \cite[Corollary 4.2 and Section 7]{Adams2} about limit volumes of orientable hyperbolic 3-orbifolds, the three smallest
volumes are $\frac{\omega_3}{12}\approx 0.305322$, $0.444457$ and $0.457983$. These values are realised in a unique way by orientable orbifolds
which are explicitly described in the proof of \cite[Lemma 7.1]{Adams2}. In fact, their fundamental groups are all arithmetic
and related to Bianchi groups
PSL$(2,{\cal O}_d)$ where ${\cal O}_d$ is the ring of integers of the imaginary quadratic number field $\mathbb Q(\sqrt{-d})$.
In particular, the
fundamental group of the orbifold with volume $\frac{\omega_3}{12}$ is the Bianchi group PSL$(2,{\cal O}_1)$ which is commensurable to
the tetrahedral Coxeter group $[3,4,4]$ and to the fundamental group of the Borromean Rings complement
(see also \cite[Section 9.2]{MR1}). As a consequence, a non-orientable hyperbolic 3-orbifold with a single, non-rigid cusp
of volume smaller than or equal to $\vol (V_*)$ is arithmetic.
\end{proof}
\subsection{The cusp type \texorpdfstring{$\{3,3,3\}$}{(3,3,3)}}\label{step2}
By Proposition \ref{prop:case1}, a non-arithmetic non-compact orbifold $V=\HH^3/\Gamma$ of minimal volume has only one cusp, and the cusp
is a rigid one of type $\{2,3,6\}\,,\,\{2,4,4\}$ or $\{3,3,3\}$. The next result allows us to exclude the type $\{3,3,3\}$
from further consideration.
\begin{proposition}\label{prop:case2}
A non-arithmetic cusped hyperbolic 3-orbifold $V$ of minimal volume cannot have a cusp of type $\{3,3,3\}$.
\end{proposition}
\begin{proof}
In \cite[Theorem 4.2]{Adams1}, Adams proved that an orientable hyperbolic 3-orbifold with a cusp of type $\{3,3,3\}$
has volume either $\,\frac{\mu_3}{6}\,,\,\frac{\mu_3}{3}\,,\,\frac{5\,\mu_3}{12}\,$ or at least $\,\frac{\mu_3}{2}$. Furthermore, he
showed that there are unique orbifolds whose volumes equal these first three values, and that they are the double covers
of certain {\it unique} orientable orbifolds with one cusp of type $\{2,3,6\}$ of small volume. He
provides an explicit description of the latter orbifolds as follows
(see \cite[p. 10]{Adams1}).

The unique orientable orbifold with one cusp of type $\{2,3,6\}$ and of volume $\,\frac{\mu_3}{12}$
is the orientable double cover of the arithmetic orbifold $\HH^3/[3,3,6]$ implying that the orbifold
with cusp of type $\{3,3,3\}$ and of volume
$\frac{\mu_3}{6}$ is the quotient of $\HH^3$ by the rotation subgroup of the arithmetic tetrahedral Coxeter group with symbol $\,[3,3^{[3]}]$.

The unique orientable
orbifold with one cusp of type $\{2,3,6\}$ and of volume $\,\frac{\mu_3}{6}$
is the orientable double cover of the quotient space of $\HH^3$ by the $\mathbb Z_2$-extension of the arithmetic Coxeter group $[3,6,3]$.
Let us add that this orbifold as well as its double cover with one cusp of type $\{3,3,3\}$ are {\it not} Coxeter orbifolds anymore.

Finally, observe that the third smallest value $\frac{5\,\mu_3}{12}\approx 0.422892$ in the above sequence
is strictly bigger than the covolume $2\,\vol (V_*)$
of the rotation subgroup of $[5,3,6]$. As a consequence, a
non-arithmetic cusped orbifold $V$ of minimal volume cannot have a cusp of type $\{3,3,3\}$.
\end{proof}
\subsection{Some notations}\label{notations}
Before we continue with the proof, let us fix some notations.
Consider a non-compact orbifold $V=\HH^3/\Gamma$ with precisely one (maximal)
cusp, denoted by $C$, such that $C$ is a rigid one.
We assume that the {\it cusp point} (or {\it parabolic fixed point})
associated to $C$ is based at $\infty$ and that
$C=B_{\infty}/\Gamma_{\infty}$ with $B_{\infty}$ the horoball at height 1 from the ground plane $\{t=0\}$.
Denote by $\tau\ge1$ the minimal translation length induced by the translation lattice $L\subset \Gamma_{\infty}$.

Suppose that $C$ is of type $\{p,q,r\}$ with $(p,q,r)=(2,3,6)$ or $(2,4,4)$ (see Section \ref{cusp}).
In particular, $C$ contains 3 singular (rotation) axes
of orders $2,q,r$, respectively.
By Proposition \ref{prop:case2}, these
assumptions do hold if $V$ is non-arithmetic of minimal volume.

Consider a cusp diagram $D\subset H_{\infty}$ of $\Gamma_{\infty}$. If $C$ is of type $\{2,3,6\}$, then
$D$ is a Euclidean regular triangle of edge length $\tau$, and if $C$ is of type $\{2,4,4\}$, then
$D$ is given by a Euclidean square of edge length $\tau$. In particular, the singular axes
of $C$
give rise to singular points $a_2, a_q$ and $a_r$ in $D$ in the following way.
The midpoint of each edge of $D$ is a singular point $a_2$ of order 2, the centre of $D$ is a singular point $a_q$ of order $q$,
and each vertex of $D$ is a singular point $a_r$ of order $r$.
Since the horoballs covering $C$ do not intersect in their interiors, the vertical axis $l_s$ in the upper half space
$U^3$ passing through
a singular point $a_s\,,\,s\in\{2,q,r\}\,,$
can lie in the interior of a full-sized horoball $B_u$ only if the horoball $B_u$ is centred at the intersection $u$ of $l_s$ with
$\{t=0\}$, that is, if $\,a_s=(u,1)$. We say that $B_u$ is {\it centred at $a_s$} and write $B$ omitting the index $u$, for short.
By barycentric decomposition, $D$ is tiled into copies of its {\it characteristic triangle} $\Delta$ with vertices $a_s\,,\,s\in\{2,q,r\}\,,$
and which is either
a right-angled triangle $[3,6]$ or a right-angled triangle $[4,4]$; see Figure \ref{fig:Figure2}.
\begin{figure}
  \centering
  \begin{minipage}{0.49\textwidth}
    \centering
    \def\svgwidth{0.8\textwidth}
    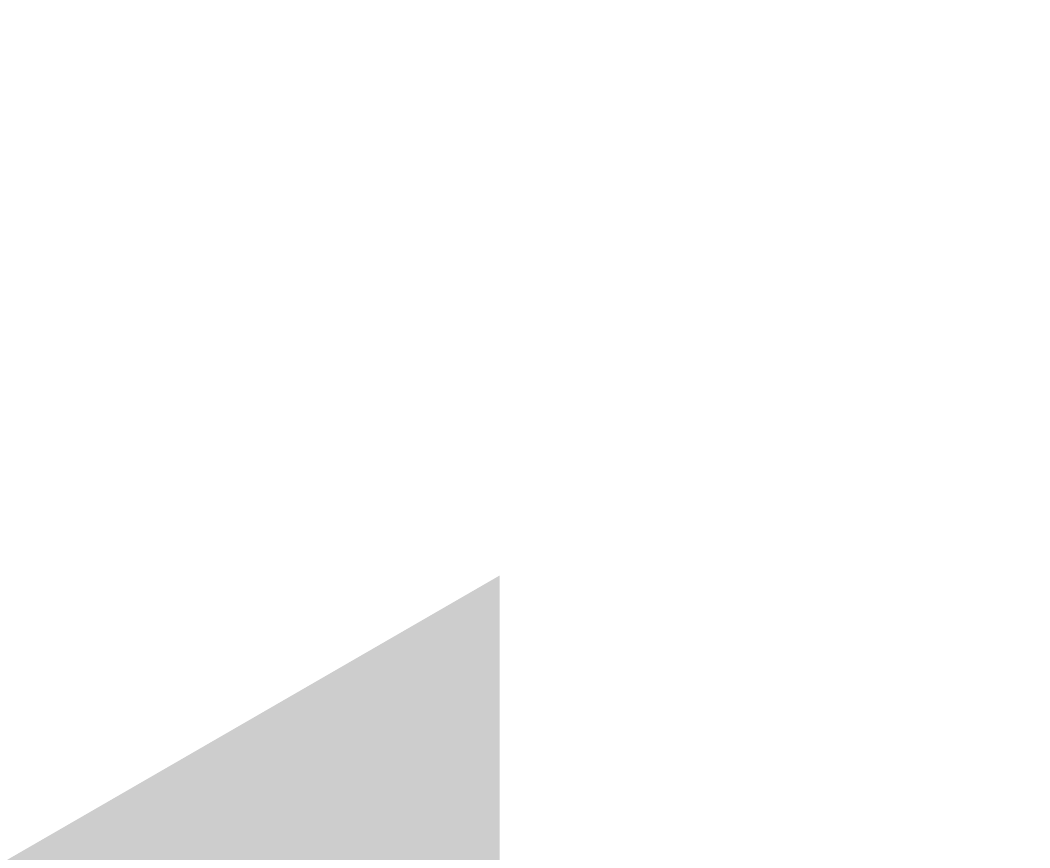
  \end{minipage}\hfill
  \begin{minipage}{0.49\textwidth}
    \centering
    \def\svgwidth{0.8\textwidth}
    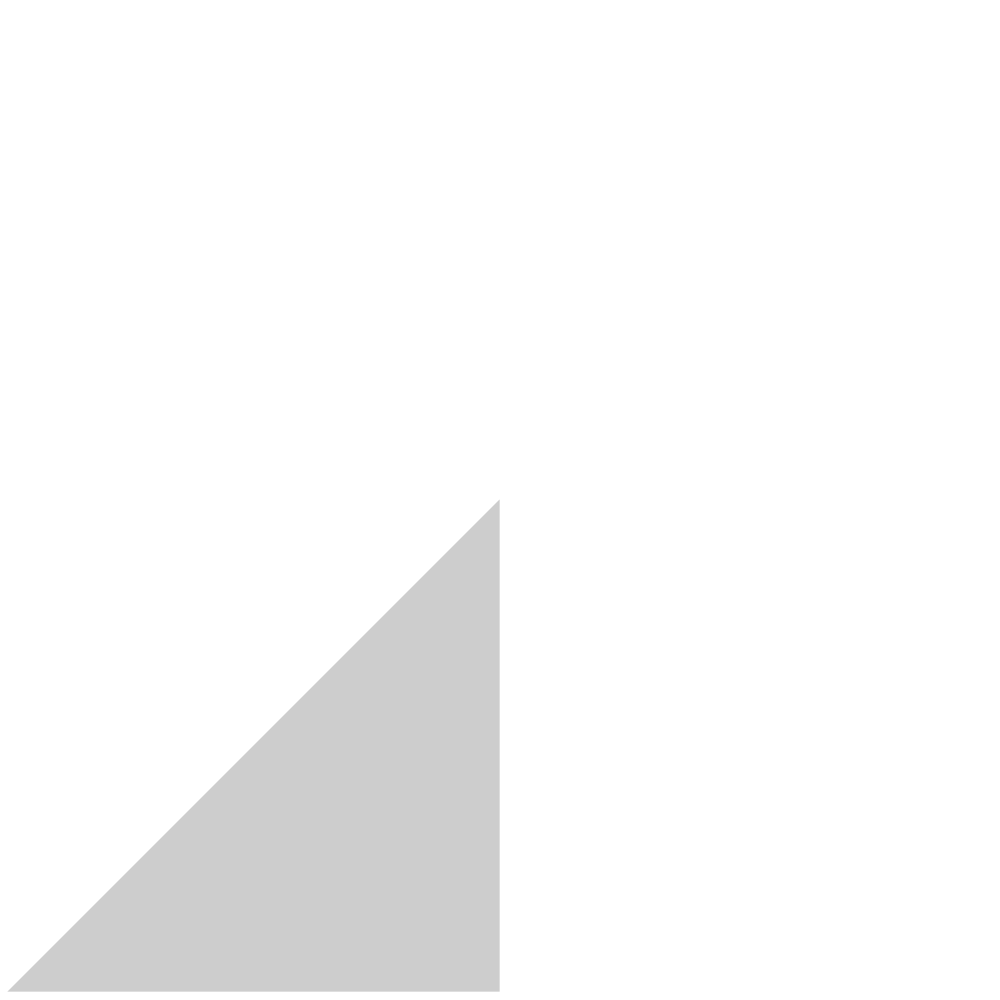
  \end{minipage}
  \caption{Tiling of
  $D$ into copies of its characteristic triangle $\Delta$}
  \label{fig:Figure2}
\end{figure}
The aim is to determine or estimate in terms of $\tau$ or -- more easily -- by means of the smallest distance $d$ of full-sized horoballs
the cathetus length $d_0(a_2,a_p)$ or $d_0(a_2,a_q)$
of $\Delta$. The position of the full-sized horoballs, projecting to full-sized horodisks (of radius $\frac{1}{2}$) in $H_{\infty}$,
will play a crucial role.
These investigations will allow us to bound the volume of $V$ by the expression (see \eqref{eq:cusp-volume} and \eqref{eq:density-bound})
\begin{equation}\label{eq:vol-Delta}
 \vol (V)\ge\frac{\vol (C)}{d_3(\infty)}=\frac{\mu_3}{\sqrt{3}}\cdot\vol_0(\Delta)\,\,.
\end{equation}

\subsection{One equivalence class of full-sized horoballs}\label{step3}
Assume that the crystallographic
group $\Gamma_{\infty}$ gives rise to only one equivalence class of full-sized horoballs. If there is a full-sized horoball
which is {\it not} centred at any of the singular points $a_s\,,\,s\in\{2,q,r\}\,,$ in the cusp diagram $D$, then the least
cusp volume $\vol(C)$ is given by (see also \cite[Figure 1(a), p. 4, and Figure 6(a), p. 12]{Adams1} and Figure \ref{fig:Figure3})
\begin{equation}\label{eq:no-singular}
 \vol(C)= \begin{cases}
\frac{\sqrt{3}}{12}\,\big(1+\frac{\sqrt{3}}{2}\big)\approx 0.269338 & \text{\quad if $C$ is of type $\{2,3,6\}\,\,,$} \\
\frac{1}{4}& \text{\quad if $C$ is of type $\{2,4,4\}\,\,.$}
\end{cases}
\end{equation}

\begin{figure}
  \centering
  \begin{minipage}{0.49\textwidth}
    \centering
    \def\svgwidth{0.8\textwidth}
    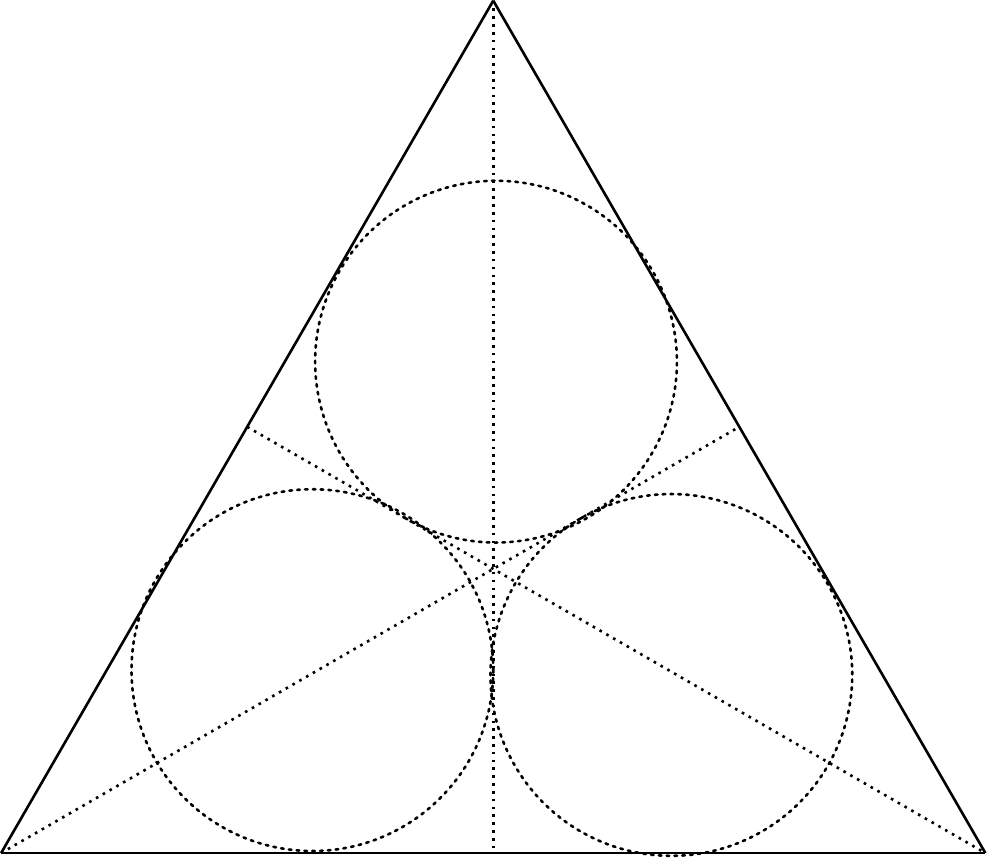
  \end{minipage}\hfill
  \begin{minipage}{0.49\textwidth}
    \centering
    \def\svgwidth{0.8\textwidth}
    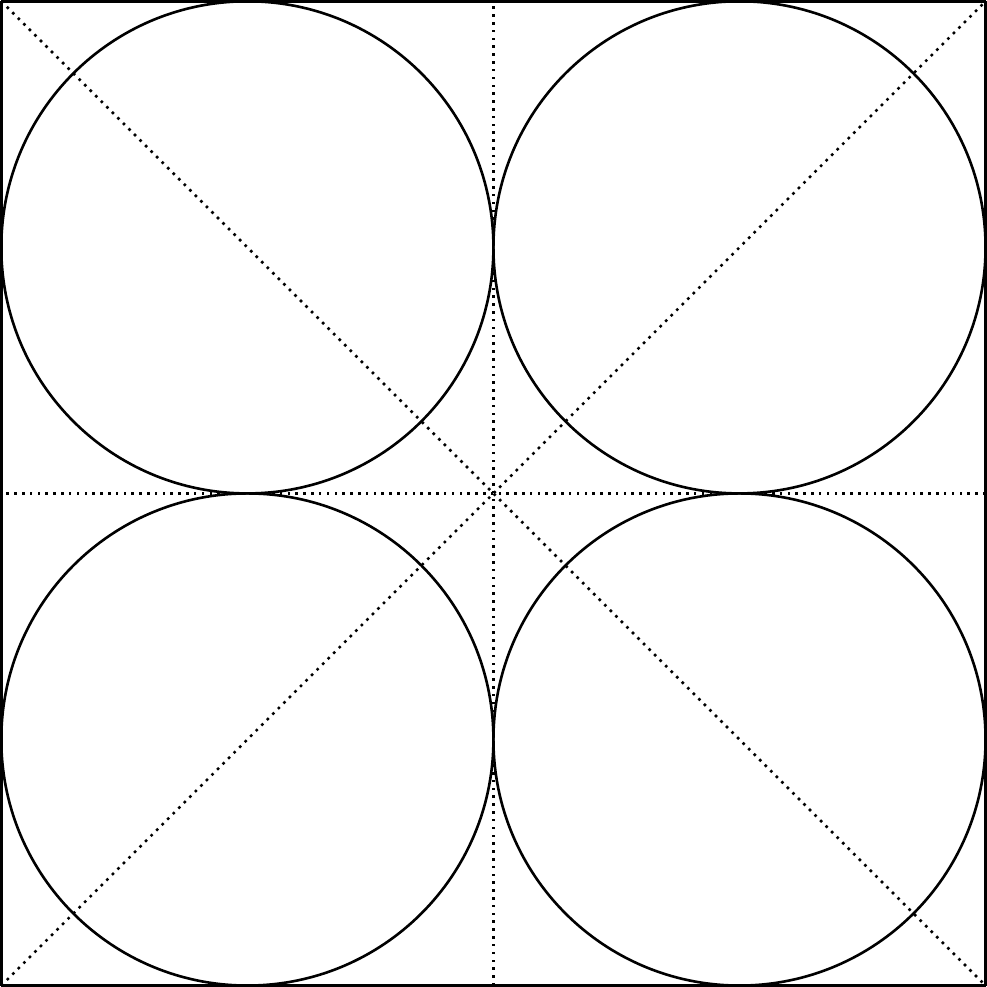
  \end{minipage}
  \caption{Full-sized horoballs not centred at singular points}
  \label{fig:Figure3}
\end{figure}

As a consequence, the volume of $V$ is strictly bigger than $v_*=\vol(\HH^3/[5,3,6])$ (compare with \eqref{eq:minvol-bound}).

From now on, we assume that
there is a full-sized horoball centred at one of the singular points $a_s\,,\,s\in\{2,q,r\},$ in the cusp diagram $D$.
We treat the cases $s\in\{2,3,6\}$ and $s\in\{2,4,4\}$ separately.
\subsubsection{The case \texorpdfstring{$\{2,3,6\}$}{(2,3,6)}}\label{step3a}
(i)\quad First we assume that there are at least two full-sized horoballs which touch one another, that is, the minimal distance
$d$ of the centres of full-sized horoballs equals 1. In particular, in the cusp diagram $D$, there will be at least two full-sized disks,
centred at equivalent singular points, which touch one another; see Figure \ref{fig:Figure4}.

\begin{figure}
  \centering
  \begin{minipage}{0.4\textwidth}
    \def\svgwidth{\textwidth}
    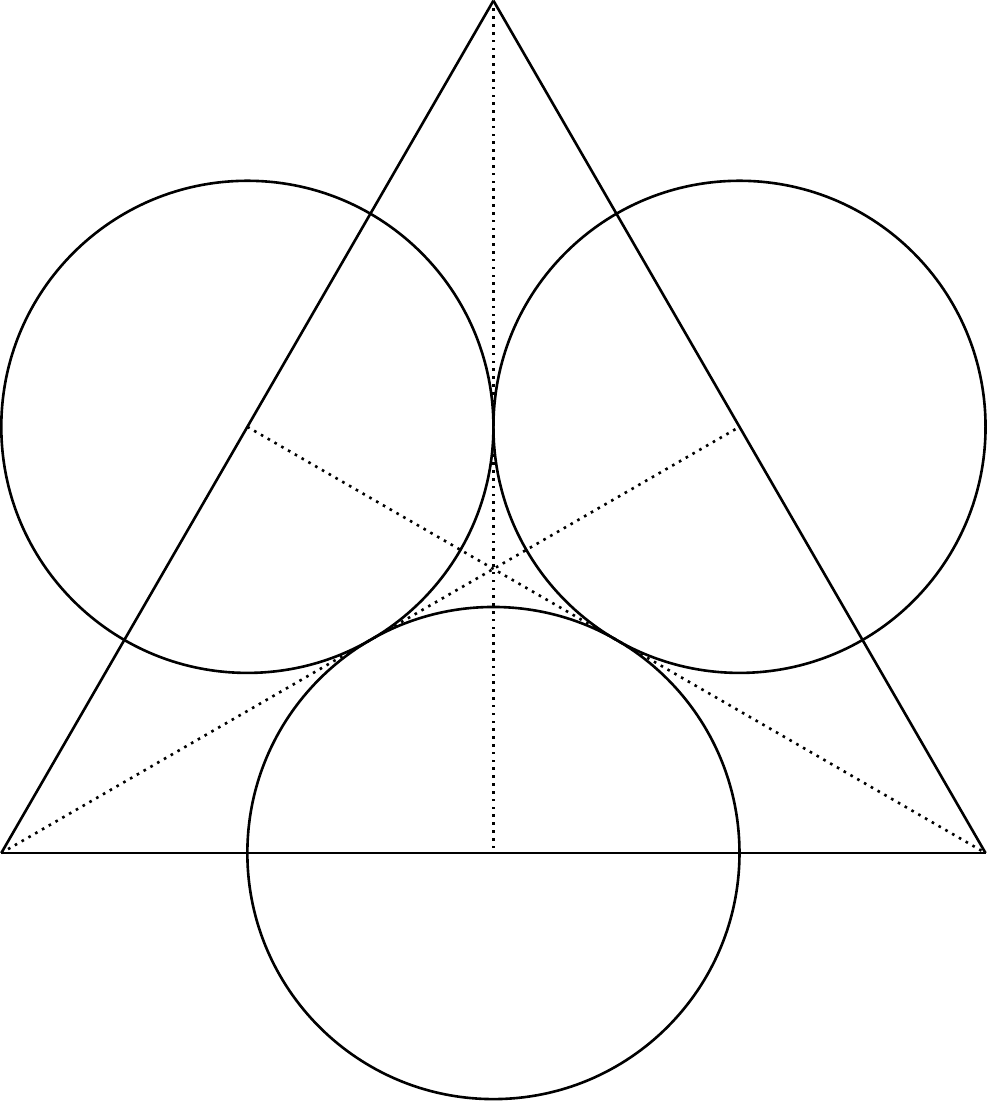
  \end{minipage}

  \vspace{1em}
  \begin{minipage}{0.45\textwidth}
    \def\svgwidth{\textwidth}
    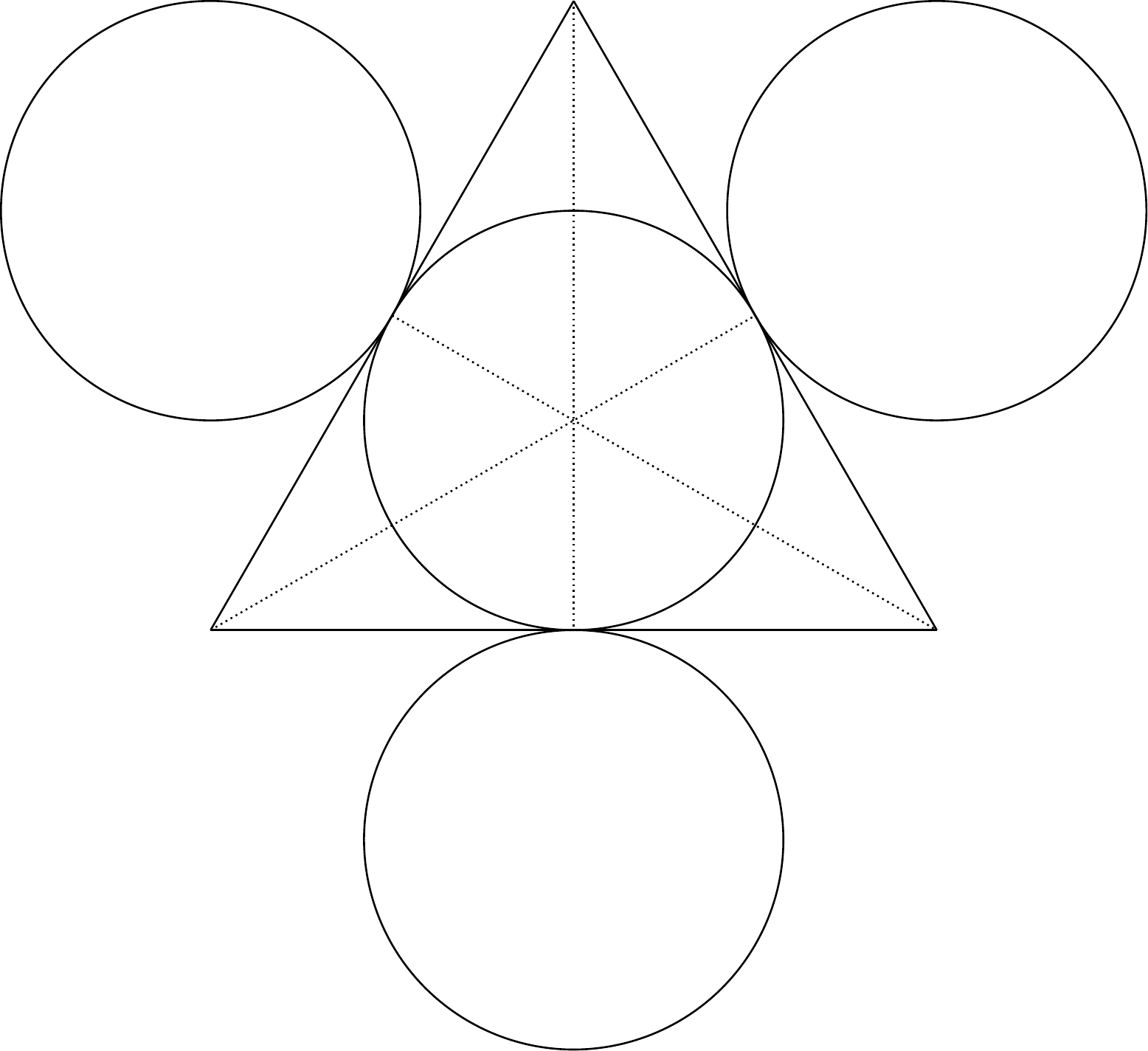
  \end{minipage}\hfill
  \begin{minipage}{0.45\textwidth}
    \def\svgwidth{\textwidth}
    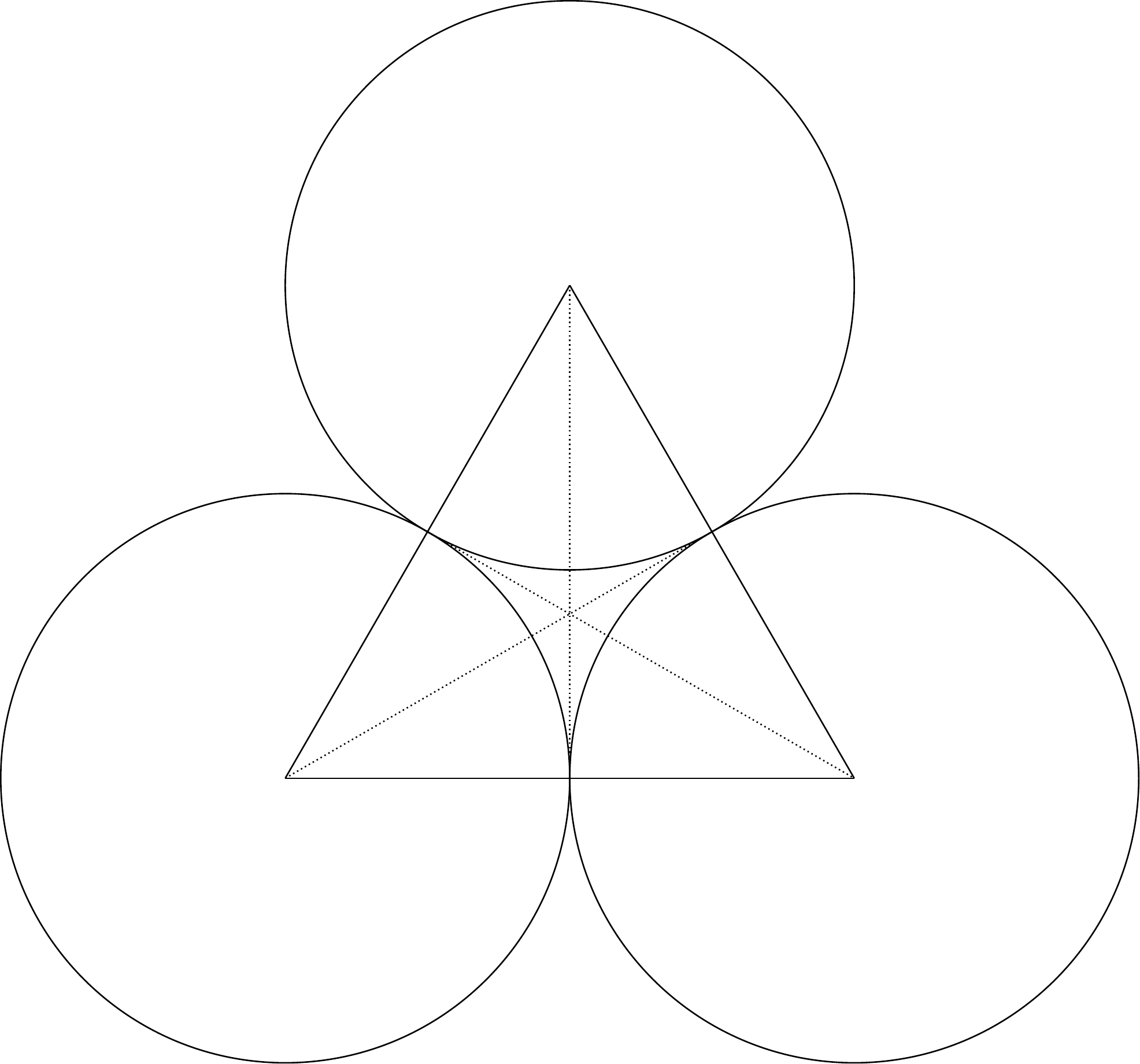
  \end{minipage}
  \caption{A $\{2,3,6\}$-cusp with full-sized horoballs centred at equivalent singular points}
  \label{fig:Figure4}
\end{figure}

{\it Order 6.\quad}If there is a full-sized
horoball (of radius $\frac{1}{2}$) centred at the singular point $a_6$ in $\Delta$, then $\,d_0(a_2,a_6)=\frac{1}{2}$,
and the cusp volume $\vol(C)$
equals $\frac{\sqrt{3}}{48}$. Since there are exactly three full-sized horoballs touching one another and the horoball $B_{\infty}$,
it is easy to see that the (up to isometry) unique orbifold $V_0$ corresponding to this configuration
is given by the quotient of $\HH^3$ modulo the {\it arithmetic} Coxeter group $[3,3,6]$ (see \cite{Mey1}, \cite{Mey2}).
The orbifold is covered by
Gieseking's manifold and is known to be the cusped hyperbolic 3-orbifold of minimal volume (which equals
$\,\frac{\mu_3}{24}=\frac{1}{8}\,\loba(\frac{\pi}{3})\approx0.042289$).
In particular, the cusp density $\delta(C)$ satisfies $\vol(C)/\vol(V_0)=d_3(\infty)$ (see also \eqref{eq:cuspdensity-bound}).

{\it Order 3.\quad}Suppose that there is a full-sized horoball
centred at the singular point $a_3$ in $\Delta$. Then,
$\,d_0(a_2,a_3)=\frac{1}{2}$ and $\,d_0(a_2,a_6)=\frac{\sqrt{3}}{2}$ so that
the cusp volume
equals $\frac{\sqrt{3}}{16}$ \cite[Figure 1(c), p. 4]{Adams1}.
By Lemma \ref{elliptic},
it is not difficult to see that there is a unique orbifold
which corresponds to this configuration. Its fundamental group is given by the $\mathbb Z_2$-extension of
the {\it arithmetic} Coxeter group $[3,6,3]$ (and has covolume
$\frac{1}{4}\,\loba(\frac{\pi}{3})\approx0.084579$).

{\it Order 2.\quad}Suppose that there is a full-sized horoball
centred at the vertex $a_2\in \Delta$.
Then,
$\,d_0(a_2,a_3)=\frac{1}{\sqrt{3}}$ and $\,d_0(a_2,a_6)=1$, and the
cusp volume is given by
$\frac{\sqrt{3}}{12}\approx 0.144338$ (see \cite[Figure 1(b), p. 4]{Adams1}). Note that this volume is also the smallest cusp volume for the case that we have full-sized horoballs at both the  singular points $a_2$ and $a_6$.
There is a unique orbifold
which corresponds to this configuration, and it is given by
the {\it arithmetic} Coxeter simplex with symbol $[3^{1,1},6]$ and 2 cusps. By cutting along the symmetry plane of $[3^{1,1},6]$, we get an orbifold whose fundamental group is commensurable to $[3^{1,1},6]$ and isomorphic to the arithmetic Coxeter group $[4,3,6]$ of covolume
$\frac{5}{16}\,\loba(\frac{\pi}{3})\approx0.105723$ (see \cite[p. 347]{JKRT1}, for example).

\smallskip
(ii)\quad Suppose now that the full-sized horoballs do {\it not} touch another. Then,
the minimal distance $d$ of their centres satisfies $d>1$.
Assume that one full-sized horoball is centred at the singular point $a_s$, $s\in\{2,3,6\}$, in $\Delta$.
Since $\Gamma_{\infty}$ identifies all full-sized horoballs, Lemma \ref{elliptic} yields an order two
elliptic element $\gamma_s\in\Gamma$ with axis perpendicular at $a_s=:(u_s,1)$ to the order $s$ axis $l_s$.
The element $\gamma_s$ sends the horoball $B_{\infty}$ to the full-sized horoball $B_{u_s}$ and vice versa.
Furthermore, $\gamma_s$ sends the $s$ neighboring full-sized horoballs -- each at distance $d$ from $u_{s}$ --
to $s$ smaller horoballs. By the proof of Lemma \ref{horo5}, these smaller horoballs are of Euclidean diameter $\frac{1}{d^2}$,
and by Corollary \ref{horo3}, their base points are at a distance $\frac{1}{d}$ from $u_s$. Following the terminology
of Adams \cite [p. 5]{Adams1},
call each of these balls a $(\frac{1}{d})$-{\it ball}. By construction, the $(\frac{1}{d})$-{balls} are the biggest
horoballs of diameter less than 1 that are tangent to full-sized horoballs.

Let $B(h)$ be a horoball of diameter $h$ such that $\,\frac{1}{d^2}\le h<1\,$ covering $C$ which is {\it not} tangent to any larger horoball
covering $C$. Then, in the interior of its upper hemisphere, there are no points of tangency with horoballs covering $C$. By Corollary
\ref{horo2}, this upper hemisphere is an open disk of radius 1 with respect to the induced metric on its boundary.
By mapping $B(h)$ to $B_{\infty}$ by means of an element in $\Gamma$, we obtain a disk of radius 1 on the horosphere $H_{\infty}$
which contains no point of tangency with full-sized horoballs.
Such a disk is called a {\it disk of no tangency}. Notice that the existence of a disk of no tangency in the cusp diagram
on $H_{\infty}$ for $C$
has the same effect as an extra full-sized horoball touching the centre of the disk would have.

Furthermore, by Corollary \ref{horo3}, one can deduce the following fact for a cusp of any type
(for a proof, see \cite[pp. 5--6]{Adams1}).

\begin{lem}\label{no-tangency}
 Let $B$ be a full-sized horoball covering the cusp $C$, and let $B_{\frac{1}{d}}$ be a $(\frac{1}{d})$-{ball} touching $B$. If the base point
 $x$ of $B_{\frac{1}{d}}$ has Euclidean distance $d_0(x,y)<1$ from the base point $y$ of another full-sized horoball $B'$
 or if $d_0(x,z)<\frac{1}{d}$ with $z$ a base point
 of another $(\frac{1}{d})$-{ball} $B'_{\frac{1}{d}}$, without $B_{\frac{1}{d}}$ touching $B'$ or $B'_{\frac{1}{d}}$, then there is
 a disk of no tangency in the cusp diagram of $C$.
\end{lem}

With these preliminary remarks, we investigate in detail each of the cases $s\in\{2,3,6\}$.

{\it Order 6.\quad} Suppose that a full-sized disk $B=B_1$ is centred at the singular point $b_1=a_6$ in $D$, and that
$B_2, B_3$ are full-sized disks centred at the vertices $b_2,b_3$ at distance $d$ from $B_1$ in the cusp triangle $D$.
For $1\le i\le 3$, let $x_i$
denote the center (at height 1) of the $(\frac{1}{d})$-{ball} touching $B_i$ in $D$.
Following Adams \cite[Figure 2]{Adams1}, we define the Euclidean distances
$u=d_0(x_i,x_j)\,,\,v\,$ and $w$
as in Figure \ref{fig:Figure5}.

\begin{figure}
  \centering
  \def\svgwidth{0.7\textwidth}
  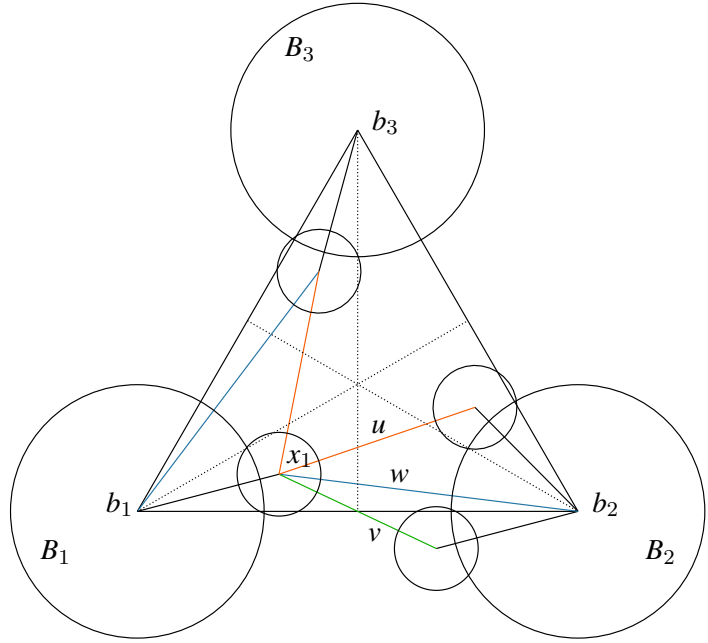
  \caption{$\left(\frac{1}{d}\right)$-balls in the $\{2,3,6\}$-cusp triangle $D$}
  \label{fig:Figure5}
\end{figure}

Consider the triangle $\Delta(b_1,x_1,b_2)$ with edge lengths $d_0(b_1,x_1)=\frac{1}{d}$, $d_0(x_1,b_2)=w$
and $d_0(b_1,b_2)=d$,
and denote by $\theta=\measuredangle(x_1,b_2)$ the angle at $b=b_1$.
By symmetry of $D$, we may suppose that $0\le \theta\le\frac{\pi}{6}$. The elementary law of cosines allows one
to deduce the following identities (see \cite[(1)--(3)]{Adams1})
 \begin{align}
        u^2&=d^2+\frac{3}{d^2}-\sqrt{3}\,\sin\theta-3\cos\theta=
        d^2+\frac{3}{d^2}-2\,\sqrt{3}\,\cos(\frac{\pi}{6}-\theta)\label{eq:1}  \\
        v^2&=d^2+\frac{4}{d^2}-4\,\cos\theta \label{eq:2} \\
        w^2&=d^2+\frac{1}{d^2}-2\,\cos\theta \label{eq:3}\,\,.
\end{align}

The strategy as developed in detail by Adams \cite[pp. 6--10]{Adams1} consists now in analysing the different
tangency possibilities of the $(\frac{1}{d})$-balls with respect
to one or several full-sized horoballs and among themselves.

A crucial point will be the dichotomy of $\Gamma_{\infty}$ preserving orientation, or not. If $\Gamma_{\infty}$ contains a reflection, then we say that the
horoball configuration in the cusp diagram $D$ has a {\it mirror symmetry}.

Altogether, we will be able to determine or estimate explicitly
the lengths $d$ and $\tau$, the associated cusp volume $\vol(C)=\frac{\sqrt{3}\,d^2}{48}$ (in case of no mirror symmetry)
and to describe the corresponding orbifold $V$ under the volume constraint $\vol(V)\le
\vol(V_*)=v_*$
(see \eqref{eq:minvol-bound}). Furthermore, one has to check whether the resulting orbifold $V$ is arithmetic or not. This is not difficult if its fundamental group is given by a hyperbolic Coxeter group (see Section \ref{Coxeter}) or by an explicit presentation in $\PSL_2(\mathbb C)$ (see \cite[Section 8]{MR1}).
\begin{remark}\label{bound-d}
Assuming the rough volume estimate \hbox{\rm (}see \eqref{eq:vol536}, \eqref{eq:density-bound}\hbox{\rm )}
\[
\hbox{\rm vol}(V)\ge \frac{\hbox{\rm vol}(C)}{d_3(\infty)}=\frac{\sqrt{3}\,d^2}{48\,d_3(\infty)}>\hbox{\rm vol}(V_*)=v_*\approx0.171502\,\,,
\]
where $d_3(\infty)\approx 0.853276$, we obtain $d>2.013813$.
Hence, we only have to consider values for $d>1$ with $d\le 2.013813$ or -- in the case that $\Gamma_{\infty}$ is orientation-preserving -- values for $d$ with $d\le 1.423982$.
\end{remark}
We start with the easy case of $(\frac{1}{d})$-balls touching at least two full-sized horoballs. This case has a complete answer showing that
the resulting orbifold is arithmetic (see \cite[p.~3 and p.~10]{Adams1} and \cite{NR}).
We provide one illustrating example and suppose that a single $(\frac{1}{d})$-ball
touches three full-sized horoballs giving a cusp diagram as depicted in
\cite[Figure 3(b)]{Adams1}. It easily follows that $\,d=\sqrt[\leftroot{-2}\uproot{2}4]{3}\,$
yielding a cusp volume of $\frac{1}{16}$. Associated
to this configuration is a unique
orbifold with fundamental group that is commensurable to the arithmetic Coxeter group
$[3,6,3]$ (see \cite[p.~10]{Adams1} in the corresponding oriented case).

The delicate case is
when each $(\frac{1}{d})$-ball touches a unique full-sized horoball. Let us consider the mutual position of the $(\frac{1}{d})$-balls.
By Lemma \ref{3balls}, we can exclude the case that three $(\frac{1}{d})$-balls touch each other.

Next, observe that we can assume that the distance $w$ as depicted in Figure \ref{fig:Figure5} satisfies
\begin{equation}\label{eq:wd}
w\ge1\quad\hbox{implying that}\quad d\ge \frac{1}{2}\,\big(\,\sigma+\sqrt{\sigma^2-4}\,\big)\approx 1.515464\quad\hbox{\rm for}\quad\sigma=\sqrt{3+\sqrt{3}}\,\,.
\end{equation}
In fact, if there is {\it no} disk of no tangency, this is a direct consequence of Lemma \ref{no-tangency}. If there {\it is} a disk of no tangency in $D$, then $d\ge\sqrt{3}$. Since $0\le \theta\le\frac{\pi}{6}$, equation \eqref{eq:3} for $w$ implies that $w\ge\frac{2}{\sqrt{3}}>1$ and the asserted estimate for $d$.

{\it Consequence.} If $\Gamma_{\infty}$ is orientation-preserving, the bound $\,d\ge1.515464$ from \eqref{eq:wd}
yields a too big cusp estimate $\frac{\vol(C)}{d_3(\infty)}$ in comparison with $v_*$. In particular, we do not have to analyse in depth cusp diagrams
$D$ with no mirror symmetry.

Now, if two $(\frac{1}{d})$-balls touch one another but no further $(\frac{1}{d})$-ball touches this pair, then the point of
tangency created by this pair must coincide with a 2-fold singular point $a_2$ in $D$. Since the tangency point is the midpoint of an edge $e$ of $D$, one has $v=\frac{1}{d^2}$ with $v$ given by equation \eqref{eq:2}. Implementing this identity into the equation \eqref{eq:3} for $w$, the property $w\ge1$ yields the inequality
\begin{equation}\label{eq:vd}
d^6-2d^4-2d^2+1\ge0\,\,,
\end{equation}
and hence $\,d\ge2\cos\frac{\pi}{5}>1.6$.
One can conclude the case by different considerations (for Adams' conclusion, see \cite[pp. 7-8]{Adams1}).
First, assume that the centres of the two $(\frac{1}{d})$-balls are {\it not} aligned with $b_1$ and $b_2$ on the edge $e$ of $D$. Then, the group $\Gamma_{\infty}$ consists of rotations, only, and we are done.
\newline
Next, assume that the centres of the two $(\frac{1}{d})$-balls  lie on the edge $e$. Then,
\[
 d=\frac{2}{d} +\frac{1}{d^2}\,\,,
\]
and hence, $d=2\cos\frac{\pi}{5}$.
There is a unique orbifold associated to this situation, which is
based on an orthoscheme $R(\frac{\pi}{5},\frac{\pi}{3})$ and its reflection group (see Section \ref{volume}). Indeed, the orbifold equals $V_*$ with non-arithmetic fundamental group $[5,3,6]$ (see \cite[pp. 6-8]{Adams1}).


Assume now that the $(\frac{1}{d})$-balls {\it do not touch} each other. 
Following \cite[p. 8]{Adams1}, and by Lemma \ref{elliptic}, there is an order 2 elliptic element exchanging the horoball $B_\infty$ and a full-sized horoball $B_j$ while sending a neighbouring full-sized horoball $B_k$ to a $(\frac{1}{d})$-ball $B_x$, say (compare Figure \ref{fig:einsdurchw236}).
Then, the $(\frac{1}{d})$-ball $B_y$ (at distance $w$ from $B_j$) touching $B_k$ gets sent to a ball $B_s$ (at distance $\frac{1}{w}$ from $B_j$) touching $B_x$.
Like in \cite{Adams1} we call this ball a {\it $(\frac{1}{w})$-ball}. Its diameter equals $\frac{1}{w^2d^2}$ by Lemma \ref{horo5}. We study now the positions of the $(\frac{1}{w})$-balls and distinguish several cases.

$\bullet\quad$ Suppose that the $(\frac{1}{w})$-balls of the three $(\frac{1}{d})$-balls in $D$ coincide.
This implies that the $(\frac{1}{w})$-ball is centred in the 3-fold singular point $p=a_3$ in $D$ yielding $w=\frac{\sqrt{3}}{d}$.
Knowing all the side lengths of the triangle $\Delta(j,k,x)$ one can calculate the angle $\theta$ with the law of cosines as $\cos\theta=\frac{5}{2\sqrt{7}}$ which yields $d=\sqrt[\leftroot{-2}\uproot{2}4]{7}>1.6$ by means of \eqref{eq:3}. Furthermore, one can derive that the points $j,x$ and $y$ are aligned.
The situation is sketched in Figure \ref{fig:einsdurchw236}.
\begin{figure}
  \centering
  \def\svgwidth{0.5\textwidth}
  \footnotesize
  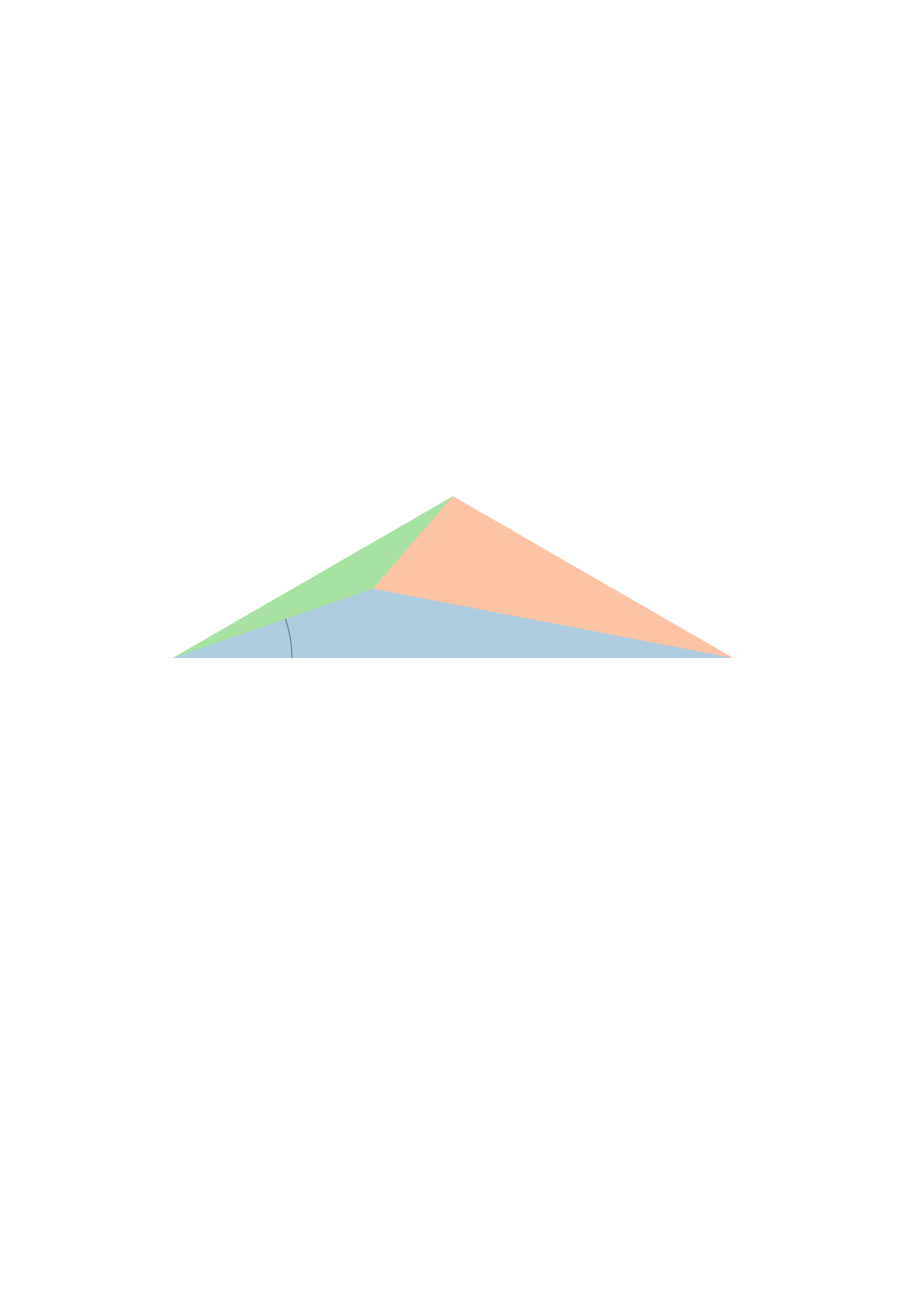
  \caption{A single $(\frac{1}{w})$-ball in a $\{2,3,6\}$-cusp triangle $D$}
  \label{fig:einsdurchw236}
\end{figure}

As a consequence, the horoball configuration in the cusp diagram $D$ has no mirror symmetry. Hence, the cusp volume  equals
\[
\vol(C)=\frac{\sqrt{21}}{24}>0.19>v_*\,\,,
\]
and we can exclude this case from our considerations.
\begin{remark}\label{21-24}
Adams considered the case above for an oriented orbifold $O=\mathbb H^3/\Gamma$ and calculated the corresponding cusp volume correctly as $\frac{\sqrt{21}}{24}$. However, in \cite[p. 10]{Adams1}, he gives a wrong fundamental polyhedron for $\Gamma$ and a volume which is too big. In the Appendix \ref{subsec:appendix1},
we correct Adams' analysis, compute the volume and prove the arithmeticity of the orbifold $O$.
\end{remark}

$\bullet\quad$ Suppose that two $(\frac{1}{w})$-balls of two  $(\frac{1}{d})$-balls coincide, having center $x$, say, and that no further $(\frac{1}{w})$-ball touches them. If $D$ has no mirror symmetry, then by means of the estimate $d\ge1.515464$ given by \eqref{eq:wd}, we obtain
\[
\frac{\vol(C)}{d_3(\infty)}=\frac{\sqrt{3}\,d^2}{24\,d_3(\infty)}>0.19>v_*\,\,.
\]
Hence, we may assume that
the center $x$ coincides with an edge center $a_2$ of $D$. It follows that $d=\frac{2}{w}$ which, by means of \eqref{eq:3} and with $\theta=0$,
yields $d^4-2d^2-3=0$ with solution $d=\sqrt{3}$. There is a unique orbifold realising this configuration. It is arithmetic and of volume $\,\frac{3}{8}\loba(\frac{\pi}{3})$, being commensurable to the 2-cusped quotient space by the Coxeter group $[6,3,6]$.

$\bullet\quad$ Suppose that the $(\frac{1}{w})$-balls are distinct and that $D$ has a mirror symmetry. In general, the $(\frac{1}{w})$-balls cannot touch full-sized balls since then two $(\frac{1}{d})$-balls touch another, in contrast to our assumption above. Furthermore, three $(\frac{1}{w})$-balls cannot touch one another pairwise by Lemma \ref{3balls}.

Since $D$ has a mirror symmetry, the group $\Gamma_{\infty}$ is a reflection group. As a consequence, the centres of the $(\frac{1}{d})$-balls lie either on the edges or on the angle bisectors of $D$.

Assume first that the centres of the $(\frac{1}{d})$-balls lie on the edges of $D$. The situation is depicted in Figure \ref{fig:oneoverdinline}.
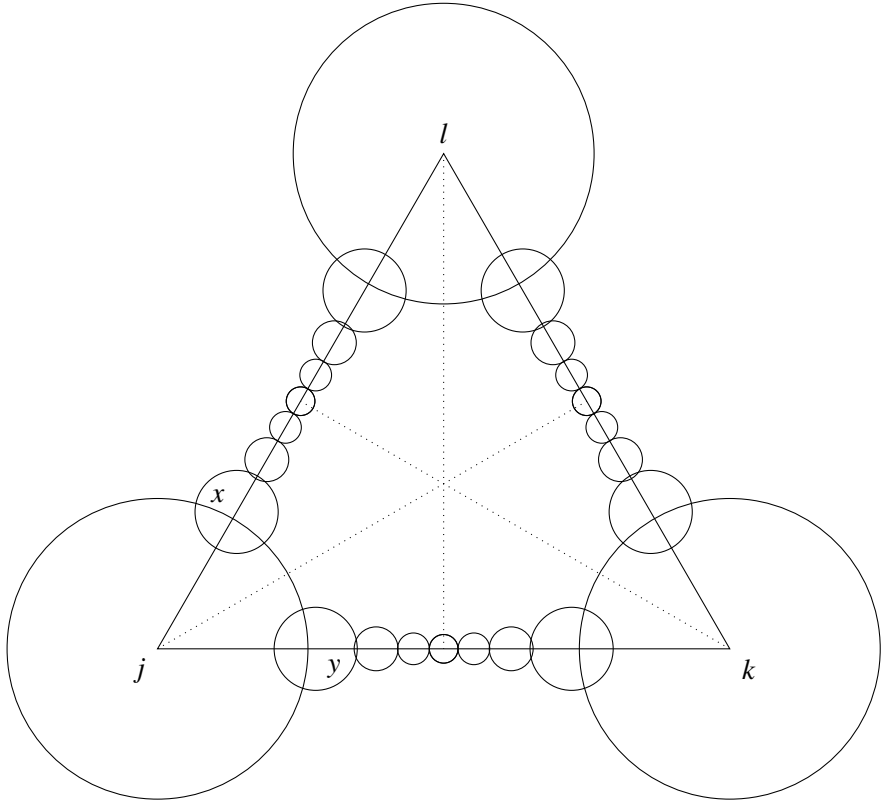
\begin{figure}[ht]
  \centering
  \begin{tikzpicture}[scale=4]
    \pgfmathsetmacro{\nn}{10}
    \pgfmathsetmacro{\d}{2*cos(180/\nn)}
    \pgfmathsetmacro{\height}{\d*sqrt(3)/2}
    \pgfmathsetmacro{\heightb}{\d/sqrt(3)}
    \pgfmathsetmacro{\w}{\d-1/\d}
    \draw (210:\heightb) -- (330:\heightb) -- (90:\heightb) -- cycle;
    \draw (210:\heightb) circle[radius=0.5] node[below left]{$j$};
    \draw (330:\heightb) circle[radius=0.5] node[below right]{$k$};
    \draw (90:\heightb) circle[radius=0.5] node[above]{$l$};
    \draw ($(210:\heightb)+(1/\d, 0)$) node[below right] {$y$};
    \draw ($(210:\heightb)+(60:1/\d)$) node[above left] {$x$};
    \pgfmathsetmacro{\dkglobal}{\d}
    \pgfmathsetmacro{\rkglobal}{1/\d}
    \foreach \k in {1,...,\numexpr (\nn-2)/2\relax}
    {
      \foreach \x in {210, 330, 90}
      {
        \foreach \y in {-150, 150}
        {
          \draw ($(\x:\heightb) + (\x+\y:1/\dkglobal)$)
            circle[radius=0.5*\rkglobal^2];
        }
      }
      \pgfmathparse{\d-1/\dkglobal}
      \global\let\dkglobal=\pgfmathresult
      \pgfmathparse{\rkglobal/\dkglobal}
      \global\let\rkglobal=\pgfmathresult
    }
    \draw[dotted] (210:\heightb) -- +(30:\height);
    \draw[dotted] (330:\heightb) -- +(30+120:\height);
    \draw[dotted] (90:\heightb) -- +(30-120:\height);
  \end{tikzpicture}
  \caption{$(\frac{1}{d})$-balls aligned along the edges of $D$}
  \label{fig:oneoverdinline}
\end{figure}
Consider the horoballs smaller than the full-sized horoballs in a hierarchical way as follows. For the $(\frac{1}{d})$-balls, we have the following three possibilities
\begin{enumerate}
  \item Each $(\frac{1}{d})$-ball touches two full-sized balls.
  \item Each $(\frac{1}{d})$-ball touches one full-sized ball and one other $(\frac{1}{d})$-ball.
  \item Each $(\frac{1}{d})$-ball touches one full-sized ball and \emph{no} other $(\frac{1}{d})$-ball.
\end{enumerate}
In the third case, there are $(\frac{1}{w})$-balls which are smaller images of the $(\frac{1}{d})$-balls under an order $2$ elliptic element $r$ which maps $B_\infty$ to a full-sized horoball according to Lemma \ref{elliptic}.
For those $(\frac{1}{w})$-balls, there are the analogous three cases as above, and so on.
This defines a sequence of horoball diagrams whose first elements are depicted in Figure \ref{fig:seq}.
Note that Adams \cite{Adams1} discussed the depicted horoball diagrams in detail (see above).
In particular, the third horoball diagram corresponds to the group $[5,3,6]$ while the fourth and the two first ones correspond to arithmetic groups.

For a given horoball diagram of this type, define a sequence
\[d_1 := d\,\,,\,\, d_2:= w=d - \frac{1}{d_1}, \ldots
\]
in the sense that in the third case where the $(\frac{1}{d_k})$-balls do not touch, there are $(\frac{1}{d_{k+1}})$-balls which are the smaller images of the $(\frac{1}{d_k})$-balls under the isometry $r$.
This implies the recursion
\begin{equation}
  d_{k+1} = d - \frac{1}{d_{k}}\,\,,\,\,k\ge1\,\,.
  \label{eq:recursion}
\end{equation}

\begin{figure}[ht]
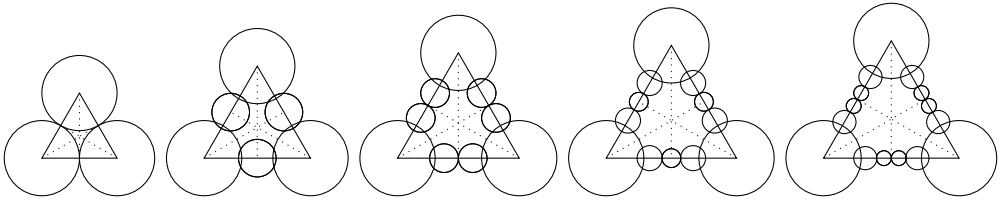

  \centering
  \oneoverdalignedthree\hfill
  \oneoverdaligned{4}\hfill
  \oneoverdaligned{5}\hfill
  \oneoverdaligned{6}\hfill
  \oneoverdaligned{7}
  \caption{The beginning of the sequence of horoball diagrams with aligned $(\frac{1}{d})$-balls.}
  \label{fig:seq}
\end{figure}

Let $k_\text{max}:=\max\left\{k\mid\text{the horoball diagram $D$ contains a $\frac{1}{d_k}$-ball}\right\}$. There are two cases.

If one $(\frac{1}{d_{k_{max}}})$-ball touches two $(\frac{1}{d_{{k_{max}}-1}})$-balls, the centre of the $(\frac{1}{d_{k_{max}}})$-ball lies in the midpoint $a_2$ between two neighbouring full-sized balls, leading to the equation
\begin{equation}
  \frac{1}{d_{k_{max}}} = \frac{d}{2}.
  \label{eq:caseA}
\end{equation}
If there is a pair of touching $(\frac{1}{d_{k_{max}}})$-balls, it implies that one of those $(\frac{1}{d_{k_{max}}})$-balls is fixed by $r$ because otherwise $r$ sends $(\frac{1}{d_k})$-balls to $(\frac{1}{d_{k-1}})$-balls or to $(\frac{1}{d_{k+1}})$-balls but there is no $(\frac{1}{d_{{k_{max}}+1}})$-ball.
This yields the equation
\begin{equation}
  1 = d - \frac{1}{d_{k_{max}}}.
  \label{eq:caseB}
\end{equation}

The recursion \eqref{eq:recursion} together with the end conditions \eqref{eq:caseA} or \eqref{eq:caseB} prove that $d$ has to be the root of a polynomial related to the Chebyshev polynomials of the first kind or of the second kind, respectively.
In case \eqref{eq:caseA}, this implies
\begin{equation}
  d = 2\,\cos\left(\frac{\pi}{2k_{max}+2} \right)\,\,,
  \label{eq:dcaseA}
\end{equation}
and in case \eqref{eq:caseB}, it implies
\begin{equation}
  d = 2\,\cos\left(\frac{\pi}{2k_{max}+3} \right)\,\,.
  \label{eq:dcaseB}
\end{equation}
Details can be found in \cite{Drewitz}.
This recursion allows us to investigate the horoball diagrams and to characterise the group $\Gamma$ assuming that $\Gamma_\infty$ is a reflection group.
\begin{proposition}\label{prop:reflect}
  In the case of the $\left\{ 2,3,6 \right\}$-cusp where $\Gamma_\infty$ is a reflection group and the full-sized horoballs are centred in the singular points $a_6$, the group $\Gamma$ having a horoball diagram according to Figure \ref{fig:oneoverdinline} contains the Coxeter group
\[
\gr{6}\gd\gr{\alpha_k}\bullet
\]
where $\,d =: 2\,\cos\alpha_k\,$ is given according to \eqref{eq:dcaseA} and \eqref{eq:dcaseB}, respectively.
  \begin{proof}
    Since the Coxeter group $\Gamma_\infty=\gr{6}\gd\bullet\,$ is contained in $\Gamma$, it only remains to prove the existence of the last generator in $\Gamma$.
    Consider an elliptic isometry $r$ of order 2 mapping $B_\infty$ to a full-sized horoball and vice versa.

    If $r$ is a reflection, we can see that it is the desired one:
    Its reflection plane $R$ is a hemisphere orthogonal to the two reflection planes of $\Gamma_\infty$ passing through $a_6$.
    The hemisphere has radius 1, and its centre has distance $\frac{d}{2}$ to the last reflection plane of $\Gamma_\infty$ implying that the angle of intersection $\alpha$ satisfies $\cos\alpha=\frac{d}{2}$, that is, $\alpha=\alpha_k$.

    In the case where $r$ is a rotation of order 2, we can choose it such that the full-sized ball $B_j$ is mapped to the $(\frac{1}{d})$-ball $B_x$, and $B_l$ is mapped to $B_y$; see Figure \ref{fig:oneoverdinline}.
    Then, $r$ is the rotation in the edge $ab$ of the orthoscheme defined by the planes mentioned above; see Figure \ref{fig:orthofun}.
    If $s\in\Gamma_\infty$ is the reflection in the face $(a,b, \infty)$, then $rs=sr$ is the reflection in the hemisphere $R$ as above and belongs to $\Gamma$.
    \begin{figure}[ht]
      \centering
      \begin{tikzpicture}
        \pgfmathsetmacro{\width}{4}
        \pgfmathsetmacro{\height}{0.3*\width}
        \pgfmathsetmacro{\angle}{45}
        \coordinate (A) at (0, 0);
        \coordinate (B) at ($(\width/2, 0)+cos(\angle)*(\width/2, 0) + sin(\angle)*(0, 0, \width/2)$);
        \coordinate (C) at (\width, 0, 0);
        \coordinate (D) at (\width, \height, 0);
        \draw (A) node[left] {$\infty=q$} -- (B) -- (C) -- (D) -- cycle;
        \draw[thick, red] (B) node[below]{$a$} -- (D) node[above]{$b$};
        \draw[dashed] (A) -- (C);
        \coordinate (w3) at ($0.5*(A)+0.5*(B)$);
        \draw ($0.9*(w3)+0.1*(C)$) -- (w3) -- ($0.9*(w3)+0.1*(D)$);
        \draw (w3) node[below left]{$\scriptstyle \frac{\pi}{3}$};
        \coordinate (w6) at ($0.5*(A)+0.5*(D)$);
        \draw ($0.9*(w6)+0.1*(C)$) -- (w6) -- ($0.9*(w6)+0.1*(B)$);
        \draw (w6) node[above left]{$\scriptstyle \frac{\pi}{6}$};
        \coordinate (wa) at ($0.5*(C)+0.5*(D)$);
        \draw ($0.94*(wa)+0.06*(A)$) -- (wa) -- ($0.9*(wa)+0.1*(B)$);
        \draw (wa) node[right]{$\alpha_k$};
      \end{tikzpicture}
      \caption{Fundamental orthoscheme}
      \label{fig:orthofun}
    \end{figure}
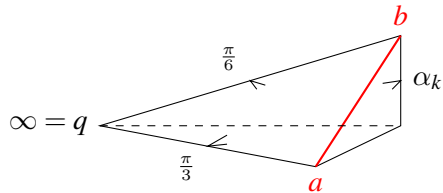
  \end{proof}
\end{proposition}

As a consequence of Proposition \ref{prop:reflect}, the
groups $\Gamma$ with distances $\sqrt{2}\le d\le\sqrt{3}$ between full-sized horoballs are of smallest covolume if $\Gamma$ coincides with (a finite index subgroup of) the arithmetic Coxeter group $[l,3,6]$ for $l=3,4$ or $6$, or with the non-arithmetic Coxeter group $[5,3,6]$.

If the distance $d$ between full-sized horoballs satisfies $d>\sqrt{3}$, then the hyperbolic Coxeter group
$\,\gr{6}\gd\gr{\alpha_k}\bullet\,$ with $2\cos\alpha_k=d\,$
of Proposition \ref{prop:reflect} has infinite covolume since the orthoscheme $R(\alpha_k,\frac{\pi}{3})$ has an ultraideal vertex given by the graph $\,\gd\gr{\alpha_k}\bullet\,$ (see Section \ref{volume}).
In order to construct an orbifold with finite volume from it, its fundamental group $\Gamma$ has to contain additional isometries.
\begin{proposition}\label{prop:polarplane}
  In the case of Proposition \ref{prop:reflect} with $\sqrt{3}<d\leq 2$, the minimal possible orbifold volume is obtained from the truncated Coxeter orthoscheme $R_t(\alpha_k,\frac{\pi}{3})$ with Coxeter graph
  \[
  \gr{6}\gd\gr{\alpha_k}\bullet\cdot\cdots\bullet\qquad\hbox{where}\quad d = 2\,\cos\alpha_k\,\,.
  \]
  \begin{proof}
    If the horoball diagram appears as in Figure \ref{fig:oneoverdinline} with $d>\sqrt{3}$, then the orthoscheme $\,\gr{6}\gd\gr{\alpha_k}\bullet\,$ has an ultraideal vertex $v_0$.
    Hence, there has to be an additional horoball $B_+$ close to the centre of $D$ in order to obtain finite volume.
    In view of the realisations of the horoball $B_+$, Proposition \ref{prop:reflect} guarantees the existence of reflections mapping $B_\infty$ to the full-sized balls $B_j$, $B_k$, and $B_l$.
    Denote the associated reflection planes by $P_j$, $P_k$, and $P_l$, respectively.
    The horoball $B_+$ cannot intersect any of those planes because otherwise it would intersect its image under the corresponding reflection.
    The biggest possible horoball $B_+$ (yielding the minimal volume orbifold) hence {\it touches} all three planes $P_j$, $P_k$, and $P_l$.
    Those three planes are orthogonal to the polar plane $P_0$ of $v_0$.
    Therefore, $P_j$, $P_k$, and $P_l$ are left invariant under the reflection $r_0$ in the plane $P_0$.
    The horoballs $B_\infty$ and $B_+$ are the only ones touching all three planes at once, implying that one is the image of the other under $r_0$.
    Thus, the polar plane $P_0$ is the bisector of $B_\infty$ and $B_+$.

    In our constructions, a fundamental polyhedron of $\Gamma$ can be found by considering the Dirichlet-Vorono\u{\i} cell of $B_\infty$ modulo $\Gamma_\infty$.
    The previous arguments show that a fundamental polyhedron of $\Gamma$ has to include at least the truncated orthoscheme $R_t(\alpha_k,\frac{\pi}{3})$ as asserted.

    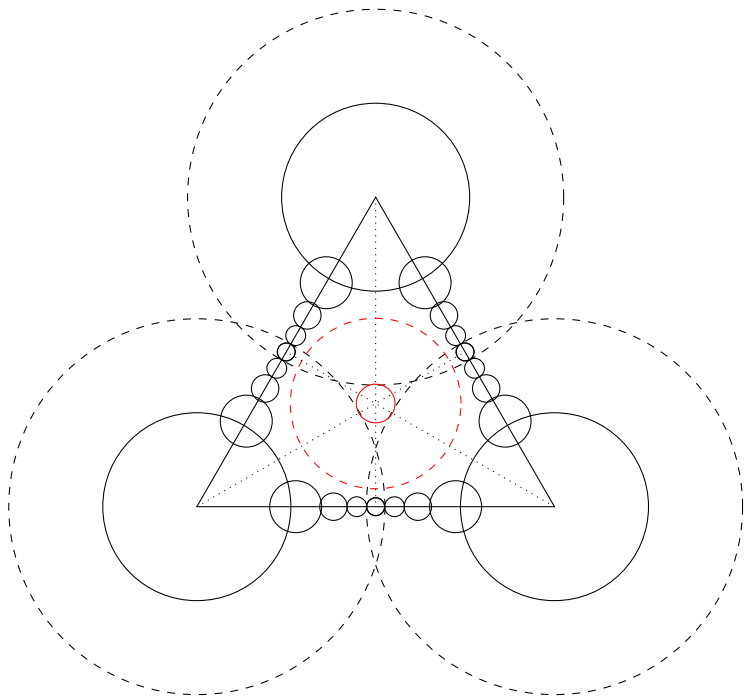
\begin{figure}[ht]
      \centering
      \begin{tikzpicture}[scale=2.5]
        \pgfmathsetmacro{\nn}{10}
        \pgfmathsetmacro{\d}{2*cos(180/\nn)}
        \pgfmathsetmacro{\height}{\d*sqrt(3)/2}
        \pgfmathsetmacro{\heightb}{\d/sqrt(3)}
        \pgfmathsetmacro{\w}{\d-1/\d}
        \draw (210:\heightb) -- (330:\heightb) -- (90:\heightb) -- cycle;
        \draw (210:\heightb) circle[radius=0.5];
        \draw (330:\heightb) circle[radius=0.5];
        \draw (90:\heightb) circle[radius=0.5];
        \draw[dashed] (210:\heightb) circle[radius=1];
        \draw[dashed] (330:\heightb) circle[radius=1];
        \draw[dashed] (90:\heightb) circle[radius=1];
        \pgfmathsetmacro{\dkglobal}{\d}
        \pgfmathsetmacro{\rkglobal}{1/\d}
        \foreach \k in {1,...,\numexpr (\nn-2)/2 \relax}
        {
          \foreach \x in {210, 330, 90}
          {
            \foreach \y in {-150, 150}
            {
              \draw ($(\x:\heightb) + (\x+\y:1/\dkglobal)$)
                circle[radius=0.5*\rkglobal^2];
            }
          }
          \pgfmathparse{\d-1/\dkglobal}
          \global\let\dkglobal=\pgfmathresult
          \pgfmathparse{\rkglobal/\dkglobal}
          \global\let\rkglobal=\pgfmathresult
        }
        \draw[red] (0, 0) circle[radius=(\heightb^2-1)/2];
        \draw[dashed, red] (0, 0) circle[radius=sqrt(\heightb^2-1)];
        \draw[dotted] (210:\heightb) -- +(30:\height);
        \draw[dotted] (330:\heightb) -- +(30+120:\height);
        \draw[dotted] (90:\heightb) -- +(30-120:\height);
      \end{tikzpicture}
      \caption{$(\frac{1}{d})$-balls in line}
      \label{fig:oneoverdinlinepolar}
    \end{figure}
  \end{proof}
\end{proposition}
By Proposition \ref{prop:polarplane}, the minimal possible orbifold volume in the case $\,\sqrt{3}<d=2\cos\alpha_k\le2\,$ is given by the volume of the truncated Coxeter orthoscheme $R_t(\alpha_k,\frac{\pi}{3})$. In Section \ref{volume}, an explicit formula for $\vol( R_t(\alpha_k,\frac{\pi}{3}))$ is given by \eqref{eq:trunc-36}. Furthermore, by Remark \ref{Schlaefli}, $\vol (R_t(\alpha_k,\frac{\pi}{3}))$ is strictly increasing when $\alpha_k$ decreases, that is, when $d$ increases. In particular, for $\,\sqrt{3}<d\le2\,$, we get the lower volume bound
\[
\vol (R_t(\alpha_k,\frac{\pi}{3}))\ge\vol (R_t(\frac{\pi}{7},\frac{\pi}{3}))>0.317811>v_*\,\,.
\]

In view of Proposition \ref{prop:polarplane}, it remains to identify groups $\Gamma$ with full-sized horoballs satisfying $2<d\le2.013813$ (see Remark \ref{bound-d}).
It turns out that $\Gamma$ has the same basic structure as before, and truncation of its fundamental orthoscheme with the polar plane $P_0$ gives a polyhedron $R_t$ of smallest volume because of the following proposition.
\begin{proposition}\label{prop:dbigger2}
For any horoball on an angle bisector as depicted in Figure \ref{fig:oneoverdinlinedbig}, its bisector with $B_{\infty}$ intersects the polar plane $P_0$ exactly below the boundary of the triangle $D$ at height $1$ in the upper half space $U^3$.
  Any other cutting plane leads to a bigger volume than the volume of $R_t$.
  \begin{proof}
    The proof is basic trigonometry.
    If the centre of a horoball $B$ on an angle bisector of $D$ has distance $a$ from the boundary of $D$ and touches the bisector of the two closest full-sized horoballs, then its radius $r$ satisfies
    \begin{equation}
      (1+r)^2 = l^2 + r^2 = \frac{d^2}{4} + a^2 + r^2\,\,,
      \label{eq:radius}
    \end{equation}
    where $l$ is the distance between the centres of $B$ and of the next full-sized horoball.
    Figure \ref{fig:cutalong} depicts a vertical cut in $U^3$ through the centres of the two horoballs.
    The identity  \eqref{eq:radius} implies
    \begin{equation*}
      2r = \frac{d^2}{4} + a^2 - 1.
    \end{equation*}
 Using that the radius of the bisector is $\sqrt{2r}$, the height $h$ of the bisector is then independent of the distance $a$ :
    \begin{equation*}
      h^2 = \left(\sqrt{2r}\right)^2 - a^2 = \frac{d^2}{4} - 1.
    \end{equation*}
  \end{proof}
\end{proposition}
\begin{figure}[ht]
  \centering
  \begin{tikzpicture}[scale=2.5]
    \pgfmathsetmacro{\l}{1.3}
    \pgfmathsetmacro{\radius}{\l^2-1}
    \pgfmathsetmacro{\height}{2}
    \pgfmathsetmacro{\eps}{0.2}
    \draw (0, 0.5) circle[radius=0.5];
    \draw[blue] (\l,\radius/2) circle[radius=\radius/2];
    \draw[thick] (-1.5, 0) -- (\l+1.5, 0);
    \draw (-1, 1) -- (\l+1, 1);
    \draw[dashed] (1, 0) arc[radius=1, start angle=0, end angle=180];
    \pgfmathsetmacro{\bisrad}{sqrt(\radius)}
    \draw[dashed, blue] (\l+\bisrad, 0) arc[radius=\bisrad, start angle=0, end angle=180];
    \draw (0, 0) -- (\l, \radius/2) node[midway, above]{1} node[very near end, above]{$r$}
      -- (\l, 0) node[midway, right]{$r$};
    \draw[|-|] (0, -\eps) -- (\l, -\eps)
      node[below, midway]{$l = \sqrt{\frac{d^2}{4} + a^2}$};
    \draw[|-|] (\l, -\eps) -- (\l+\bisrad, -\eps)
    node[below, midway]{$\sqrt{2r}$};
  \end{tikzpicture}
  \caption{Vertical cut along $l$ through the centres of the two horoballs in $U^3$ }
  \label{fig:cutalong}
\end{figure}
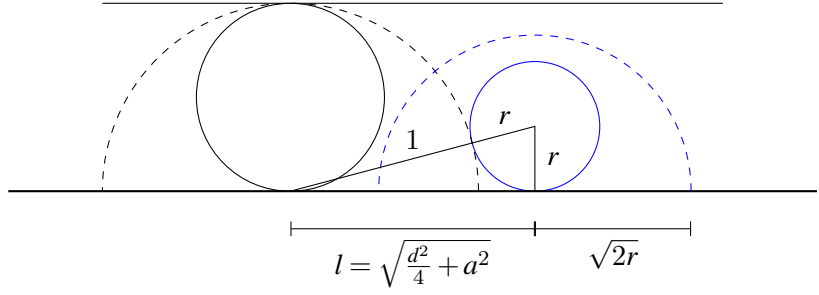
By means of Proposition \ref{prop:dbigger2} and as a final step in the analysis of horoball diagrams $D$ with $(\frac{1}{d_k})$-balls aligned on the edges of $D$ according to Figure \ref{fig:oneoverdinline}, we can state the following result.
\begin{proposition}\label{prop:trunc-dbigger2}
  Under the assumptions of Proposition \ref{prop:reflect} with $d>2$, the minimal possible volume is obtained by an orbifold
  related to a truncated orthoscheme $R_t(\beta,\frac{\pi}{6})$
  with Vinberg graph
\[
\gr{\beta}\gr{6}\gd\bullet\cdot\cdots\bullet\qquad\hbox{where}\quad 0<\beta<\frac{\pi}{15}\,\,.
\]
\end{proposition}
\begin{proof}
We have to verify the bound $\beta<\frac{\pi}{15}$, only. In fact, we are assuming the upper bound $d\le2.013813$ according to Remark \ref{bound-d}. By a similar trigonometrical computation in the upper half space $U^3$ as in the proof of Proposition \ref{prop:dbigger2} above, one can relate $\beta$ and $d$ according to
\[
\cos\beta=\frac{d}{2\,\sqrt{d^2-3}}\,\,.
\]
In this way, we get the estimate $\beta<\frac{\pi}{15}$.
\end{proof}
By Proposition \ref{prop:trunc-dbigger2}, the minimal possible orbifold volume in the case $\,d>2\,$ is given by the volume of the truncated Coxeter orthoscheme $R_t(\beta,\frac{\pi}{6})$. By \eqref{eq:trunc-63} of Section \ref{volume}, we have an explicit formula for $\vol( R_t(\beta,\frac{\pi}{6}))$ which, by its strict monotonicity behavior with respect to $\beta<\frac{\pi}{15}$, yields the lower volume bound
\[
\vol (R_t(\beta,\frac{\pi}{6}))\ge\vol (R_t(\frac{\pi}{15},\frac{\pi}{6}))>0.416491>v_*\,\,.
\]

These investigations complete the analysis of the case that the centres of the $(\frac{1}{d})$-balls lie on the edges of $D$.
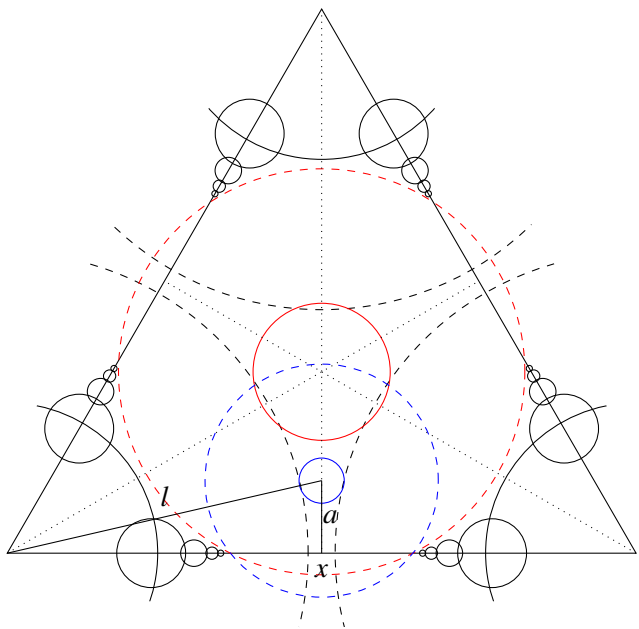
\begin{figure}[ht]
  \centering
    \begin{tikzpicture}[scale=4]
    \pgfmathsetmacro{\nn}{10}
    \pgfmathsetmacro{\d}{2.09}
    \pgfmathsetmacro{\height}{\d*sqrt(3)/2}
    \pgfmathsetmacro{\heightb}{\d/sqrt(3)}
    \pgfmathsetmacro{\w}{\d-1/\d}
    \clip (210:\heightb*1.1) -- (270:\heightb/1.4) -- (330:\heightb*1.1)
      -- (30:\heightb/1.4) -- (90:\heightb*1.1) -- (150:\heightb/1.4) -- cycle;
    \draw (210:\heightb) -- (330:\heightb) -- (90:\heightb) -- cycle;
    \draw (210:\heightb) circle[radius=0.5];
    \draw (330:\heightb) circle[radius=0.5];
    \draw (90:\heightb) circle[radius=0.5];
    \draw[dashed] (210:\heightb) circle[radius=1];
    \draw[dashed] (330:\heightb) circle[radius=1];
    \draw[dashed] (90:\heightb) circle[radius=1];
    \pgfmathsetmacro{\dkglobal}{\d}
    \pgfmathsetmacro{\rkglobal}{1/\d}
    \foreach \k in {1,...,\numexpr (\nn-2)/2 \relax}
    {
      \foreach \x in {210, 330, 90}
      {
        \foreach \y in {-150, 150}
        {
          \draw ($(\x:\heightb) + (\x+\y:1/\dkglobal)$)
            circle[radius=0.5*\rkglobal^2];
        }
      }
      \pgfmathparse{\d-1/\dkglobal}
      \global\let\dkglobal=\pgfmathresult
      \pgfmathparse{\rkglobal/\dkglobal}
      \global\let\rkglobal=\pgfmathresult
    }
    \draw[red] (0, 0) circle[radius=(\heightb^2-1)/2];
    \draw[dashed, red] (0, 0) circle[radius=sqrt(\heightb^2-1)];
    \draw[dotted] (210:\heightb) -- +(30:\height);
    \draw[dotted] (330:\heightb) -- +(30+120:\height);
    \draw[dotted] (90:\heightb) -- +(30-120:\height);
    \pgfmathsetmacro{\aaa}{0.4*\d/sqrt(3)/2}
    \pgfmathsetmacro{\aaaradius}{\d^2/4+\aaa^2-1}
    \draw[dashed, blue] (0, \aaa-\heightb/2) circle[radius=sqrt(\aaaradius)];
    \draw[blue] (0, \aaa-\heightb/2) circle[radius=\aaaradius/2];
    \draw (210:\heightb) -- (0, \aaa-\heightb/2) node[midway,above] {$l$};
    \draw (0, -\heightb/2) -- (0, \aaa-\heightb/2) node[midway] {\hspace{0.6em}$a$};
    \draw (0, -\heightb/2) node[below]{$x$};
  \end{tikzpicture}
 \caption{$(\frac{1}{d})$-balls on the edges of $D$ with $d>2$}
  \label{fig:oneoverdinlinedbig}
\end{figure}

In order to finish the part of full-sized horoballs centred at singular points of order $6$ and of distance $d>1$, we need to study the case when the centres of the $(\frac{1}{d})$-balls lie on the angle bisectors of $D$. We are still assuming that $\Gamma_{\infty}$ is a reflection group.

The $(\frac{1}{d})$-balls cannot touch each other because then three balls would touch pairwise which is impossible because the full-sized balls do not touch.
This implies the existence of $(\frac{1}{w})$-balls.
For a fixed distance $d$, the position of the $(\frac{1}{w})$-balls is given by the similarity of the two triangles $(j, z, l)$ and $(j, x, p)$.
See Figure \ref{fig:horoedwschief} for the labels.
This similarity and the reflection symmetry of the horoball diagram $D$ also prove that the centre of the full-sized horoball $B_j$ is on a common line with those of the $(\frac{1}{d})$-ball $B_z$ and the $(\frac{1}{w})$-ball $B_p$.

\begin{figure}
  \begin{center}
    \begin{tikzpicture}[scale=3]
      \pgfmathsetmacro{\d}{sqrt(1+sqrt(3))}
      \pgfmathsetmacro{\height}{\d*sqrt(3)/2}
      \pgfmathsetmacro{\angle}{30}
      \pgfmathsetmacro{\w}{sqrt(\d^2+1/\d^2-2*cos(\angle))}
      \pgfmathsetmacro{\angleb}{acos(0.5*(\w/\d+\d/\w-1/(\w*\d^3)))}
      \draw (0, 0) -- (\d, 0) -- (60:\d) -- cycle;
      \draw (0, 0) circle[radius=0.5] node[below left]{$j$};
      \draw (\d, 0) circle[radius=0.5] node[below right]{$k$};
      \draw (60:\d) circle[radius=0.5] node[left]{$l$};
      \draw ($(0, 0)+(\angle:1/\d)$) circle[radius=0.5/\d^2] node[below right]{$x$};
      \draw ($(\d, 0)+(\angle+120:1/\d)$) circle[radius=0.5/\d^2] node[above right]{$y$};
      \draw ($(60:\d)+(\angle-120:1/\d)$) circle[radius=0.5/\d^2] node[right]{$z$};
      \draw ($(0, 0)+(\angle+\angleb:1/\w)$)
        circle[radius=0.5/\d^2/\w^2] node[above]{$p$};
      \draw ($(\d, 0)+(\angle+\angleb+120:1/\w)$)
        circle[radius=0.5/\d^2/\w^2] node[right]{$q$};
      \draw ($(60:\d)+(\angle+\angleb-120:1/\w)$)
        circle[radius=0.5/\d^2/\w^2] node[above left]{$t$};
      \foreach \x in {-1, 1}{
        \draw ($(0, 0)+(60*\x+\angle:1/\d)$) circle[radius=0.5/\d^2];
        \draw ($(\d, 0)+(60*\x+\angle+120:1/\d)$) circle[radius=0.5/\d^2];
        \draw ($(60:\d)+(60*\x+\angle-120:1/\d)$) circle[radius=0.5/\d^2];
      }
      \draw ($(0, 0)+(-60+\angle+\angleb:1/\w)$)
          circle[radius=0.5/\d^2/\w^2];
      \draw ($(\d, 0)+(-60+\angle+\angleb+120:1/\w)$)
          circle[radius=0.5/\d^2/\w^2];
      \draw ($(60:\d)+(-60+\angle+\angleb-120:1/\w)$)
          circle[radius=0.5/\d^2/\w^2];
      \draw[dotted] (0, 0) -- +(30:\height);
      \draw[dotted] (\d, 0) -- +(30+120:\height);
      \draw[dotted] (60:\d) -- +(30-120:\height);
      \coordinate (A) at (0, 0);
      \coordinate (B) at (30:1/\d);
      \coordinate (C) at (\angle+\angleb:1/\w);
      \draw (A) -- (B) -- (C)
        pic [draw=green!50!black, fill=green!20, angle radius=3mm] {angle = C--B--A};
      \coordinate (D) at (60:\d);
      \coordinate (E) at ($(60:\d) + (-90:1/\d)$);
      \draw (D) -- (E) -- (A)
        pic [draw=green!50!black, fill=green!20, angle radius=3mm] {angle = D--E--A};
    \end{tikzpicture}
  \end{center}
  \caption{Horoball diagram $D$ with $d=\sqrt{1+\sqrt{3}}$}
  \label{fig:horoedwschief}
\end{figure}
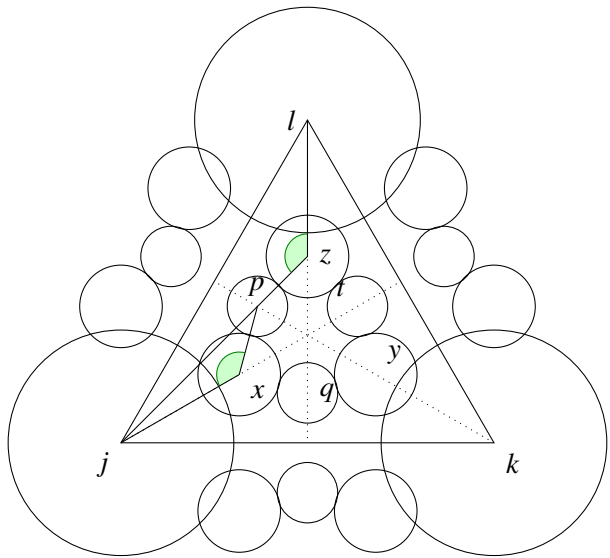

Note that Figure \ref{fig:horoedwschief} depicts the case where the $(\frac{1}{w})$-balls touch two $(\frac{1}{d})$-balls but the above  statement about the alignment of their centres stays true if the $(\frac{1}{w})$-balls do not touch two $(\frac{1}{d})$-balls.
This in turn means that each $(\frac{1}{d})$-ball touches two $(\frac{1}{w})$-balls because of the present symmetries.
Hence, there are twelve $(\frac{1}{w})$-balls around a full-sized horoball.

Knowing that the centres $j$, $p$, and $z$ are on the same line, we can calculate $d$ by using
\begin{equation*}
  w = \frac{1}{w}+\frac{1}{wd^2}
\end{equation*}
and the identity $w^2=d^2 + \frac{1}{d^2}-2\cos\frac{\pi}{6}$ according to \eqref{eq:3}.
This implies $d^2=1+\sqrt{3}$.

The following considerations allow us to describe the orbifold associated to this horoball diagram.
They are also valid if $d$ is bigger, that is, if each of the $(\frac{1}{w})$-balls touches only one $(\frac{1}{d})$-ball.
Since we assume that $\Gamma_\infty$ is a reflection group, there is a reflection $s\in\Gamma_\infty$ which exchanges $B_j$ and $B_k$.
By Lemma \ref{elliptic}, there is an order 2 elliptic element $r$ in $\Gamma$ which swaps
\begin{equation*}
  B_j  \leftrightarrow B_\infty\qquad,\qquad B_k  \leftrightarrow B_x\,\,.
\end{equation*}
Hence, $rsr$ is the reflection in the bisector of $B_\infty$ and $B_x$ and an element of the group $\Gamma$.
This proves the following result.

\begin{proposition}
  In the case of a $\left\{ 2,3,6 \right\}$-cusp where $\Gamma_\infty$ is a reflection group and where the $(\frac{1}{d})$-balls are centred on the angle bisectors of the horoball diagram $D$, there are reflections in $\Gamma$ mapping $B_\infty$ to any $(\frac{1}{d})$-ball.
  \label{prop:reflectioned}
\end{proposition}

In the situation of Figure \ref{fig:horoedwschief}, the distance $u$ between two neighbouring $(\frac{1}{d})$-balls can be calculated as $\frac{1}{d}$.
Using this information, one can consider the polyhedron created from the triangular cone with apex $\infty$ and angles $\frac{\pi}{2}, \frac{\pi}{3}, \frac{\pi}{6}$ by cutting with the bisector/reflection plane of $rsr$.
The bisector is orthogonal to one side of the cone and has angle $\frac{\pi}{3}$ with the other two because its centre has (Euclidean) distance $\frac{1}{2d}$ from the hyperplanes and the radius of the bisector is $\frac{1}{d}$.
Thus, it can be seen that this polyhedron is the fundamental polyhedron of the (non-arithmetic) reflection group $[(3^3,6)]$ with Coxeter graph
\begin{center}
  \begin{tikzpicture}[anchor=base, scale=0.5, baseline,
      bullet/.style={circle, fill=black, minimum size=5pt, inner sep=0pt}]
    \draw (-1, 1) -- (1, 1) node[midway, above]{6};
    \draw (-1,1) node[bullet]{}
      -- (-1, -1) node[bullet]{}
      -- (1, -1) node[bullet]{}
      -- (1, 1) node[bullet]{};
  \end{tikzpicture}\,\,.
\end{center}

\medskip
One can see that $r$ induces the internal symmetry of that polyhedron (and its graph).
Hence, $\Gamma$ is equal to the group extension $\Gamma_{\circ}^r=[(3^3,6)]\ast C_r$ of $[(3^3,6)]$ by means of the cyclic group generated by $r$; see Remark \ref{incomm}. The volume of the tetrahedron $[(3^3,6)]$ as given by \eqref{eq:volcycle} is bigger than $0.36$ so that $\vol(\mathbb H^3/\Gamma_{\circ}^r)>v_*$.

If we do not assume that the $(\frac{1}{w})$-balls touch two $(\frac{1}{d})$-balls, we have basically the same case distinction as when the $(\frac{1}{d_k})$-balls are centred on the edges $D$.
\begin{enumerate}
  \item There is a single $(\frac{1}{d_{k_{\text{max}}}})$-ball, or
  \item there are two touching $(\frac{1}{d_{k_{\text{max}}}})$-balls,
\end{enumerate}
where $k_\text{max}:=\max\left\{k\mid\text{the horoball diagram $D$ contains a $(\frac{1}{d_k})$-ball}\right\}$.
See Figure \ref{fig:anglepi6} for a sketch of the first couple of horoball diagrams in that sequence.

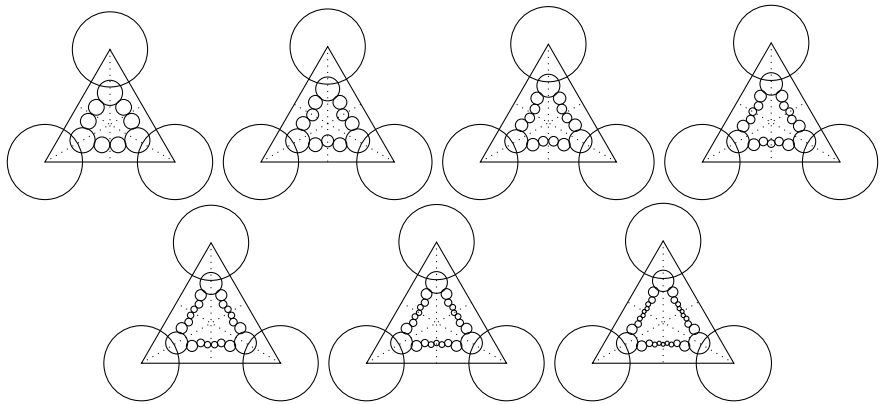
\begin{figure}[ht]
  \centering
    \begin{tikzpicture}
      \pgfmathsetmacro{\d}{1.73}
      \pgfmathsetmacro{\height}{\d*sqrt(3)/2}
      \pgfmathsetmacro{\angle}{30}
      \pgfmathsetmacro{\w}{sqrt(\d^2+1/\d^2-2*cos(\angle))}
      \pgfmathsetmacro{\angleb}{acos(0.5*(\w/\d+\d/\w-1/(\w*\d^3)))}
      \pgfmathsetmacro{\x}{\w-1/(\w*\d^2)}
      \draw (0, 0) -- (\d, 0) -- (60:\d) -- cycle;
      \draw (0, 0) circle[radius=0.5];
      \draw (\d, 0) circle[radius=0.5];
      \draw (60:\d) circle[radius=0.5];
      \draw ($(0, 0)+(\angle:1/\d)$) circle[radius=0.5/\d^2];
      \draw ($(\d, 0)+(\angle+120:1/\d)$) circle[radius=0.5/\d^2];
      \draw ($(60:\d)+(\angle-120:1/\d)$) circle[radius=0.5/\d^2];
      \draw ($(0, 0)+(\angle+\angleb:1/\w)$)
        circle[radius=0.5/\d^2/\w^2];
      \draw ($(\d, 0)+(\angle+\angleb+120:1/\w)$)
        circle[radius=0.5/\d^2/\w^2];
      \draw ($(60:\d)+(\angle+\angleb-120:1/\w)$)
        circle[radius=0.5/\d^2/\w^2];
      \draw ($(0, 0)+(\angle-\angleb:1/\w)$)
        circle[radius=0.5/\d^2/\w^2];
      \draw ($(\d, 0)+(\angle-\angleb+120:1/\w)$)
        circle[radius=0.5/\d^2/\w^2];
      \draw ($(60:\d)+(\angle-\angleb-120:1/\w)$)
        circle[radius=0.5/\d^2/\w^2];
      \draw[dotted] (0, 0) -- +(30:\height);
      \draw[dotted] (\d, 0) -- +(30+120:\height);
      \draw[dotted] (60:\d) -- +(30-120:\height);
    \end{tikzpicture}
    \begin{tikzpicture}
      \pgfmathsetmacro{\d}{sqrt(sqrt(2)+sqrt(3))}
      \pgfmathsetmacro{\height}{\d*sqrt(3)/2}
      \pgfmathsetmacro{\angle}{30}
      \pgfmathsetmacro{\w}{sqrt(\d^2+1/\d^2-2*cos(\angle))}
      \pgfmathsetmacro{\angleb}{acos(0.5*(\w/\d+\d/\w-1/(\w*\d^3)))}
      \pgfmathsetmacro{\x}{\w-1/(\w*\d^2)}
      \draw (0, 0) -- (\d, 0) -- (60:\d) -- cycle;
      \draw (0, 0) circle[radius=0.5];
      \draw (\d, 0) circle[radius=0.5];
      \draw (60:\d) circle[radius=0.5];
      \draw ($(0, 0)+(\angle:1/\d)$) circle[radius=0.5/\d^2];
      \draw ($(\d, 0)+(\angle+120:1/\d)$) circle[radius=0.5/\d^2];
      \draw ($(60:\d)+(\angle-120:1/\d)$) circle[radius=0.5/\d^2];
      \draw ($(0, 0)+(\angle+\angleb:1/\w)$)
        circle[radius=0.5/\d^2/\w^2];
      \draw ($(\d, 0)+(\angle+\angleb+120:1/\w)$)
        circle[radius=0.5/\d^2/\w^2];
      \draw ($(60:\d)+(\angle+\angleb-120:1/\w)$)
        circle[radius=0.5/\d^2/\w^2];
      \draw ($(0, 0)+(\angle-\angleb:1/\w)$)
        circle[radius=0.5/\d^2/\w^2];
      \draw ($(\d, 0)+(\angle-\angleb+120:1/\w)$)
        circle[radius=0.5/\d^2/\w^2];
      \draw ($(60:\d)+(\angle-\angleb-120:1/\w)$)
        circle[radius=0.5/\d^2/\w^2];
      \draw ($(0, 0)+(\angle-\angleb:1/\x)$)
        circle[radius=0.5/\d^2/\w^2/\x^2];
      \draw ($(\d, 0)+(\angle-\angleb+120:1/\x)$)
        circle[radius=0.5/\d^2/\w^2/\x^2];
      \draw ($(60:\d)+(\angle-\angleb-120:1/\x)$)
        circle[radius=0.5/\d^2/\w^2/\x^2];
      \draw[dotted] (0, 0) -- +(30:\height);
      \draw[dotted] (\d, 0) -- +(30+120:\height);
      \draw[dotted] (60:\d) -- +(30-120:\height);
    \end{tikzpicture}
    \begin{tikzpicture}
      \pgfmathsetmacro{\d}{1.81}
      \pgfmathsetmacro{\height}{\d*sqrt(3)/2}
      \pgfmathsetmacro{\angle}{30}
      \pgfmathsetmacro{\w}{sqrt(\d^2+1/\d^2-2*cos(\angle))}
      \pgfmathsetmacro{\angleb}{acos(0.5*(\w/\d+\d/\w-1/(\w*\d^3)))}
      \pgfmathsetmacro{\x}{\w-1/(\w*\d^2)}
      \pgfmathsetmacro{\y}{sqrt(\d^2+1/\x^2 - 2*\d/\x*cos(30-\angleb))}
      \draw (0, 0) -- (\d, 0) -- (60:\d) -- cycle;
      \draw (0, 0) circle[radius=0.5];
      \draw (\d, 0) circle[radius=0.5];
      \draw (60:\d) circle[radius=0.5];
      \draw ($(0, 0)+(\angle:1/\d)$) circle[radius=0.5/\d^2];
      \draw ($(\d, 0)+(\angle+120:1/\d)$) circle[radius=0.5/\d^2];
      \draw ($(60:\d)+(\angle-120:1/\d)$) circle[radius=0.5/\d^2];
      \draw ($(0, 0)+(\angle+\angleb:1/\w)$)
        circle[radius=0.5/\d^2/\w^2];
      \draw ($(\d, 0)+(\angle+\angleb+120:1/\w)$)
        circle[radius=0.5/\d^2/\w^2];
      \draw ($(60:\d)+(\angle+\angleb-120:1/\w)$)
        circle[radius=0.5/\d^2/\w^2];
      \coordinate (wa) at ($(0, 0)+(\angle-\angleb:1/\w)$);
      \coordinate (wa2) at ($(\d, 0)+(\angle+\angleb+120:1/\w)$);
      \coordinate (wb) at ($(\d, 0)+(\angle-\angleb+120:1/\w)$);
      \coordinate (wb2) at ($(60:\d)+(\angle+\angleb-120:1/\w)$);
      \coordinate (wc) at ($(60:\d)+(\angle-\angleb-120:1/\w)$);
      \coordinate (wc2) at ($(0, 0)+(\angle+\angleb:1/\w)$);
      \draw ($(0, 0)+(\angle-\angleb:1/\w)$)
        circle[radius=0.5/\d^2/\w^2];
      \draw ($(\d, 0)+(\angle-\angleb+120:1/\w)$)
        circle[radius=0.5/\d^2/\w^2];
      \draw ($(60:\d)+(\angle-\angleb-120:1/\w)$)
        circle[radius=0.5/\d^2/\w^2];
      \draw ($(0, 0)+(\angle+\angleb:1/\x)$)
        circle[radius=0.5/\d^2/\w^2/\x^2];
      \draw ($(\d, 0)+(\angle+\angleb+120:1/\x)$)
        circle[radius=0.5/\d^2/\w^2/\x^2];
      \draw ($(60:\d)+(\angle+\angleb-120:1/\x)$)
        circle[radius=0.5/\d^2/\w^2/\x^2];
      \draw ($(0, 0)+(\angle-\angleb:1/\x)$)
        circle[radius=0.5/\d^2/\w^2/\x^2];
      \draw ($(\d, 0)+(\angle-\angleb+120:1/\x)$)
        circle[radius=0.5/\d^2/\w^2/\x^2];
      \draw ($(60:\d)+(\angle-\angleb-120:1/\x)$)
        circle[radius=0.5/\d^2/\w^2/\x^2];
      \draw[dotted] (0, 0) -- +(30:\height);
      \draw[dotted] (\d, 0) -- +(30+120:\height);
      \draw[dotted] (60:\d) -- +(30-120:\height);
    \end{tikzpicture}
    \begin{tikzpicture}
      \pgfmathsetmacro{\d}{sqrt(2*cos(36)+sqrt(3))}
      \pgfmathsetmacro{\height}{\d*sqrt(3)/2}
      \pgfmathsetmacro{\angle}{30}
      \pgfmathsetmacro{\w}{sqrt(\d^2+1/\d^2-2*cos(\angle))}
      \pgfmathsetmacro{\angleb}{acos(0.5*(\w/\d+\d/\w-1/(\w*\d^3)))}
      \pgfmathsetmacro{\x}{\w-1/(\w*\d^2)}
      \pgfmathsetmacro{\y}{sqrt(\d^2+1/\x^2 - 2*\d/\x*cos(30-\angleb))}
      \draw (0, 0) -- (\d, 0) -- (60:\d) -- cycle;
      \draw (0, 0) circle[radius=0.5];
      \draw (\d, 0) circle[radius=0.5];
      \draw (60:\d) circle[radius=0.5];
      \draw ($(0, 0)+(\angle:1/\d)$) circle[radius=0.5/\d^2];
      \draw ($(\d, 0)+(\angle+120:1/\d)$) circle[radius=0.5/\d^2];
      \draw ($(60:\d)+(\angle-120:1/\d)$) circle[radius=0.5/\d^2];
      \draw ($(0, 0)+(\angle+\angleb:1/\w)$)
        circle[radius=0.5/\d^2/\w^2];
      \draw ($(\d, 0)+(\angle+\angleb+120:1/\w)$)
        circle[radius=0.5/\d^2/\w^2];
      \draw ($(60:\d)+(\angle+\angleb-120:1/\w)$)
        circle[radius=0.5/\d^2/\w^2];
      \coordinate (wa) at ($(0, 0)+(\angle-\angleb:1/\w)$);
      \coordinate (wa2) at ($(\d, 0)+(\angle+\angleb+120:1/\w)$);
      \coordinate (wb) at ($(\d, 0)+(\angle-\angleb+120:1/\w)$);
      \coordinate (wb2) at ($(60:\d)+(\angle+\angleb-120:1/\w)$);
      \coordinate (wc) at ($(60:\d)+(\angle-\angleb-120:1/\w)$);
      \coordinate (wc2) at ($(0, 0)+(\angle+\angleb:1/\w)$);
      \draw ($(0, 0)+(\angle-\angleb:1/\w)$)
        circle[radius=0.5/\d^2/\w^2];
      \draw ($(\d, 0)+(\angle-\angleb+120:1/\w)$)
        circle[radius=0.5/\d^2/\w^2];
      \draw ($(60:\d)+(\angle-\angleb-120:1/\w)$)
        circle[radius=0.5/\d^2/\w^2];
      \draw ($(0, 0)+(\angle+\angleb:1/\x)$)
        circle[radius=0.5/\d^2/\w^2/\x^2];
      \draw ($(\d, 0)+(\angle+\angleb+120:1/\x)$)
        circle[radius=0.5/\d^2/\w^2/\x^2];
      \draw ($(60:\d)+(\angle+\angleb-120:1/\x)$)
        circle[radius=0.5/\d^2/\w^2/\x^2];
      \draw ($(0, 0)+(\angle-\angleb:1/\x)$)
        circle[radius=0.5/\d^2/\w^2/\x^2];
      \draw ($(\d, 0)+(\angle-\angleb+120:1/\x)$)
        circle[radius=0.5/\d^2/\w^2/\x^2];
      \draw ($(60:\d)+(\angle-\angleb-120:1/\x)$)
        circle[radius=0.5/\d^2/\w^2/\x^2];
      \draw ($(0, 0)+(\angle+15:1/\y)$)
        circle[radius=0.5/\d^2/\w^2/\x^2/\y^2];
      \draw ($(\d, 0)+(120+\angle+15:1/\y)$)
        circle[radius=0.5/\d^2/\w^2/\x^2/\y^2];
      \draw ($(60:\d)+(-120+\angle+15:1/\y)$)
        circle[radius=0.5/\d^2/\w^2/\x^2/\y^2];
      \draw[dotted] (0, 0) -- +(30:\height);
      \draw[dotted] (\d, 0) -- +(30+120:\height);
      \draw[dotted] (60:\d) -- +(30-120:\height);
    \end{tikzpicture}
    \begin{tikzpicture}
      \pgfmathsetmacro{\d}{1.85}
      \pgfmathsetmacro{\height}{\d*sqrt(3)/2}
      \pgfmathsetmacro{\angle}{30}
      \pgfmathsetmacro{\w}{sqrt(\d^2+1/\d^2-2*cos(\angle))}
      \pgfmathsetmacro{\angleb}{acos(0.5*(\w/\d+\d/\w-1/(\w*\d^3)))}
      \pgfmathsetmacro{\x}{\w-1/(\w*\d^2)}
      \pgfmathsetmacro{\y}{sqrt(\d^2+1/\x^2 - 2*\d/\x*cos(30-\angleb))}
      \pgfmathsetmacro{\anglec}{acos(0.5*(\d^2+\y^2-1/\x^2)/(\d*\y))}
      \pgfmathsetmacro{\z}{sqrt(\d^2+1/\y^2 - 2*\d/\y*cos(30-\anglec))}
      \pgfmathsetmacro{\angled}{acos(0.5*(\d^2+\z^2-1/\y^2)/(\d*\z))}
      \draw (0, 0) -- (\d, 0) -- (60:\d) -- cycle;
      \draw (0, 0) circle[radius=0.5];
      \draw (\d, 0) circle[radius=0.5];
      \draw (60:\d) circle[radius=0.5];
      \draw ($(0, 0)+(\angle:1/\d)$) circle[radius=0.5/\d^2];
      \draw ($(\d, 0)+(\angle+120:1/\d)$) circle[radius=0.5/\d^2];
      \draw ($(60:\d)+(\angle-120:1/\d)$) circle[radius=0.5/\d^2];
      \draw ($(0, 0)+(\angle+\angleb:1/\w)$)
        circle[radius=0.5/\d^2/\w^2];
      \draw ($(\d, 0)+(\angle+\angleb+120:1/\w)$)
        circle[radius=0.5/\d^2/\w^2];
      \draw ($(60:\d)+(\angle+\angleb-120:1/\w)$)
        circle[radius=0.5/\d^2/\w^2];
      \draw ($(0, 0)+(\angle-\angleb:1/\w)$)
        circle[radius=0.5/\d^2/\w^2];
      \draw ($(\d, 0)+(\angle-\angleb+120:1/\w)$)
        circle[radius=0.5/\d^2/\w^2];
      \draw ($(60:\d)+(\angle-\angleb-120:1/\w)$)
        circle[radius=0.5/\d^2/\w^2];
      \draw ($(0, 0)+(\angle+\angleb:1/\x)$)
        circle[radius=0.5/\d^2/\w^2/\x^2];
      \draw ($(\d, 0)+(\angle+\angleb+120:1/\x)$)
        circle[radius=0.5/\d^2/\w^2/\x^2];
      \draw ($(60:\d)+(\angle+\angleb-120:1/\x)$)
        circle[radius=0.5/\d^2/\w^2/\x^2];
      \draw ($(0, 0)+(\angle-\angleb:1/\x)$)
        circle[radius=0.5/\d^2/\w^2/\x^2];
      \draw ($(\d, 0)+(\angle-\angleb+120:1/\x)$)
        circle[radius=0.5/\d^2/\w^2/\x^2];
      \draw ($(60:\d)+(\angle-\angleb-120:1/\x)$)
        circle[radius=0.5/\d^2/\w^2/\x^2];
      \draw ($(0, 0)+(\angle+\anglec:1/\y)$)
        circle[radius=0.5/\d^2/\w^2/\x^2/\y^2];
      \draw ($(\d, 0)+(120+\angle+\anglec:1/\y)$)
        circle[radius=0.5/\d^2/\w^2/\x^2/\y^2];
      \draw ($(60:\d)+(-120+\angle+\anglec:1/\y)$)
        circle[radius=0.5/\d^2/\w^2/\x^2/\y^2];
      \draw ($(0, 0)+(\angle-\anglec:1/\y)$)
        circle[radius=0.5/\d^2/\w^2/\x^2/\y^2];
      \draw ($(\d, 0)+(120+\angle-\anglec:1/\y)$)
        circle[radius=0.5/\d^2/\w^2/\x^2/\y^2];
      \draw ($(60:\d)+(-120+\angle-\anglec:1/\y)$)
        circle[radius=0.5/\d^2/\w^2/\x^2/\y^2];
      \draw[dotted] (0, 0) -- +(30:\height);
      \draw[dotted] (\d, 0) -- +(30+120:\height);
      \draw[dotted] (60:\d) -- +(30-120:\height);
    \end{tikzpicture}
    \begin{tikzpicture}
      \pgfmathsetmacro{\d}{sqrt(2*sqrt(3))}
      \pgfmathsetmacro{\height}{\d*sqrt(3)/2}
      \pgfmathsetmacro{\angle}{30}
      \pgfmathsetmacro{\w}{sqrt(\d^2+1/\d^2-2*cos(\angle))}
      \pgfmathsetmacro{\angleb}{acos(0.5*(\w/\d+\d/\w-1/(\w*\d^3)))}
      \pgfmathsetmacro{\x}{\w-1/(\w*\d^2)}
      \pgfmathsetmacro{\y}{sqrt(\d^2+1/\x^2 - 2*\d/\x*cos(30-\angleb))}
      \pgfmathsetmacro{\anglec}{acos(0.5*(\d^2+\y^2-1/\x^2)/(\d*\y))}
      \pgfmathsetmacro{\z}{sqrt(\d^2+1/\y^2 - 2*\d/\y*cos(30-\anglec))}
      \pgfmathsetmacro{\angled}{acos(0.5*(\d^2+\z^2-1/\y^2)/(\d*\z))}
      \draw (0, 0) -- (\d, 0) -- (60:\d) -- cycle;
      \draw (0, 0) circle[radius=0.5];
      \draw (\d, 0) circle[radius=0.5];
      \draw (60:\d) circle[radius=0.5];
      \draw ($(0, 0)+(\angle:1/\d)$) circle[radius=0.5/\d^2];
      \draw ($(\d, 0)+(\angle+120:1/\d)$) circle[radius=0.5/\d^2];
      \draw ($(60:\d)+(\angle-120:1/\d)$) circle[radius=0.5/\d^2];
      \draw ($(0, 0)+(\angle+\angleb:1/\w)$)
        circle[radius=0.5/\d^2/\w^2];
      \draw ($(\d, 0)+(\angle+\angleb+120:1/\w)$)
        circle[radius=0.5/\d^2/\w^2];
      \draw ($(60:\d)+(\angle+\angleb-120:1/\w)$)
        circle[radius=0.5/\d^2/\w^2];
      \draw ($(0, 0)+(\angle-\angleb:1/\w)$)
        circle[radius=0.5/\d^2/\w^2];
      \draw ($(\d, 0)+(\angle-\angleb+120:1/\w)$)
        circle[radius=0.5/\d^2/\w^2];
      \draw ($(60:\d)+(\angle-\angleb-120:1/\w)$)
        circle[radius=0.5/\d^2/\w^2];
      \draw ($(0, 0)+(\angle+\angleb:1/\x)$)
        circle[radius=0.5/\d^2/\w^2/\x^2];
      \draw ($(\d, 0)+(\angle+\angleb+120:1/\x)$)
        circle[radius=0.5/\d^2/\w^2/\x^2];
      \draw ($(60:\d)+(\angle+\angleb-120:1/\x)$)
        circle[radius=0.5/\d^2/\w^2/\x^2];
      \draw ($(0, 0)+(\angle-\angleb:1/\x)$)
        circle[radius=0.5/\d^2/\w^2/\x^2];
      \draw ($(\d, 0)+(\angle-\angleb+120:1/\x)$)
        circle[radius=0.5/\d^2/\w^2/\x^2];
      \draw ($(60:\d)+(\angle-\angleb-120:1/\x)$)
        circle[radius=0.5/\d^2/\w^2/\x^2];
      \draw ($(0, 0)+(\angle+\anglec:1/\y)$)
        circle[radius=0.5/\d^2/\w^2/\x^2/\y^2];
      \draw ($(\d, 0)+(120+\angle+\anglec:1/\y)$)
        circle[radius=0.5/\d^2/\w^2/\x^2/\y^2];
      \draw ($(60:\d)+(-120+\angle+\anglec:1/\y)$)
        circle[radius=0.5/\d^2/\w^2/\x^2/\y^2];
      \draw ($(0, 0)+(\angle-\anglec:1/\y)$)
        circle[radius=0.5/\d^2/\w^2/\x^2/\y^2];
      \draw ($(\d, 0)+(120+\angle-\anglec:1/\y)$)
        circle[radius=0.5/\d^2/\w^2/\x^2/\y^2];
      \draw ($(60:\d)+(-120+\angle-\anglec:1/\y)$)
        circle[radius=0.5/\d^2/\w^2/\x^2/\y^2];
      \draw ($(0, 0)+(\angle+\angled:1/\z)$)
        circle[radius=0.5/\d^2/\w^2/\x^2/\y^2/\z^2];
      \draw ($(\d, 0)+(120+\angle+\angled:1/\z)$)
        circle[radius=0.5/\d^2/\w^2/\x^2/\y^2/\z^2];
      \draw ($(60:\d)+(-120+\angle+\angled:1/\z)$)
        circle[radius=0.5/\d^2/\w^2/\x^2/\y^2/\z^2];
      \draw[dotted] (0, 0) -- +(30:\height);
      \draw[dotted] (\d, 0) -- +(30+120:\height);
      \draw[dotted] (60:\d) -- +(30-120:\height);
    \end{tikzpicture}
    \begin{tikzpicture}
      \pgfmathsetmacro{\d}{sqrt(2*cos(180/7)+sqrt(3))}
      \pgfmathsetmacro{\height}{\d*sqrt(3)/2}
      \pgfmathsetmacro{\heightb}{\d/sqrt(3)}
      \pgfmathsetmacro{\angle}{30}
      \pgfmathsetmacro{\w}{sqrt(\d^2+1/\d^2-2*cos(\angle))}
      \pgfmathsetmacro{\angleb}{acos(0.5*(\w/\d+\d/\w-1/(\w*\d^3)))}
      \pgfmathsetmacro{\x}{\w-1/(\w*\d^2)}
      \pgfmathsetmacro{\y}{sqrt(\d^2+1/\x^2 - 2*\d/\x*cos(30-\angleb))}
      \pgfmathsetmacro{\anglec}{acos(0.5*(\d^2+\y^2-1/\x^2)/(\d*\y))}
      \pgfmathsetmacro{\z}{sqrt(\d^2+1/\y^2 - 2*\d/\y*cos(30-\anglec))}
      \pgfmathsetmacro{\angled}{acos(0.5*(\d^2+\z^2-1/\y^2)/(\d*\z))}
      \pgfmathsetmacro{\zz}{sqrt(\d^2+1/\z^2 - 2*\d/\z*cos(30-\angled))}
      \pgfmathsetmacro{\anglee}{acos(0.5*(\d^2+\zz^2-1/\z^2)/(\d*\zz))}
      \draw (0, 0) -- (\d, 0) -- (60:\d) -- cycle;
      \draw (0, 0) circle[radius=0.5];
      \draw (\d, 0) circle[radius=0.5];
      \draw (60:\d) circle[radius=0.5];
      \draw ($(0, 0)+(\angle:1/\d)$) circle[radius=0.5/\d^2];
      \draw ($(\d, 0)+(\angle+120:1/\d)$) circle[radius=0.5/\d^2];
      \draw ($(60:\d)+(\angle-120:1/\d)$) circle[radius=0.5/\d^2];
      \draw ($(0, 0)+(\angle+\angleb:1/\w)$)
      circle[radius=0.5/\d^2/\w^2];
      \draw ($(\d, 0)+(\angle+\angleb+120:1/\w)$)
        circle[radius=0.5/\d^2/\w^2];
      \draw ($(60:\d)+(\angle+\angleb-120:1/\w)$)
        circle[radius=0.5/\d^2/\w^2];
      \draw ($(0, 0)+(\angle-\angleb:1/\w)$)
        circle[radius=0.5/\d^2/\w^2];
      \draw ($(\d, 0)+(\angle-\angleb+120:1/\w)$)
        circle[radius=0.5/\d^2/\w^2];
      \draw ($(60:\d)+(\angle-\angleb-120:1/\w)$)
        circle[radius=0.5/\d^2/\w^2];
      \draw ($(0, 0)+(\angle+\angleb:1/\x)$)
        circle[radius=0.5/\d^2/\w^2/\x^2];
      \draw ($(\d, 0)+(\angle+\angleb+120:1/\x)$)
        circle[radius=0.5/\d^2/\w^2/\x^2];
      \draw ($(60:\d)+(\angle+\angleb-120:1/\x)$)
        circle[radius=0.5/\d^2/\w^2/\x^2];
      \draw ($(0, 0)+(\angle-\angleb:1/\x)$)
        circle[radius=0.5/\d^2/\w^2/\x^2];
      \draw ($(\d, 0)+(\angle-\angleb+120:1/\x)$)
        circle[radius=0.5/\d^2/\w^2/\x^2];
      \draw ($(60:\d)+(\angle-\angleb-120:1/\x)$)
        circle[radius=0.5/\d^2/\w^2/\x^2];
      \draw ($(0, 0)+(\angle+\anglec:1/\y)$)
        circle[radius=0.5/\d^2/\w^2/\x^2/\y^2];
      \draw ($(\d, 0)+(120+\angle+\anglec:1/\y)$)
        circle[radius=0.5/\d^2/\w^2/\x^2/\y^2];
      \draw ($(60:\d)+(-120+\angle+\anglec:1/\y)$)
        circle[radius=0.5/\d^2/\w^2/\x^2/\y^2];
      \draw ($(0, 0)+(\angle-\anglec:1/\y)$)
        circle[radius=0.5/\d^2/\w^2/\x^2/\y^2];
      \draw ($(\d, 0)+(120+\angle-\anglec:1/\y)$)
        circle[radius=0.5/\d^2/\w^2/\x^2/\y^2];
      \draw ($(60:\d)+(-120+\angle-\anglec:1/\y)$)
        circle[radius=0.5/\d^2/\w^2/\x^2/\y^2];
      \foreach \vz in {-1, 1}{
        \draw ($(0, 0)+(\angle+\vz*\angled:1/\z)$)
          circle[radius=0.5/\d^2/\w^2/\x^2/\y^2/\z^2];
        \draw ($(\d, 0)+(120+\angle+\vz*\angled:1/\z)$)
          circle[radius=0.5/\d^2/\w^2/\x^2/\y^2/\z^2];
        \draw ($(60:\d)+(-120+\angle+\vz*\angled:1/\z)$)
          circle[radius=0.5/\d^2/\w^2/\x^2/\y^2/\z^2];
        }
      \draw ($(0, 0)+(\angle+\anglee:1/\zz)$)
        circle[radius=0.5/\d^2/\w^2/\x^2/\y^2/\z^2/\zz^2];
      \draw ($(\d, 0)+(120+\angle+\anglee:1/\zz)$)
        circle[radius=0.5/\d^2/\w^2/\x^2/\y^2/\z^2/\zz^2];
      \draw ($(60:\d)+(-120+\angle+\anglee:1/\zz)$)
        circle[radius=0.5/\d^2/\w^2/\x^2/\y^2/\z^2/\zz^2];
      \draw[dotted] (0, 0) -- +(30:\height);
      \draw[dotted] (\d, 0) -- +(30+120:\height);
      \draw[dotted] (60:\d) -- +(30-120:\height);
    \end{tikzpicture}
  \caption{Horoball diagrams with angle $\theta=\frac{\pi}{6}$}
  \label{fig:anglepi6}
\end{figure}

In the part above, we discussed the first case with $k_\text{max}=2$, $d_2=w$ and a horoball diagram $D$ depicted in Figure \ref{fig:horoedwschief}.
There, the isometry $r$ maps the full-sized horoball $B_k$ to the $(\frac{1}{d_1})$-ball $B_x$ and the $(\frac{1}{d_1})$-ball $B_y$ to the $(\frac{1}{d_2})$-ball $B_q$.

In general, there is a similar pattern as follows.
A $(\frac{1}{d_k})$-ball close to the full-sized ball $B_k$ gets mapped to a $(\frac{1}{d_{k+1}})$-ball closer to the full-sized ball $B_j$.

(a)\quad If there is a unique $(\frac{1}{d_{k_\text{max}}})$-ball in the corresponding part of the horoball diagram $D$, then a $(\frac{1}{d_{k_\text{max}-1}})$-ball on the right gets mapped to this $(\frac{1}{d_{k_\text{max}}})$-ball.

(b)\quad If there are two $(\frac{1}{d_{k_\text{max}}})$-ball in  this part of the
$D$, then a $(\frac{1}{d_{k_\text{max}-1}})$-ball on the right gets mapped to the left $(\frac{1}{d_{k_\text{max}}})$-ball.
This implies that in the second case the right-hand $(\frac{1}{d_{k_\text{max}}})$-ball is {\it fixed}.

In both cases, the reflection $rsr$ then maps a $(\frac{1}{d_k})$-ball on the right to a $(\frac{1}{d_{k+2}})$-ball on the left as long is $k\leq k_\text{max}-2$.
In the case (a), a $(\frac{1}{d_{k_\text{max}-1}})$-ball on the right is fixed by $rsr$ because it is the only ball left.
In the case (b), the $(\frac{1}{d_{k_\text{max}-1}})$-ball on the right has to be mapped to the right $(\frac{1}{d_{k_\text{max}}})$-ball.
Confer to Figure \ref{fig:ballactions} for a sketch of illustrative examples for the cases (a) and (b).
Each pair of horoballs with matching colour and pattern is exchanged by $rsr$.
Those horoballs with a solid fill are fixed.

\begin{figure}
  \centering
  \begin{tikzpicture}[scale=2.5]
    \pgfmathsetmacro{\d}{1.81}
    \pgfmathsetmacro{\height}{\d*sqrt(3)/2}
    \pgfmathsetmacro{\angle}{30}
    \pgfmathsetmacro{\w}{sqrt(\d^2+1/\d^2-2*cos(\angle))}
    \pgfmathsetmacro{\angleb}{acos(0.5*(\w/\d+\d/\w-1/(\w*\d^3)))}
    \pgfmathsetmacro{\x}{\w-1/(\w*\d^2)}
    \pgfmathsetmacro{\y}{sqrt(\d^2+1/\x^2 - 2*\d/\x*cos(30-\angleb))}
    \draw (0, 0) -- (\d, 0) -- (60:\d) -- cycle;
    \draw[fill=gray, fill opacity=0.5] (0, 0) circle[radius=0.5];
    \draw[pattern color=green, pattern=horizontal lines] (\d, 0) circle[radius=0.5];
    \draw (60:\d) circle[radius=0.5];
    \draw ($(0, 0)+(\angle:1/\d)$) circle[radius=0.5/\d^2];
    \draw[pattern color=red, pattern=north east lines] ($(\d, 0)+(\angle+120:1/\d)$) circle[radius=0.5/\d^2];
    \draw ($(60:\d)+(\angle-120:1/\d)$) circle[radius=0.5/\d^2];
    \draw ($(0, 0)+(\angle+\angleb:1/\w)$)
      circle[radius=0.5/\d^2/\w^2];
    \draw[pattern color=blue, pattern=crosshatch dots] ($(\d, 0)+(\angle+\angleb+120:1/\w)$)
      circle[radius=0.5/\d^2/\w^2];
    \draw ($(60:\d)+(\angle+\angleb-120:1/\w)$)
      circle[radius=0.5/\d^2/\w^2];
    \coordinate (wa) at ($(0, 0)+(\angle-\angleb:1/\w)$);
    \coordinate (wa2) at ($(\d, 0)+(\angle+\angleb+120:1/\w)$);
    \coordinate (wb) at ($(\d, 0)+(\angle-\angleb+120:1/\w)$);
    \coordinate (wb2) at ($(60:\d)+(\angle+\angleb-120:1/\w)$);
    \coordinate (wc) at ($(60:\d)+(\angle-\angleb-120:1/\w)$);
    \coordinate (wc2) at ($(0, 0)+(\angle+\angleb:1/\w)$);
    \draw[pattern color=green, pattern=horizontal lines] ($(0, 0)+(\angle-\angleb:1/\w)$)
      circle[radius=0.5/\d^2/\w^2];
    \draw ($(\d, 0)+(\angle-\angleb+120:1/\w)$)
      circle[radius=0.5/\d^2/\w^2];
    \draw ($(60:\d)+(\angle-\angleb-120:1/\w)$)
      circle[radius=0.5/\d^2/\w^2];
    \draw ($(0, 0)+(\angle+\angleb:1/\x)$)
      circle[radius=0.5/\d^2/\w^2/\x^2];
    \draw[pattern color=blue, pattern=crosshatch dots] ($(\d, 0)+(\angle+\angleb+120:1/\x)$)
      circle[radius=0.5/\d^2/\w^2/\x^2];
    \draw ($(60:\d)+(\angle+\angleb-120:1/\x)$)
      circle[radius=0.5/\d^2/\w^2/\x^2];
    \draw[pattern color=red, pattern=north east lines] ($(0, 0)+(\angle-\angleb:1/\x)$)
      circle[radius=0.5/\d^2/\w^2/\x^2];
    \draw ($(\d, 0)+(\angle-\angleb+120:1/\x)$)
      circle[radius=0.5/\d^2/\w^2/\x^2];
    \draw ($(60:\d)+(\angle-\angleb-120:1/\x)$)
      circle[radius=0.5/\d^2/\w^2/\x^2];
    \draw[dotted] (0, 0) -- +(30:\height);
    \draw[dotted] (\d, 0) -- +(30+120:\height);
    \draw[dotted] (60:\d) -- +(30-120:\height);
  \end{tikzpicture}\hfill
  \begin{tikzpicture}[scale=2.5]
    \pgfmathsetmacro{\d}{sqrt(2*cos(36)+sqrt(3))}
    \pgfmathsetmacro{\height}{\d*sqrt(3)/2}
    \pgfmathsetmacro{\angle}{30}
    \pgfmathsetmacro{\w}{sqrt(\d^2+1/\d^2-2*cos(\angle))}
    \pgfmathsetmacro{\angleb}{acos(0.5*(\w/\d+\d/\w-1/(\w*\d^3)))}
    \pgfmathsetmacro{\x}{\w-1/(\w*\d^2)}
    \pgfmathsetmacro{\y}{sqrt(\d^2+1/\x^2 - 2*\d/\x*cos(30-\angleb))}
    \draw (0, 0) -- (\d, 0) -- (60:\d) -- cycle;
    \draw[fill=gray, fill opacity=0.5] (0, 0) circle[radius=0.5];
    \draw[pattern color=green, pattern=horizontal lines] (\d, 0) circle[radius=0.5];
    \draw (60:\d) circle[radius=0.5];
    \draw ($(0, 0)+(\angle:1/\d)$) circle[radius=0.5/\d^2];
    \draw[pattern color=red, pattern=north east lines] ($(\d, 0)+(\angle+120:1/\d)$) circle[radius=0.5/\d^2];
    \draw ($(60:\d)+(\angle-120:1/\d)$) circle[radius=0.5/\d^2];
    \draw ($(0, 0)+(\angle+\angleb:1/\w)$)
      circle[radius=0.5/\d^2/\w^2];
    \draw[pattern color=blue, pattern=crosshatch dots] ($(\d, 0)+(\angle+\angleb+120:1/\w)$)
      circle[radius=0.5/\d^2/\w^2];
    \draw ($(60:\d)+(\angle+\angleb-120:1/\w)$)
      circle[radius=0.5/\d^2/\w^2];
    \coordinate (wa) at ($(0, 0)+(\angle-\angleb:1/\w)$);
    \coordinate (wa2) at ($(\d, 0)+(\angle+\angleb+120:1/\w)$);
    \coordinate (wb) at ($(\d, 0)+(\angle-\angleb+120:1/\w)$);
    \coordinate (wb2) at ($(60:\d)+(\angle+\angleb-120:1/\w)$);
    \coordinate (wc) at ($(60:\d)+(\angle-\angleb-120:1/\w)$);
    \coordinate (wc2) at ($(0, 0)+(\angle+\angleb:1/\w)$);
    \draw[pattern color=green, pattern=horizontal lines] ($(0, 0)+(\angle-\angleb:1/\w)$)
      circle[radius=0.5/\d^2/\w^2];
    \draw ($(\d, 0)+(\angle-\angleb+120:1/\w)$)
      circle[radius=0.5/\d^2/\w^2];
    \draw ($(60:\d)+(\angle-\angleb-120:1/\w)$)
      circle[radius=0.5/\d^2/\w^2];
    \draw ($(0, 0)+(\angle+\angleb:1/\x)$)
      circle[radius=0.5/\d^2/\w^2/\x^2];
    \draw[fill=gray, fill opacity=0.5] ($(\d, 0)+(\angle+\angleb+120:1/\x)$)
      circle[radius=0.5/\d^2/\w^2/\x^2];
    \draw ($(60:\d)+(\angle+\angleb-120:1/\x)$)
      circle[radius=0.5/\d^2/\w^2/\x^2];
    \draw[pattern color=red, pattern=north east lines] ($(0, 0)+(\angle-\angleb:1/\x)$)
      circle[radius=0.5/\d^2/\w^2/\x^2];
    \draw ($(\d, 0)+(\angle-\angleb+120:1/\x)$)
      circle[radius=0.5/\d^2/\w^2/\x^2];
    \draw ($(60:\d)+(\angle-\angleb-120:1/\x)$)
      circle[radius=0.5/\d^2/\w^2/\x^2];
    \draw ($(0, 0)+(\angle+15:1/\y)$)
      circle[radius=0.5/\d^2/\w^2/\x^2/\y^2];
    \draw[pattern color=blue, pattern=crosshatch dots] ($(\d, 0)+(120+\angle+15:1/\y)$)
      circle[radius=0.5/\d^2/\w^2/\x^2/\y^2];
    \draw ($(60:\d)+(-120+\angle+15:1/\y)$)
      circle[radius=0.5/\d^2/\w^2/\x^2/\y^2];
    \draw[dotted] (0, 0) -- +(30:\height);
    \draw[dotted] (\d, 0) -- +(30+120:\height);
    \draw[dotted] (60:\d) -- +(30-120:\height);
  \end{tikzpicture}
  \caption{Action of $rsr$ on some horoballs according to the cases (a) and (b)}
  \label{fig:ballactions}
\end{figure}
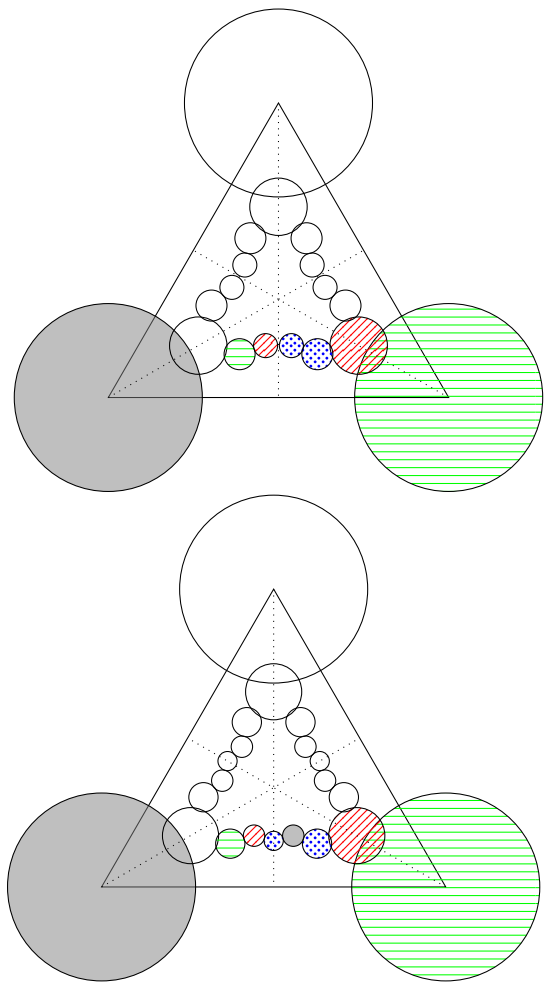

\begin{lem}\label{prop:case-b}
  The case (b) is impossible.
  \label{prop:thingimpossible}
  \begin{proof}
    The previously discussed fact that $rsr$ maps $(\frac{1}{d_k})$-balls to $(\frac{1}{d_{k+2}})$-balls proves that the $(\frac{1}{d_{2k+1}})$-balls are on one line with the $(\frac{1}{d_1})$-balls.
    We already showed that a neighboring pair of a $(\frac{1}{d_1})$-ball and a $(\frac{1}{d_2})$-ball are on a line with a full-sized horoball.
    This implies that the $(\frac{1}{d_{2k}})$-balls are not on the same line as the $(\frac{1}{d_{2k+1}})$-balls.
    However, this means that the right-hand $(\frac{1}{d_{k_\text{max}}})$-ball is not on the same line as the right-hand $(\frac{1}{d_{k_\text{max}-1}})$-ball.
    Thus it cannot be true that they get mapped to one another by $rsr$, proving the statement.
  \end{proof}
\end{lem}

Hence, it remains to discuss the general case (a) with a single $(\frac{1}{d_{k_\text{max}}})$-ball.
It turns out to be helpful to consider the subgroup $\Gamma'$ of $\Gamma$ generated by $\Gamma_\infty$ and $rsr$.
In fact, this move allows us to drop the rotation $r$ and to purge the full-sized horoballs as well as the $(\frac{1}{d_{2k}})$-balls from the horoball diagram $D$.
Then, the $(\frac{1}{d_{2k+1}})$-balls can be increased to yield a  cusp which is a maximal cusp again. More precisely, the $(\frac{1}{d})$-balls become full-sized horoballs (of diameter $\frac{1}{d}$), and then, the horoball diagram $D$ is renormalised in order to have the horosphere $H_{\infty}=\partial B_{\infty}$ at distance $1$ from $\partial U^3$.
Figure \ref{fig:reflsubgrp} depicts the cases with small $k_\text{max}$.

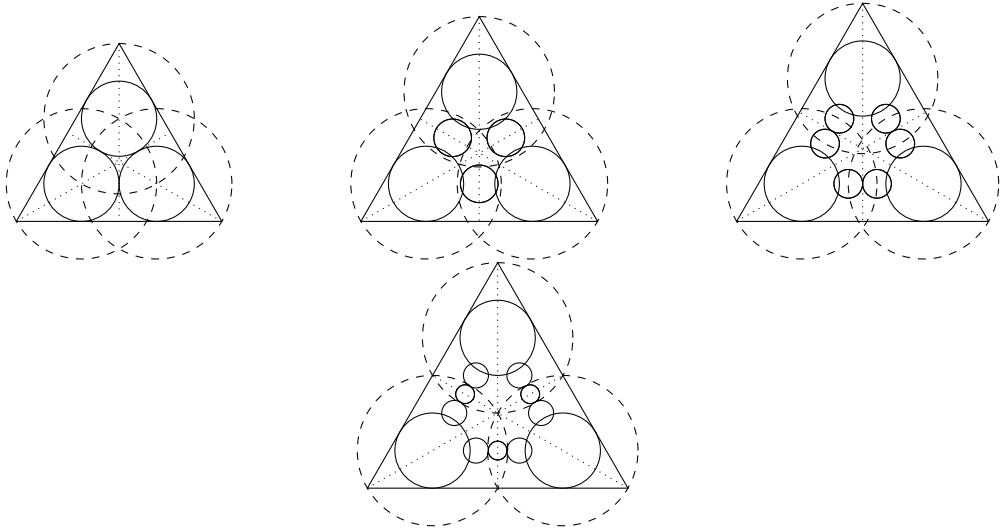
\begin{figure}[ht]
  \centering
  \begin{tikzpicture}
    \pgfmathsetmacro{\nn}{3}
    \pgfmathsetmacro{\d}{2*cos(180/\nn)}
    \pgfmathsetmacro{\height}{\d*sqrt(3)/2}
    \pgfmathsetmacro{\heightb}{\d/sqrt(3)}
    \pgfmathsetmacro{\w}{\d-1/\d}
    \draw (210:\heightb+1) -- (330:\heightb+1) -- (90:\heightb+1) -- cycle;
    \draw (210:\heightb) circle[radius=0.5];
    \draw (330:\heightb) circle[radius=0.5];
    \draw (90:\heightb) circle[radius=0.5];
    \draw[dashed] (210:\heightb) circle[radius=1];
    \draw[dashed] (330:\heightb) circle[radius=1];
    \draw[dashed] (90:\heightb) circle[radius=1];
    \draw[dotted] (210:\heightb+1) -- +(30:\height+1.5);
    \draw[dotted] (330:\heightb+1) -- +(30+120:\height+1.5);
    \draw[dotted] (90:\heightb+1) -- +(30-120:\height+1.5);
  \end{tikzpicture}\hfill
  \begin{tikzpicture}
    \pgfmathsetmacro{\nn}{4}
    \pgfmathsetmacro{\d}{2*cos(180/\nn)}
    \pgfmathsetmacro{\height}{\d*sqrt(3)/2}
    \pgfmathsetmacro{\heightb}{\d/sqrt(3)}
    \pgfmathsetmacro{\w}{\d-1/\d}
    \draw (210:\heightb+1) -- (330:\heightb+1) -- (90:\heightb+1) -- cycle;
    \draw (210:\heightb) circle[radius=0.5];
    \draw (330:\heightb) circle[radius=0.5];
    \draw (90:\heightb) circle[radius=0.5];
    \draw[dashed] (210:\heightb) circle[radius=1];
    \draw[dashed] (330:\heightb) circle[radius=1];
    \draw[dashed] (90:\heightb) circle[radius=1];
    \draw ($(210:\heightb)+(1/\d, 0)$);
    \draw ($(210:\heightb)+(60:1/\d)$);
    \pgfmathsetmacro{\dkglobal}{\d}
    \pgfmathsetmacro{\rkglobal}{1/\d}
    \foreach \k in {1,...,\numexpr (\nn-2)/2\relax}
    {
      \foreach \x in {210, 330, 90}
      {
        \foreach \y in {-150, 150}
        {
          \draw ($(\x:\heightb) + (\x+\y:1/\dkglobal)$)
            circle[radius=0.5*\rkglobal^2];
        }
      }
      \pgfmathparse{\d-1/\dkglobal}
      \global\let\dkglobal=\pgfmathresult
      \pgfmathparse{\rkglobal/\dkglobal}
      \global\let\rkglobal=\pgfmathresult
    }
    \draw[dotted] (210:\heightb+1) -- +(30:\height+1.5);
    \draw[dotted] (330:\heightb+1) -- +(30+120:\height+1.5);
    \draw[dotted] (90:\heightb+1) -- +(30-120:\height+1.5);
  \end{tikzpicture}\hfill
  \begin{tikzpicture}
    \pgfmathsetmacro{\nn}{5}
    \pgfmathsetmacro{\d}{2*cos(180/\nn)}
    \pgfmathsetmacro{\height}{\d*sqrt(3)/2}
    \pgfmathsetmacro{\heightb}{\d/sqrt(3)}
    \pgfmathsetmacro{\w}{\d-1/\d}
    \draw (210:\heightb+1) -- (330:\heightb+1) -- (90:\heightb+1) -- cycle;
    \draw (210:\heightb) circle[radius=0.5];
    \draw (330:\heightb) circle[radius=0.5];
    \draw (90:\heightb) circle[radius=0.5];
    \draw[dashed] (210:\heightb) circle[radius=1];
    \draw[dashed] (330:\heightb) circle[radius=1];
    \draw[dashed] (90:\heightb) circle[radius=1];
    \draw ($(210:\heightb)+(1/\d, 0)$);
    \draw ($(210:\heightb)+(60:1/\d)$);
    \pgfmathsetmacro{\dkglobal}{\d}
    \pgfmathsetmacro{\rkglobal}{1/\d}
    \foreach \k in {1,...,\numexpr (\nn-2)/2\relax}
    {
      \foreach \x in {210, 330, 90}
      {
        \foreach \y in {-150, 150}
        {
          \draw ($(\x:\heightb) + (\x+\y:1/\dkglobal)$)
            circle[radius=0.5*\rkglobal^2];
        }
      }
      \pgfmathparse{\d-1/\dkglobal}
      \global\let\dkglobal=\pgfmathresult
      \pgfmathparse{\rkglobal/\dkglobal}
      \global\let\rkglobal=\pgfmathresult
    }
    \draw[dotted] (210:\heightb+1) -- +(30:\height+1.5);
    \draw[dotted] (330:\heightb+1) -- +(30+120:\height+1.5);
    \draw[dotted] (90:\heightb+1) -- +(30-120:\height+1.5);
  \end{tikzpicture}\hfill
  \begin{tikzpicture}
    \pgfmathsetmacro{\nn}{6}
    \pgfmathsetmacro{\d}{2*cos(180/\nn)}
    \pgfmathsetmacro{\height}{\d*sqrt(3)/2}
    \pgfmathsetmacro{\heightb}{\d/sqrt(3)}
    \pgfmathsetmacro{\w}{\d-1/\d}
    \draw (210:\heightb+1) -- (330:\heightb+1) -- (90:\heightb+1) -- cycle;
    \draw (210:\heightb) circle[radius=0.5];
    \draw (330:\heightb) circle[radius=0.5];
    \draw (90:\heightb) circle[radius=0.5];
    \draw[dashed] (210:\heightb) circle[radius=1];
    \draw[dashed] (330:\heightb) circle[radius=1];
    \draw[dashed] (90:\heightb) circle[radius=1];
    \draw ($(210:\heightb)+(1/\d, 0)$);
    \draw ($(210:\heightb)+(60:1/\d)$);
    \pgfmathsetmacro{\dkglobal}{\d}
    \pgfmathsetmacro{\rkglobal}{1/\d}
    \foreach \k in {1,...,\numexpr (\nn-2)/2\relax}
    {
      \foreach \x in {210, 330, 90}
      {
        \foreach \y in {-150, 150}
        {
          \draw ($(\x:\heightb) + (\x+\y:1/\dkglobal)$)
            circle[radius=0.5*\rkglobal^2];
        }
      }
      \pgfmathparse{\d-1/\dkglobal}
      \global\let\dkglobal=\pgfmathresult
      \pgfmathparse{\rkglobal/\dkglobal}
      \global\let\rkglobal=\pgfmathresult
    }
    \draw[dotted] (210:\heightb+1) -- +(30:\height+1.5);
    \draw[dotted] (330:\heightb+1) -- +(30+120:\height+1.5);
    \draw[dotted] (90:\heightb+1) -- +(30-120:\height+1.5);
  \end{tikzpicture}
  \caption{Horoball diagrams of the reflection subgroups generated by $\Gamma_{\infty}$ and $rsr$}
  \label{fig:reflsubgrp}
\end{figure}

As a consequence, the group $\Gamma'$ is a hyperbolic Coxeter group, and by the same methods as above, one can identify $\Gamma'$ by means
of the Coxeter graph describing a Coxeter polyhedron $P=P(\frac{\pi}{k})$, with an order $2$ internal symmetry plane, according to
\begin{center}
  \begin{tikzpicture}[anchor=base, scale=0.5, baseline,
      bullet/.style={circle, fill=black, minimum size=5pt, inner sep=0pt}]
    \draw (-1, 1) -- (1, 1) node[midway, above]{6};
    \draw (-1,1) node[bullet]{}
      -- (-1, -1) node[bullet]{}
      -- (1, -1) node[bullet]{} node[midway, below]{$k$}
      -- (1, 1) node[bullet]{};
  \end{tikzpicture}
\end{center}
where 
$k$ is an integer with $k\ge3$. For $3\le k\le 6$, the volume of $P(\frac{\pi}{k})$ satisfies $\vol(P(\frac{\pi}{k}))\ge \vol([3^3,6])>2\,v_*\,$ by \eqref{eq:volcycle}.

In the case $k=6$, the fundamental polyhedron $P$ of $\Gamma'$ has four cusps, and for $k>7$, there are two ultraideal vertices.
We do not have to consider those cases because we can easily estimate the orbifold volume using the cusp volume.
If $k=6$, the distance between two full-sized horoballs is $d=\sqrt{3}$ as calculated above in Equation \eqref{eq:dcaseA}.
For this case and for $d\ge \sqrt{3}$, the side length of $D$ is $e=d+\sqrt{3} \ge 2\sqrt{3}$ (compare Figure \ref{fig:volabschaetzungcycle}).
The cusp volume can then be bounded by
\begin{equation}
  \vol(C) = \frac{e^2\sqrt{3}}{48} \ge \frac{\sqrt{3}}{4}>0.43>2\,v_*.
\end{equation}

\begin{figure}
\centering
  \begin{tikzpicture}[scale=3]
    \pgfmathsetmacro{\nn}{6}
    \pgfmathsetmacro{\d}{2*cos(180/\nn)}
    \pgfmathsetmacro{\height}{\d*sqrt(3)/2}
    \pgfmathsetmacro{\heightb}{\d/sqrt(3)}
    \pgfmathsetmacro{\w}{\d-1/\d}
    \draw (210:\heightb+1) -- (330:\heightb+1) -- (90:\heightb+1) -- cycle;
    \draw (210:\heightb) circle[radius=0.5];
    \draw (330:\heightb) circle[radius=0.5];
    \draw (90:\heightb) circle[radius=0.5];
    \draw[dashed] (210:\heightb) circle[radius=1];
    \draw[dashed] (330:\heightb) circle[radius=1];
    \draw[dashed] (90:\heightb) circle[radius=1];
    \draw ($(210:\heightb)+(1/\d, 0)$);
    \draw ($(210:\heightb)+(60:1/\d)$);
    \pgfmathsetmacro{\dkglobal}{\d}
    \pgfmathsetmacro{\rkglobal}{1/\d}
    \foreach \k in {1,...,\numexpr (\nn-2)/2\relax}
    {
      \foreach \x in {210, 330, 90}
      {
        \foreach \y in {-150, 150}
        {
          \draw ($(\x:\heightb) + (\x+\y:1/\dkglobal)$)
            circle[radius=0.5*\rkglobal^2];
        }
      }
      \pgfmathparse{\d-1/\dkglobal}
      \global\let\dkglobal=\pgfmathresult
      \pgfmathparse{\rkglobal/\dkglobal}
      \global\let\rkglobal=\pgfmathresult
    }
    \draw[dotted] (210:\heightb+1) -- +(30:\height+1.5);
    \draw[dotted] (330:\heightb+1) -- +(30+120:\height+1.5);
    \draw[dotted] (90:\heightb+1) -- +(30-120:\height+1.5);
    \pgfmathsetmacro{\eps}{0.1}
    \draw[|-|] ($(210:\heightb+1)+(0, -4*\eps)$) -- ($(330:\heightb+1)+(0, -4*\eps)$) node[midway, below]{$e$};
    \draw[|-|] ($(210:\heightb)+(0, -0.5-\eps)$) -- ($(330:\heightb)+(0, -0.5-\eps)$) node[midway, below]{$d$};
    \draw[|-|] ($(210:\heightb+1)+(0, -\eps)$) -- ($(210:\heightb)+(0, -0.5-\eps)$) node[midway, below]{$\frac{\sqrt{3}}{2}$};
    \draw[|-|] ($(330:\heightb+1)+(0, -\eps)$) -- ($(330:\heightb)+(0, -0.5-\eps)$) node[midway, below]{$\frac{\sqrt{3}}{2}$};
  \end{tikzpicture}
  \caption{Horoball diagram of the reflection subgroup $\Gamma'$ with $k=6$}
  \label{fig:volabschaetzungcycle}
\end{figure}
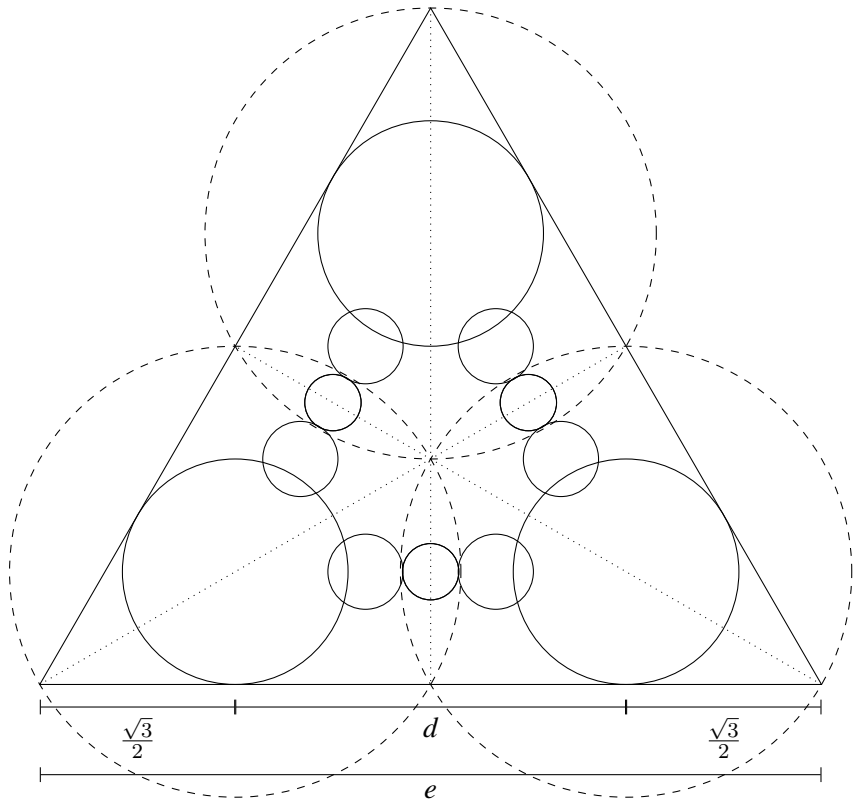
{\it Order 3.\quad} Suppose that a full-sized disk $B=B_1$ is centred at the singular point $b=a_3$ in the barycenter of the triangle $D$. Consider a
full-sized horoball $B_2$ whose center is at distance $d$ from $b$. It follows that the inradius of the cusp triangle $D$ equals
$\frac{d}{2}$ and that the edge length of $D$ equals $\sqrt{3}\,d$. In particular, the cusp volume
$\vol(C)$ is given by $\frac{\sqrt{3}\,d^2}{16}$. In the case that $\Gamma_{\infty}$ is orientation-preserving, the assumption $d>1$ yields the cusp volume estimate
\[
\frac{\sqrt{3}\,d^2}{8}>0.21>v_*\,\,.
\]

Hence, we assume that the cusp triangle $D$ has a mirror symmetry. Since $d>1$, there are three $(\frac{1}{d})$-{balls}, each of diameter $\frac{1}{d^2}$, touching
$B$ in $D$ in a symmetrically arranged way. Denote by $x$ the centre of one of these
$(\frac{1}{d})$-{balls}. The distance $d_0(b,x)$ equals $\frac{1}{d}$.

$\bullet\,\,$ If $x$ lies on the edge $e_{23}$
defined by $b=a_3$ and the midpoint $a_2$ of an edge of $D$, then
the bound
\[
 \frac{d}{2}=d_0(b,a_2)\ge d_0(b,x)=\frac{1}{d}
\]
yields $d\ge\sqrt{2}$ and, hence,
$\vol(C)\ge\frac{\sqrt{3}}{8}\simeq 0.216506$, which is too big in comparison with $v_*=\vol(\mathbb H^3/[5,3,6])\approx0.171502$
(see \eqref{eq:minvol-bound}).

$\bullet\,\,$ If $x\notin e_{23}$, then $x$ lies on the edge segment $e_{36}$
of $\Delta$ defined by $b=a_3$ and a vertex $a_6$ of $D$. Since $d>1$, the center $x$ is
different from $a_6$. Furthermore, it is
easy to see that $d$ and hence $\vol(C)$ take minimal values if the $(\frac{1}{d})$-{ball} $B_x$ is internally tangent to the border of $D$
(see Figure \ref{fig:Figure6}).
\begin{figure}
  \centering
  \def\svgwidth{0.4\textwidth}
\begingroup%
  \makeatletter%
  \providecommand\color[2][]{%
    \errmessage{(Inkscape) Color is used for the text in Inkscape, but the package 'color.sty' is not loaded}%
    \renewcommand\color[2][]{}%
  }%
  \providecommand\transparent[1]{%
    \errmessage{(Inkscape) Transparency is used (non-zero) for the text in Inkscape, but the package 'transparent.sty' is not loaded}%
    \renewcommand\transparent[1]{}%
  }%
  \providecommand\rotatebox[2]{#2}%
  \newcommand*\fsize{\dimexpr\f@size pt\relax}%
  \newcommand*\lineheight[1]{\fontsize{\fsize}{#1\fsize}\selectfont}%
  \ifx\svgwidth\undefined%
    \setlength{\unitlength}{298.01882368bp}%
    \ifx\svgscale\undefined%
      \relax%
    \else%
      \setlength{\unitlength}{\unitlength * \real{\svgscale}}%
    \fi%
  \else%
    \setlength{\unitlength}{\svgwidth}%
  \fi%
  \global\let\svgwidth\undefined%
  \global\let\svgscale\undefined%
  \makeatother%
  \begin{picture}(1,1.64872147)%
    \lineheight{1}%
    \setlength\tabcolsep{0pt}%
    \put(0,0){\includegraphics[width=\unitlength,page=1]{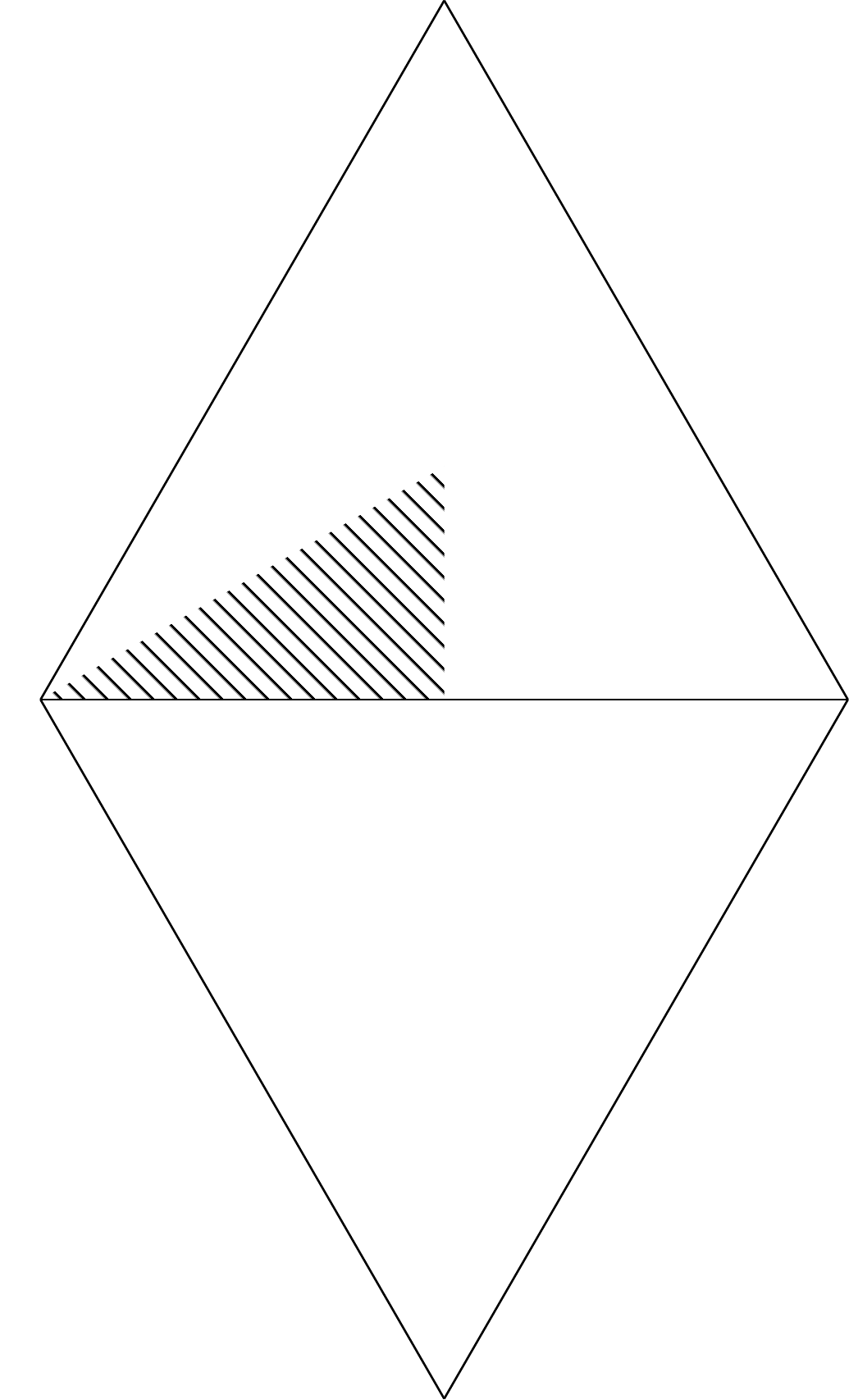}}%
    \put(0.543915,1.08753062){\color[rgb]{0,0,0}\makebox(0,0)[lt]{\lineheight{1.25}\smash{\begin{tabular}[t]{l}$a_3$\end{tabular}}}}%
    \put(0.41161991,1.18652193){\color[rgb]{0,0,0}\makebox(0,0)[lt]{\lineheight{1.25}\smash{\begin{tabular}[t]{l}$B$\end{tabular}}}}%
    \put(0.06329166,0.99850405){\color[rgb]{0,0,0}\makebox(0,0)[lt]{\lineheight{1.25}\smash{\begin{tabular}[t]{l}$D$\end{tabular}}}}%
    \put(0.52876411,0.83804052){\color[rgb]{0,0,0}\makebox(0,0)[lt]{\lineheight{1.25}\smash{\begin{tabular}[t]{l}$a_2$\end{tabular}}}}%
    \put(-0.00304747,0.77354716){\color[rgb]{0,0,0}\makebox(0,0)[lt]{\lineheight{1.25}\smash{\begin{tabular}[t]{l}$a_6$\end{tabular}}}}%
    \put(0.53395047,0.46444377){\color[rgb]{0,0,0}\makebox(0,0)[lt]{\lineheight{1.25}\smash{\begin{tabular}[t]{l}$B_2$\end{tabular}}}}%
    \put(0,0){\includegraphics[width=\unitlength,page=2]{figure6.pdf}}%
  \end{picture}%
\endgroup%

  \caption{A $\{2,3,6\}$-cusp triangle $D$ with a full-sized horoball $B$ centred in $a_3$}
  \label{fig:Figure6}
\end{figure}
In this situation, and since the radius of
$B_x$ equals $\frac{1}{2d^2}$, we deduce the inequality
\[
 d\ge\frac{1}{d^2}+\frac{1}{d}\,\,,
\]
whose solution $d\ge1.324718$ yields the cusp volume bound
\[
\vol(C)\ge0.189971>v_*\,\,.
\]

{\it Order 2.\quad} Suppose that a full-sized horoball $B=B_1$ is centred at the singular point $b_1=a_2$ in $D$, and let
$B_2, B_3$ be full-sized horoballs centred at $b_2,b_3$ in $D$, both at distance $d$ from $b_1$. Then, the edge length
of the cusp triangle $D$ equals $2d$ while the distance from $b_1$ to the centre of $D$ is given by
$r=\frac{d}{\sqrt{3}}$ (representing the inradius of $D$). In particular, we have
$\vol(C)=\frac{\sqrt{3}\,d^2}{12}$, and there are disks of no tangency centred at the vertices of $D$. Hence, we cannot assume that $w\ge1$.

However, the case that $\Gamma_{\infty}$ is orientation-preserving can be excluded since the corresponding cusp volume for $d>1$ immediately yields
the estimate
\[
\vol(C)=\frac{\sqrt{3}\,d^2}{6}>\frac{\sqrt{3}}{6}>v_*\,\,.
\]
Assume that $D$ has a mirror symmetry. We analyse the different tangency possibilities of the $(\frac{1}{d})$-balls with respect
to one or several full-sized horoballs and among themselves.
To this end, the following fact is very useful.

{\it Fact}. The centres of the four $(\frac{1}{d})$-balls touching one full-sized horoball form a rectangle whose diagonals intersect at the angle $\frac{\pi}{3}$.

$\bullet\,\,$ If a single $(\frac{1}{d})$-ball touches three full-sized horoballs in $D$, then it is centred at $a_3$. It follows that $r=\frac{1}{d}$, and hence,  $d=\sqrt[\leftroot{-2}\uproot{2}4]{3}\,$
so that $\vol(C)=\frac{1}{4}>v_*$.

$\bullet\,\,$ If a single $(\frac{1}{d})$-ball touches a pair of neighboring full-sized horoballs but does not touch a third full-sized one, then the smallest cusp volume arises if the centres of the $(\frac{1}{d})$-ball and of the full-sized horoballs which it touches are aligned. This follows easily from the above {\it fact}. We deduce that $d=\sqrt{2}$ and
that $\vol(C)=\frac{\sqrt{3}}{6}>v_*$.

$\bullet\,\,$ Suppose that each $(\frac{1}{d})$-ball touches a unique full-sized horoball. As in the case when a full-sized horoball $B$ is centred at a singular point  $a_6$ of order $6$ in $D$ with distance $w$ from
$a_6$ to the centre $x$ of a $(\frac{1}{d})$-ball not touching it (see Figure \ref{fig:Figure5}), we obtain the identical equation \eqref{eq:3}
for $w$ when the full-sized horoball $B$ is centered at a singular point $a_2$ in $D$ with the condition $\,0\le\theta\le\frac{\pi}{3}$
(see Figure \ref{fig:Figure7}).

There are again several cases to consider.
The case that three $(\frac{1}{d})$-balls are mutually tangent in $D$ can be excluded by Lemma \ref{3balls}.
If two $(\frac{1}{d})$-balls are tangent and do not touch the third one in $D$, then the smallest volume arises if their centres and the centres of the associated full-sized horoballs form an isoscele trapezoid;
see
Figure \ref{fig:Figure7}.
\begin{figure}
  \centering
  \def\svgwidth{0.5\textwidth}
\begingroup%
  \makeatletter%
  \providecommand\color[2][]{%
    \errmessage{(Inkscape) Color is used for the text in Inkscape, but the package 'color.sty' is not loaded}%
    \renewcommand\color[2][]{}%
  }%
  \providecommand\transparent[1]{%
    \errmessage{(Inkscape) Transparency is used (non-zero) for the text in Inkscape, but the package 'transparent.sty' is not loaded}%
    \renewcommand\transparent[1]{}%
  }%
  \providecommand\rotatebox[2]{#2}%
  \newcommand*\fsize{\dimexpr\f@size pt\relax}%
  \newcommand*\lineheight[1]{\fontsize{\fsize}{#1\fsize}\selectfont}%
  \ifx\svgwidth\undefined%
    \setlength{\unitlength}{284.11408871bp}%
    \ifx\svgscale\undefined%
      \relax%
    \else%
      \setlength{\unitlength}{\unitlength * \real{\svgscale}}%
    \fi%
  \else%
    \setlength{\unitlength}{\svgwidth}%
  \fi%
  \global\let\svgwidth\undefined%
  \global\let\svgscale\undefined%
  \makeatother%
  \begin{picture}(1,1.05017757)%
    \lineheight{1}%
    \setlength\tabcolsep{0pt}%
    \put(0,0){\includegraphics[width=\unitlength,page=1]{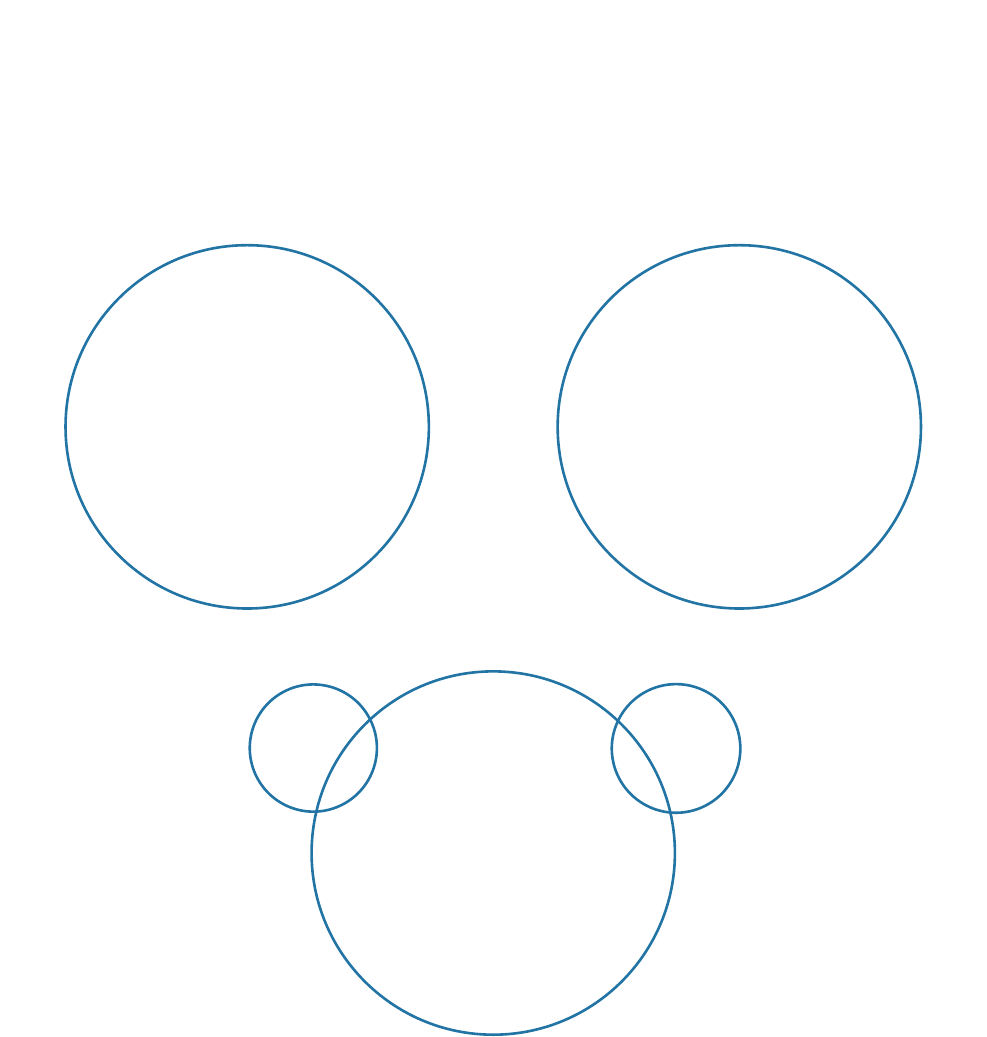}}%
    \put(0.48562864,0.14440558){\color[rgb]{0,0,0}\makebox(0,0)[lt]{\lineheight{1.25}\smash{\begin{tabular}[t]{l}$b_1$\end{tabular}}}}%
    \put(0.75270055,0.63214693){\color[rgb]{0,0,0}\makebox(0,0)[lt]{\lineheight{1.25}\smash{\begin{tabular}[t]{l}$b_2$\end{tabular}}}}%
    \put(0,0){\includegraphics[width=\unitlength,page=2]{figure7_new.pdf}}%
  \end{picture}%
\endgroup%

  \caption{Full-sized horoballs centred at singular points of order $2$}
  \label{fig:Figure7}
\end{figure}
In this case, by Ptolemy's theorem, the square of the diagonal of the trapezoid equals
\[
w^2=\frac{1}{d}+\frac{1}{d^2}\,\,.
\]
By means of equation \eqref{eq:3} with $0\le\theta\le\frac{\pi}{3}$,
we deduce that
 $d^3-d-1\ge0$ and $d\ge 1.32471$. Hence, $\vol(C)>\frac{1}{4}>v_*$.

Suppose that the $(\frac{1}{d})$-balls do not touch each other.
In particular, the minimum distance $\mu$ of their centres is bigger than $\frac{1}{d^2}$. Since $D$ has mirror symmetry,
we get a (possibly degenerate) isoscele trapezoid formed by the centres of two closest $(\frac{1}{d})$-balls in $D$ and the centres of the two full-sized horoballs which they touch. By
Ptolemy's theorem, we deduce that
\begin{equation}\label{eq:ud}
w^2=\mu\,d+\frac{1}{d^2}>\frac{1}{d}+\frac{1}{d^2}\,\,,
\end{equation}
and hence $\vol(C)>\frac{1}{4}>v_*$ in a similar way as above.

As a summary of all the investigations in Section \ref{step3a}, we obtain the following result.
\begin{proposition}\label{prop:case236}
 Let $V$ be a non-arithmetic hyperbolic 3-orbifold with a single cusp $C=B_{\infty}/\Gamma_{\infty}$ which is rigid of type $\{2,3,6\}$.
 Suppose that $\Gamma_{\infty}$ gives rise to only one equivalence class of full-sized horoballs. Then,
 \[
  \hbox{\rm \vol}(V)\ge\hbox{\rm \vol}(V_*)\,\,,
 \]
and equality holds if and only if the orbifold $V$ is isometric to $V_*=\HH^3/[5,3,6]$.
\end{proposition}

\subsubsection{The case \texorpdfstring{$\{2,4,4\}$}{(2,4,4)}}\label{step3b}
(i)\quad As in Section \ref{step3a}, (i), we start by assuming that there are at least two
full-sized horoballs, which are centred at equivalent
singular points and which touch one another in the (square) cusp diagram $D$.
That is, the minimal distance $d$ of the centres of full-sized horoballs
equals 1 (see also \cite[Section 5, p. 11]{Adams1}).

\begin{figure}
  \centering
  \def\svgwidth{0.5\textwidth}
\begingroup%
  \makeatletter%
  \providecommand\color[2][]{%
    \errmessage{(Inkscape) Color is used for the text in Inkscape, but the package 'color.sty' is not loaded}%
    \renewcommand\color[2][]{}%
  }%
  \providecommand\transparent[1]{%
    \errmessage{(Inkscape) Transparency is used (non-zero) for the text in Inkscape, but the package 'transparent.sty' is not loaded}%
    \renewcommand\transparent[1]{}%
  }%
  \providecommand\rotatebox[2]{#2}%
  \newcommand*\fsize{\dimexpr\f@size pt\relax}%
  \newcommand*\lineheight[1]{\fontsize{\fsize}{#1\fsize}\selectfont}%
  \ifx\svgwidth\undefined%
    \setlength{\unitlength}{484.43634442bp}%
    \ifx\svgscale\undefined%
      \relax%
    \else%
      \setlength{\unitlength}{\unitlength * \real{\svgscale}}%
    \fi%
  \else%
    \setlength{\unitlength}{\svgwidth}%
  \fi%
  \global\let\svgwidth\undefined%
  \global\let\svgscale\undefined%
  \makeatother%
  \begin{picture}(1,1.00000002)%
    \lineheight{1}%
    \setlength\tabcolsep{0pt}%
    \put(0,0){\includegraphics[width=\unitlength,page=1]{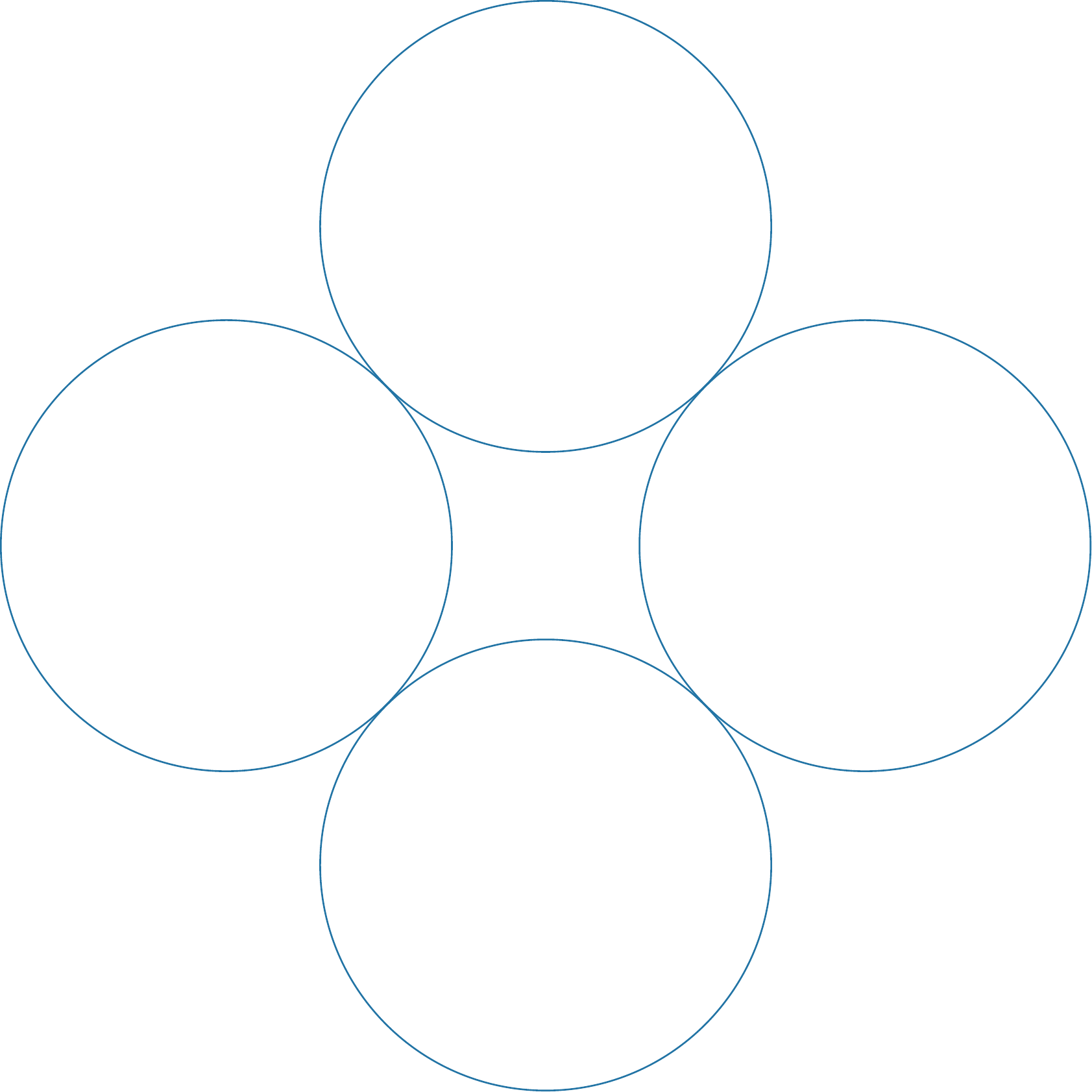}}%
    \put(0.48894149,0.16){\color[rgb]{0,0,0}\makebox(0,0)[lt]{\lineheight{1.25}\smash{\begin{tabular}[t]{l}$a_2$\end{tabular}}}}%
    \put(0.19257349,0.16){\color[rgb]{0,0,0}\makebox(0,0)[lt]{\lineheight{1.25}\smash{\begin{tabular}[t]{l}$a_4$\end{tabular}}}}%
    \put(0,0){\includegraphics[width=\unitlength,page=2]{figure8.pdf}}%
  \end{picture}%
\endgroup%

  \caption{A full-sized horoball centred at $a_2$}
  \label{fig:Figure8}
\end{figure}

{\it Order 4.\quad}Suppose that there is a full-sized horoball centred at one of the two singular points $a_4$ of order 4 in $\Delta$.
Then, $\,d_0(a_2,a_4)=\frac{1}{2}$, and the associated cusp volume $\vol(C)$ equals $\frac{1}{16}$. Similarly to the case of Section
\ref{step3a}, one can check that there is a unique orbifold $V$
related to this cusp configuration. In fact, it is the quotient of $\HH^3$ by the {\it arithmetic} Coxeter group $[3,4,4]$
(see also \cite[p. 13]{Adams1}).  The quotient space $V=\mathbb H^3/[3,4,4]$ has volume
$\frac{\omega_3}{48}=\frac{1}{6}\,\loba(\frac{\pi}{4})\approx0.0763304$ (see Section \ref{volume}).
By \cite[Theorem 5.2 and Theorem 6.1]{Adams1},
we know that the space $\mathbb H^3/[3,4,4]$ has minimal volume among {\it all} hyperbolic 3-orbifolds with at
least one cusp of type $\{2,4,4\}$.
\newline

{\it Order 2.\quad}If there is a full-sized horoball
centred at $a_2\in \Delta$, then we get
$d_0(a_2,a_4)=\frac{1}{\sqrt{2}}$, and
the cusp volume
equals $\frac{1}{8}$ (see \cite[Figure 6(b), p. 12]{Adams1} and Figure \ref{fig:Figure8}).
Again, there is a unique orbifold
corresponding to this configuration. Its fundamental group is given by
the Coxeter group $[4^{1,1},3]$ and therefore commensurable to the arithmetic Coxeter group $[3,4,4]$ (see \cite[p. 130]{JKRT2}).

\smallskip
(ii)\quad Suppose now that the full-sized horoballs do {\it not} touch one another implying that the minimal
distance $d$ between the centres $a_s$ (for $s=2$ or $s=4$) of full-sized horoballs satisfies $d>1$.

Notice that if the stabiliser $\Gamma_{\infty}$ is orientation-preserving,
then the cusp volume equals
\begin{equation}\label{eq:or-pres}
\vol(C)=\frac{d^2}{2\,s}>\frac{1}{2s}\,\,.
\end{equation}
In particular, for $s=2$, we deduce that the cusp volume
yields $\vol(C)>\frac{1}{4}>v_*$.
Moreover, the case $s=4$ can also be excluded as follows. If there is a disk of no tangency, then $d\ge\sqrt{2}$ and $\vol(C)>v_*$. If there is {\it no} disk of no tangency, then by Lemma \ref{no-tangency} we have that the distance $w$ of the centre of a $(\frac{1}{d})$-ball touching $B$ to the centre of a neighboring full-sized horoball satisfies $w\ge1$. The equation \eqref{eq:3} for $w$ remains valid under the adapted constraint $0\le\theta\le\frac{\pi}{4}$ in the square diagram $D$.
The inequality $w\ge1$ now implies that
\[
d^4-(1+\sqrt{2})\,d^2+1\ge0\,\,,
\]
and hence $\,\vol(C)=\frac{d^2}{8}>\frac{1.88}{8}>0.2>v_*$.

{\it Consequence.} We can assume that the cusp diagram $D$ has a mirror symmetry.

As in the case \ref{step3a}, (ii), we study
the $(\frac{1}{d})$-balls and their tangency behavior with respect to the full-sized horoballs
centred at a singular point $a_s$.
Observe that for each full-sized horoball $B$, there are four $(\frac{1}{d})$-balls touching $B$. The angles formed by their centres measured from the center $B$ is a multiple of $\frac{\pi}{4}$. This leads to the
following first cases; see Figure
 \ref{fig:Figure9}.

(a)\quad Suppose that one $(\frac{1}{d})$-ball $B_x$ is touching four full-sized horoballs so that the center $x$ coincides with a
singular point of order $s=4$. By symmetry, we can suppose that $x$ is the center
of the cusp
diagram $D$.

$\bullet\quad$Assume first that the centres of the full-sized horoballs lie at the singular points of order 2 in $D$, that is, $D$ has circumradius $d$.
Since the
distance between the center $x$ and the centre of each of the full-sized horoballs equals $\frac{1}{d}$,
we deduce that $d=\sqrt[\leftroot{-2}\uproot{2}4]{2}$. Hence, for the cusp volume, we get $\vol(C)=\frac{d^2}{8}=\frac{\sqrt{2}}{8}\approx0.176777$
and therefore $\vol(C)>v_*$ (see \eqref{eq:or-pres}).

\begin{figure}
  \centering
  \def\svgwidth{0.4\textwidth}
\begingroup%
  \makeatletter%
  \providecommand\color[2][]{%
    \errmessage{(Inkscape) Color is used for the text in Inkscape, but the package 'color.sty' is not loaded}%
    \renewcommand\color[2][]{}%
  }%
  \providecommand\transparent[1]{%
    \errmessage{(Inkscape) Transparency is used (non-zero) for the text in Inkscape, but the package 'transparent.sty' is not loaded}%
    \renewcommand\transparent[1]{}%
  }%
  \providecommand\rotatebox[2]{#2}%
  \newcommand*\fsize{\dimexpr\f@size pt\relax}%
  \newcommand*\lineheight[1]{\fontsize{\fsize}{#1\fsize}\selectfont}%
  \ifx\svgwidth\undefined%
    \setlength{\unitlength}{452.46172856bp}%
    \ifx\svgscale\undefined%
      \relax%
    \else%
      \setlength{\unitlength}{\unitlength * \real{\svgscale}}%
    \fi%
  \else%
    \setlength{\unitlength}{\svgwidth}%
  \fi%
  \global\let\svgwidth\undefined%
  \global\let\svgscale\undefined%
  \makeatother%
  \begin{picture}(1,1.00000002)%
    \lineheight{1}%
    \setlength\tabcolsep{0pt}%
    \put(0,0){\includegraphics[width=\unitlength,page=1]{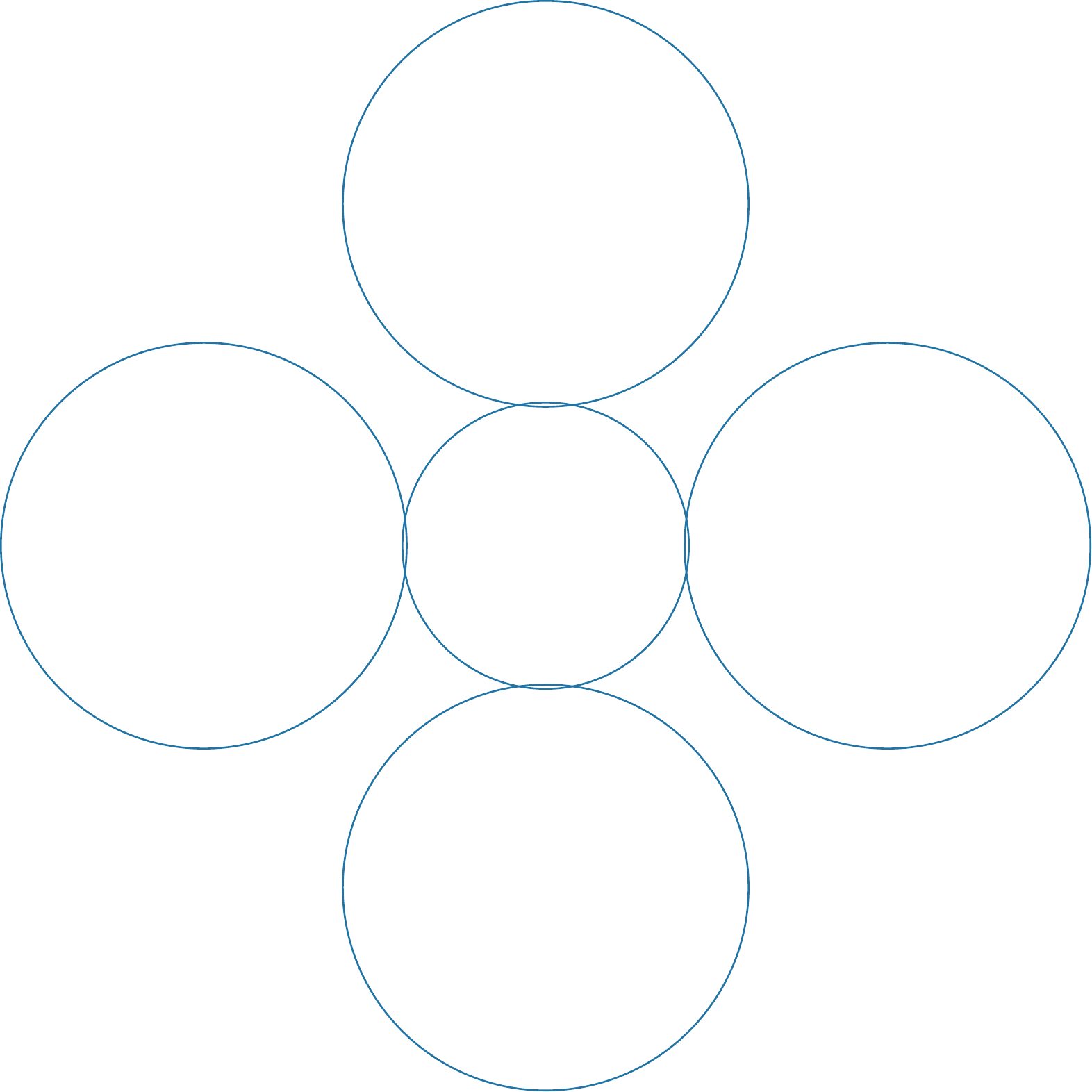}}%
    \put(0.44743044,0.27555292){\color[rgb]{0,0,0}\makebox(0,0)[lt]{\lineheight{1.25}\smash{\begin{tabular}[t]{l}$\frac{1}{d}$\end{tabular}}}}%
    \put(0.69512329,0.4183431){\color[rgb]{0,0,0}\makebox(0,0)[lt]{\lineheight{1.25}\smash{\begin{tabular}[t]{l}$d$\end{tabular}}}}%
    \put(0,0){\includegraphics[width=\unitlength,page=2]{figure9a.pdf}}%
  \end{picture}%
\endgroup%

  \def\svgwidth{0.4\textwidth}
\begingroup%
  \makeatletter%
  \providecommand\color[2][]{%
    \errmessage{(Inkscape) Color is used for the text in Inkscape, but the package 'color.sty' is not loaded}%
    \renewcommand\color[2][]{}%
  }%
  \providecommand\transparent[1]{%
    \errmessage{(Inkscape) Transparency is used (non-zero) for the text in Inkscape, but the package 'transparent.sty' is not loaded}%
    \renewcommand\transparent[1]{}%
  }%
  \providecommand\rotatebox[2]{#2}%
  \newcommand*\fsize{\dimexpr\f@size pt\relax}%
  \newcommand*\lineheight[1]{\fontsize{\fsize}{#1\fsize}\selectfont}%
  \ifx\svgwidth\undefined%
    \setlength{\unitlength}{522.46156156bp}%
    \ifx\svgscale\undefined%
      \relax%
    \else%
      \setlength{\unitlength}{\unitlength * \real{\svgscale}}%
    \fi%
  \else%
    \setlength{\unitlength}{\svgwidth}%
  \fi%
  \global\let\svgwidth\undefined%
  \global\let\svgscale\undefined%
  \makeatother%
  \begin{picture}(1,1.00000002)%
    \lineheight{1}%
    \setlength\tabcolsep{0pt}%
    \put(0,0){\includegraphics[width=\unitlength,page=1]{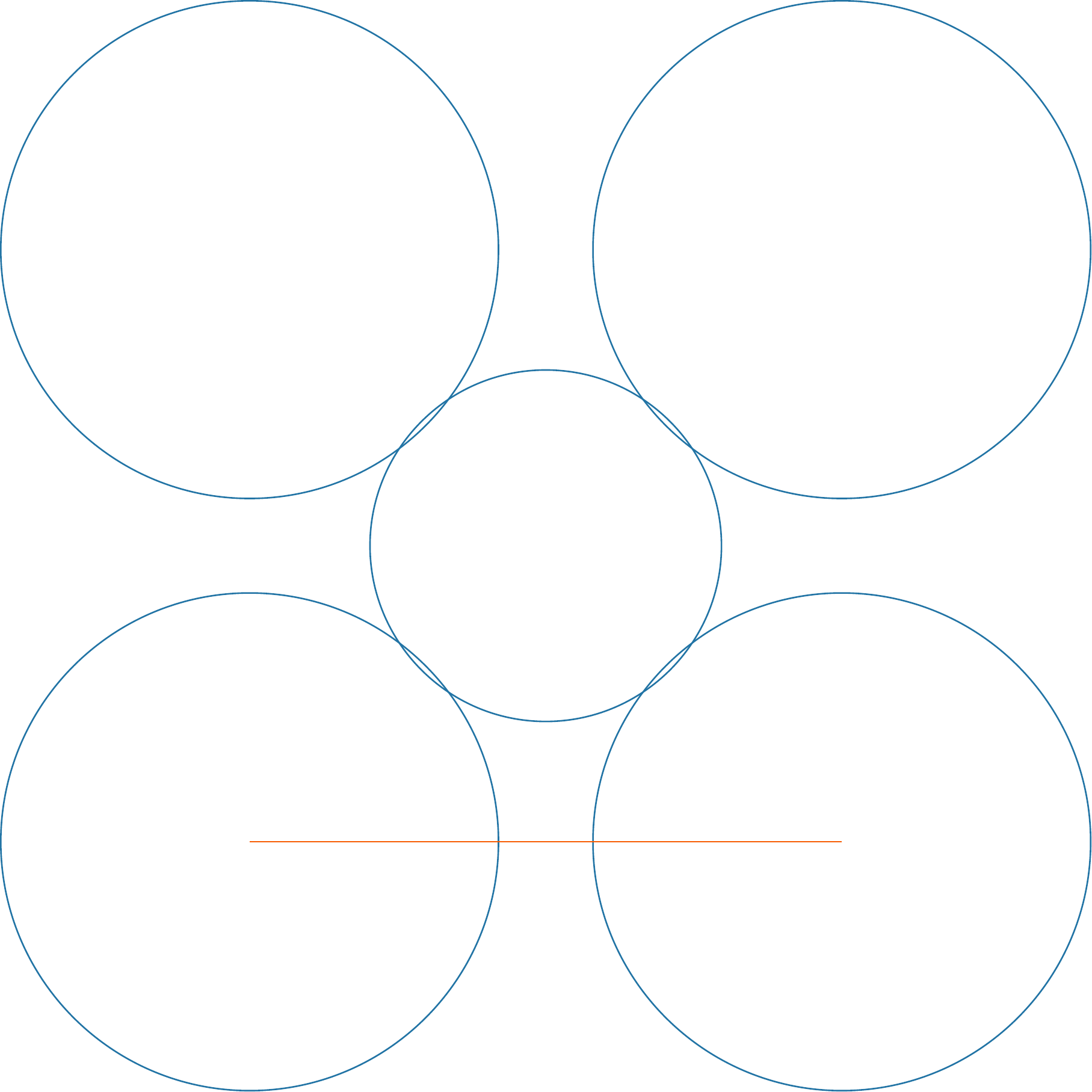}}%
    \put(0.48526919,0.17800527){\color[rgb]{0,0,0}\makebox(0,0)[lt]{\lineheight{1.25}\smash{\begin{tabular}[t]{l}$d$\end{tabular}}}}%
    \put(0,0){\includegraphics[width=\unitlength,page=2]{figure9b.pdf}}%
    \put(0.28324004,0.35106585){\color[rgb]{0,0,0}\makebox(0,0)[lt]{\lineheight{1.25}\smash{\begin{tabular}[t]{l}$\frac{1}{d}$\end{tabular}}}}%
    \put(0,0){\includegraphics[width=\unitlength,page=3]{figure9b.pdf}}%
  \end{picture}%
\endgroup%

  \caption{$(\frac{1}{d})$-balls touching four full-sized balls}  \label{fig:Figure9}
\end{figure}
$\bullet\quad$Assume next that the centres of the full-sized horoballs lie at the vertices (of order 4) of $D$.
Hence, $D$ has circumradius $\frac{1}{d}$.
We deduce that $d=\sqrt[\leftroot{-2}\uproot{2}4]{2}$ as well. However, for the cusp volume,
we obtain $\vol(C)=\frac{\sqrt{2}}{16}$. As shown by \cite[p. 13]{Adams1} in the oriented case, there is a unique orbifold corresponding to
this configuration. It is the arithmetic quotient of an ideal Coxeter tetrahedron with dihedral angles
$\frac{\pi}{4}$, $\frac{\pi}{4}$ and $\frac{\pi}{2}$
by its orientation-preserving symmetry group. Hence, in the non-oriented case, a hyperbolic orbifold with cusp volume
$\frac{\sqrt{2}}{16}$ would have a fundamental group commensurable to the arithmetic Coxeter group $[4^{[4]}]$ and therefore can be excluded from
our consideration.

(b)\quad Suppose that one $(\frac{1}{d})$-ball $B_x$ is touching exactly two full-sized horoballs.

$\bullet\quad$If these
full-sized ones are centred at singular points of order 2 in $D$, then the circumradius of $D$ equals $d$.
Since the center $x$ is aligned with the centres of the full-sized horoballs which $B_x$ touches, we deduce that $d=\sqrt{2}$ and $\vol(C)=\frac{d^2}{8}=\frac{1}{4}>v_*$.

$\bullet\quad$If the two full-sized horoballs are centred at vertices (of order 4) of $D$, then again $d=\sqrt{2}$. It follows that $\vol(C)=\frac{1}{8}$.
Observe that there is a disk of no tangency based at the centre of $D$ and touching the full-sized horoballs.
Given this, the only possible such configuration, and which leads to a singly cusped orbifold, arises by halfening
the Coxeter orthoscheme $[4,4,4]$ with vertices $q,p_1,p_2,p_3$
by means of the plane $P$ passing through the edge $p_2p_3$ and orthogonal to the
(doubly infinite) edge $qp_4$ (see Figure \ref{fig:ortho}). However, a simple computation shows that the horoball sector
$B_{\infty}\cap[4,4,4]$ associated to $q$, has non-empty intersection with $P$ in contradiction
to the
maximality of the cusp.
Hence, this configuration can not be realised.

(c)
Suppose that a $(\frac{1}{d})$-ball 
is touching only one full-sized horoball $B$.
\newline
Observe that, and similarly to the distance $w$, the distance
$v$ of the centre of $B$ to the centre of its neighboring $(\frac{1}{d})$-balls as given by \eqref{eq:2}
remains valid
under the constraint $0\le\theta\le\frac{\pi}{4}$ (the distance $u$ between the centres of two
$(\frac{1}{d})$-balls needs a slight modification, though).

$\bullet\quad$ Assume that $B$ is centred at a singular point of order 2 in $D$. As in the case of the cusp of type $\{2,3,6\}$, and by symmetry,
we get a (possibly degenerate) isoscele trapezoid formed by the centres
of two neighboring $(\frac{1}{d})$-balls at distance $\mu\ge\frac{1}{d^2}$ and by the centres of the full-sized horoballs which they touch.
We obtain the identity \eqref{eq:ud} which, combined with the equation \eqref{eq:3} for $w^2$ and the condition $0\le\theta\le\frac{\pi}{4}$, yields $d^3-\sqrt{2}\,d-1\ge0$ and the estimate
$\,d\ge1.450405>\sqrt{2}$. As a consequence, the volume of the cusp $C$ satisfies $\vol(C)=\frac{d^2}{8}>\frac{1}{4}>v_*$.

$\bullet\quad$ More delicate is the case when $B$ is centred at a singular point of order 4. 
By symmetry, we can suppose that the centre of $B$ coincides with a vertex of $D$.
For the cusp volume, we obtain $\vol(C)=\frac{d^2}{16}$.

As in \cite[Section 5]{Adams3}, we consider the  mutual positions of the $(\frac{1}{d})$-balls associated to the full-sized horoballs. The case that four $(\frac{1}{d})$-balls are closest to one another, that is, each one is tangent
to two other ones and they all are symmetrically arranged around the centre of $D$, can be excluded by applying verbatim Adams' corresponding argument.

Suppose next that two $(\frac{1}{d})$-balls are tangent, and no other $(\frac{1}{d})$-ball is tangent to them. Since $D$ has a mirror symmetry, this can only happen if they are tangent at a singular point $a_2$ of order 2 with centres
being aligned with the centres of the corresponding full-sized horoballs. We deduce that $d=\frac{2}{d}+\frac{1}{d^2}$ and hence $d=2\cos\frac{\pi}{5}$. It follows that
\[
\frac{\vol(C)}{d_3(\infty)}>0.19>v_*\,\,.
\]
It remains to investigate the case when the $(\frac{1}{d})$-balls do not touch. As in the analogous case of a cusp of type $\{2,3,6\}$, we consider the $(\frac{1}{w})$-balls with their position relative to the $(\frac{1}{d})$-balls. Assume first that four $(\frac{1}{w})$-balls coincide and, hence, have  center equal to the centre of the square $D$. Then, the circumradius of $D$ equals $\frac{1}{w}$ so that $w=\frac{\sqrt{2}}{d}$. By means of \eqref{eq:3}, one can compute that $d=\sqrt[\leftroot{-2}\uproot{2}4]{5}$ and $\cos\theta=\frac{2}{\sqrt{5}}$; see Figure \ref{fig:einsdurchw244}. Hence, the cusp diagram $D$ has no mirror symmetry, and the corresponding
cusp volume yields $\vol(C)=\frac{\sqrt{5}}{8}>\frac{1}{4}>v_*$.
By performing analogous computations as in the case of a type $\{2,3,6\}$-cusp, we can show furthermore that there is a unique corresponding oriented orbifold that is arithmetic. For more details, see Appendix \ref{subsec:appendix2}.

The case that two of the $(\frac{1}{w})$-balls coincide  in the interior of $D$ can be excluded by symmetry.

Hence, it remains to investigate the situation when all $(\frac{1}{w})$-balls are distinct. In order to finish this case, it is sufficient to check the following two extremal possibilities for a $(\frac{1}{w})$-ball touching two $(\frac{1}{d})$-balls associated to two different full-sized balls.

$\bullet\quad$ Suppose first that a $(\frac{1}{w})$-ball touches two $(\frac{1}{d})$-balls in such a way that their centres are aligned on an edge of $D$. Then, $\frac{1}{w}=\frac{d}{2}$ which by means of \eqref{eq:wd} and $\theta=0$ implies that
$d^4-2\,d^2-3=0$ with solution $d=\sqrt{3}$. Hence, $\vol(C)=\frac{3}{16}>v_*$.

$\bullet\quad$ Suppose next that a $(\frac{1}{w})$-ball touches two $(\frac{1}{d})$-balls whose centres lie on the diagonals
of $D$. Notice that there are eight $(\frac{1}{w})$-balls around each full-sized horoball.
Then, one can show in a similar way as in the case of a cusp of type $\{2,3,6\}$ that $d=\sqrt{1+\sqrt{2}}$. It follows that $\frac{\vol(C)}{d_3(\infty)}=\frac{1+\sqrt{2}}{16\,d_3(\infty)}>0.176>v_*$.

Furthermore, we can describe the associated 1-cusped orbifold in an analogous way as in the case of a cusp of type $\{2,3,6\}$ and the explanations referring to Figure \ref{fig:horoedwschief}. The orbifold has a non-arithmetic fundamental group and is built upon a polyhedron arising from halfening the Coxeter tetrahedron with cyclic Coxeter symbol $[(4^3,3)]$ and Coxeter graph with internal symmetry as given by
\vskip-.5cm
\begin{center}
  \begin{tikzpicture}[anchor=base, scale=0.5, baseline,
      bullet/.style={circle, fill=black, minimum size=5pt, inner sep=0pt}]
    \draw (-1, 1) -- (1, 1) node[midway, above]{$4$};
    \draw (-1,1) node[bullet]{}
      -- (-1, -1) node[bullet]{}
      -- (1, -1) node[bullet]{}
      -- (1, 1) node[bullet]{};%
  \node[text width=1cm] at (-.7,-.2) {$4$};
    \node[text width=1cm] at (2.2,-.2) {$4\quad.$};
  \end{tikzpicture}
\end{center}
\medskip
As a result, and by \cite[p. 348]{JKRT1}, the volume of the orbifold is about $0.27814$ and hence too large.

Summarising the investigations in Section \ref{step3b}, we can state the following result.
\begin{proposition}\label{prop:case244}
 Let $V$ be a non-arithmetic hyperbolic 3-orbifold with a single cusp $C=B_{\infty}/\Gamma_{\infty}$ which is rigid of type $\{2,4,4\}$.
 Suppose that $\Gamma_{\infty}$ gives rise to only one equivalence class of full-sized horoballs. Then,
 \[
  \hbox{\rm \vol}(V)>\hbox{\rm \vol}(V_*)\,\,.
 \]
\end{proposition}
\subsection{More than one equivalence class of full-sized horoballs}\label{step4}
Assume that the crystallographic group $\Gamma_{\infty}$  gives rise to more than one equivalence class of full-sized horoballs. We show that the corresponding orbifold has volume strictly bigger than $v_*=\vol(\mathbb H^3/[5,3,6])$ (if it exists) by treating the cases of a cusp $C$ of type $\{2,3,6\}$ or of type $\{2,4,4\}$ separately. For simplicity, we assume that the group $\Gamma$ is {\it orientation-preserving} by passing to its orientation-preserving subgroup of index 2 if necessary.
\subsubsection{The case $\{2,3,6\}$}\label{subsubsec:1}
As in the case of one equivalence class,
inequivalent horoballs have to be centred in the singular points $a_s\,,\,s\in\{2,3,6\}$ of the cusp diagram $D$ because otherwise already the cusp volume becomes too big (see Section \ref{step3} and \eqref{eq:no-singular}).

Suppose first that there are three equivalence classes
of full-sized horoballs with respect to the action of $\Gamma_{\infty}$. The smallest volume then arises if $d_0(a_2,a_3)=1$.
It follows that the (oriented) cusp volume
satisfies $\vol(C)=\frac{\sqrt{3}}{2}>2\,v_*$, and that
the orbifold volume is too big.

Consider the case of exactly two equivalence classes of full-sized horoballs.
Continue to denote by $\tau$ the shortest translation length in $\Gamma_\infty$ and by $d$ the shortest distance between two equivalent full-sized horoballs.
In the case where the full-sized horoballs are centred in singular points of order 2 and 3, we get $\tau\ge2$ which leads to a cusp volume of at least $\vol(C)\geq \frac{\sqrt{3}}{2}>2\,v_*$.

Using the fact that the full-sized horoballs are centred in singular points $a_s$, we can deduce the following slightly generalised version of Adams' Lemma \ref{elliptic} in the case of a (oriented) cusp $C$ of type $\{2,3,6\}$.
\begin{lem}\label{lem:hororotation}
Suppose that the full-sized horoballs of a $\{2,3,6\}$-cusp are centred in singular points $a_s$ in the cusp diagram $D$. Then, there is a rotation of order 2 in $\Gamma$ whose axis is tangent to a full-sized horoball and the horoball $B_{\infty}$.\end{lem}
  \begin{proof}
    Take a full-sized horoball centred at a singular point $a_s=(u,1)$ where $u\in\mathbb R^2$ belongs to the boundary $\partial U^3$. Since the group $\Gamma$ acts transitively on the set of horoballs covering $C$, there is an isometry $\gamma\in\Gamma$ mapping $u$ to $\infty$.
    Since the axis $l_s$ formed by $u$ and $\infty$ is the rotational axis of an isometry in $\Gamma_{\infty}$, the image geodesic $\gamma(l_s)$ is also the axis of a rotation of the same order.
  Thus, by composing $\gamma$ with a suitable isometry in $\Gamma_\infty$, we can assume that $\gamma$ maps $\infty$ back to $u$. It follows that the element $\gamma$ has to be a rotation of order $2$ in $\Gamma$ whose axis is tangent to the full-sized horoball and $B_{\infty}$.
  \end{proof}

There are two cases to consider for a possible placement of the centres of the full-sized horoballs.

{\it Centres at $a_6$ and $a_3$.\quad}
Assume that the full-sized horoballs are centred in singular points $a_s$ of order $6$ and $3$, respectively.
Denote by $e=\frac{\tau}{\sqrt{3}}$ the distance between $a_6$ and $a_3$ in the cusp diagram $D$.

In the minimal case, $e=1$; see Figure \ref{fig:236-6and2-touching}.
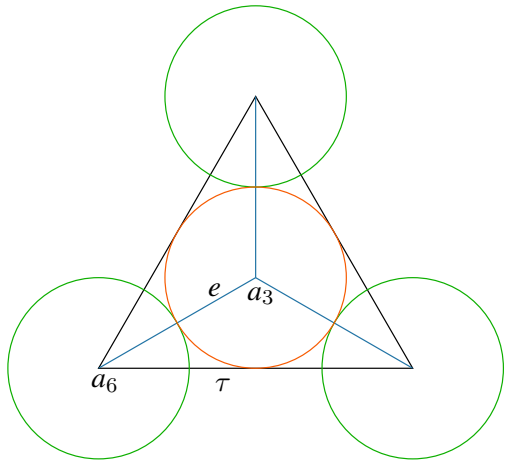
\begin{figure}[h]
  \centering
  \def\svgwidth{0.5\textwidth}
\begingroup%
  \makeatletter%
  \providecommand\color[2][]{%
    \errmessage{(Inkscape) Color is used for the text in Inkscape, but the package 'color.sty' is not loaded}%
    \renewcommand\color[2][]{}%
  }%
  \providecommand\transparent[1]{%
    \errmessage{(Inkscape) Transparency is used (non-zero) for the text in Inkscape, but the package 'transparent.sty' is not loaded}%
    \renewcommand\transparent[1]{}%
  }%
  \providecommand\rotatebox[2]{#2}%
  \newcommand*\fsize{\dimexpr\f@size pt\relax}%
  \newcommand*\lineheight[1]{\fontsize{\fsize}{#1\fsize}\selectfont}%
  \ifx\svgwidth\undefined%
    \setlength{\unitlength}{582.24708137bp}%
    \ifx\svgscale\undefined%
      \relax%
    \else%
      \setlength{\unitlength}{\unitlength * \real{\svgscale}}%
    \fi%
  \else%
    \setlength{\unitlength}{\svgwidth}%
  \fi%
  \global\let\svgwidth\undefined%
  \global\let\svgscale\undefined%
  \makeatother%
  \begin{picture}(1,0.91527017)%
    \lineheight{1}%
    \setlength\tabcolsep{0pt}%
    \put(0,0){\includegraphics[width=\unitlength,page=1]{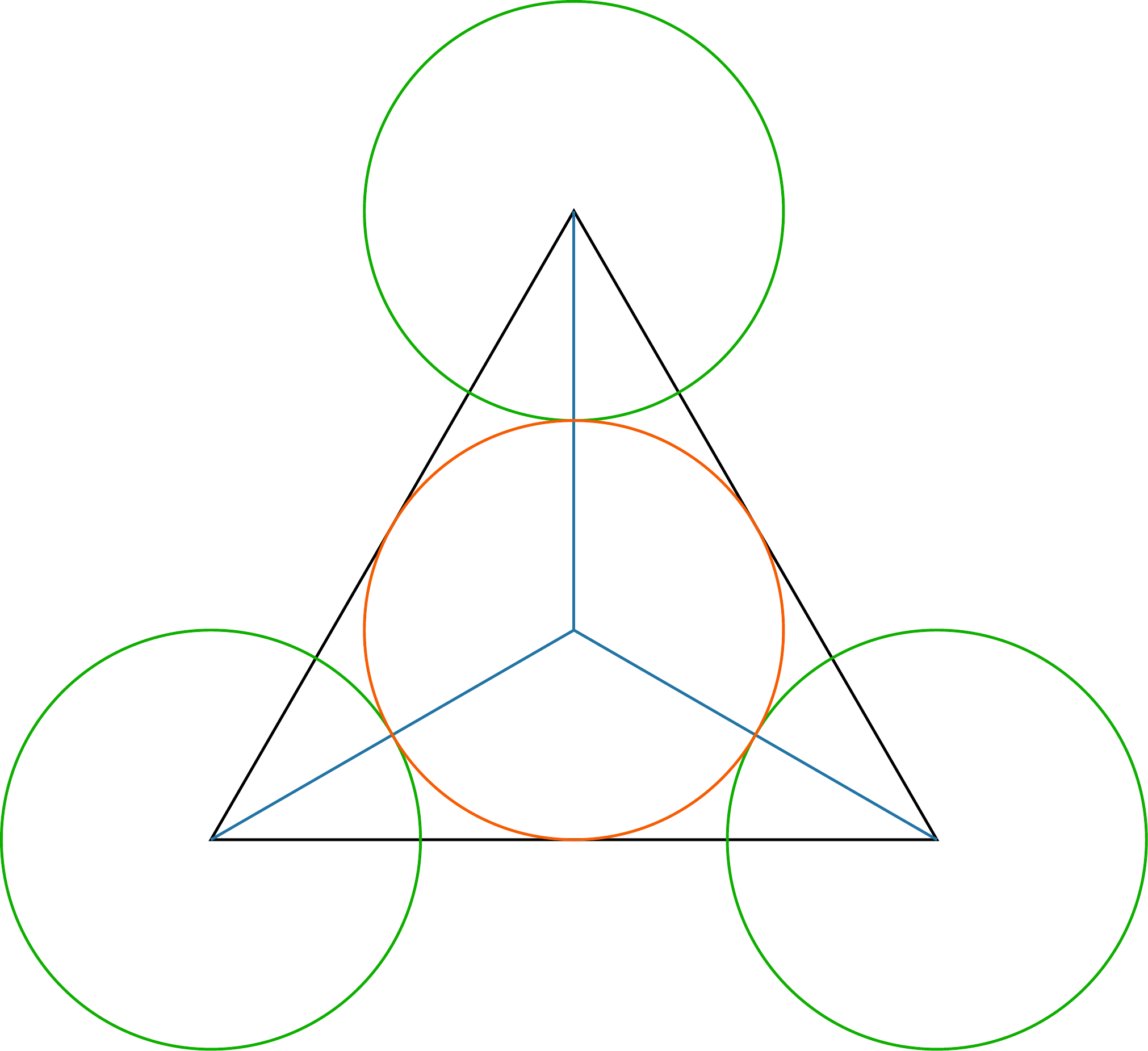}}%
    \put(0.4,0.33){\color[rgb]{0,0,0}\makebox(0,0)[lt]{\lineheight{1.25}\smash{\begin{tabular}[t]{l}$e$\end{tabular}}}}%
    \put(0.41393127,0.13797274){\color[rgb]{0,0,0}\makebox(0,0)[lt]{\lineheight{1.25}\smash{\begin{tabular}[t]{l}$\tau$\end{tabular}}}}%
    \put(0.16552068,0.1451872){\color[rgb]{0,0,0}\makebox(0,0)[lt]{\lineheight{1.25}\smash{\begin{tabular}[t]{l}$a_6$\end{tabular}}}}%
    \put(0.47939045,0.32161495){\color[rgb]{0,0,0}\makebox(0,0)[lt]{\lineheight{1.25}\smash{\begin{tabular}[t]{l}$a_3$\end{tabular}}}}%
  \end{picture}%
\endgroup%

  \caption{Inequivalent full-sized horoballs touch one another}
  \label{fig:236-6and2-touching}
\end{figure}
By Lemma \ref{lem:hororotation}, we have the existence of rotations of order 2 which exchange the full-sized horoballs centred at $a_6=(u_6,1)$ and at $a_3=(u_3,1)$ with the horoball $B_\infty$.
It is not difficult to see that the rotation which sends $\infty$ to $u_3$ and fixes $u_6$, together with the rotations in the singular axes of order 2, give a face identification of the ideal regular tetrahedron $S_\text{reg}^\infty$ with ideal vertices defined by the 3 vertices of the regular triangle $D$ and $\infty$.
One can also see that the above rotations are generated by reflections in the sides of the characteristic orthoscheme $[3,3,6]$ associated to $S_\text{reg}^\infty$.
That implies that these rotations generate an arithmetic subgroup of finite index in $\Gamma$ so that $\Gamma$ itself is arithmetic.
The rotations around the singular axis of order 3 and the rotation of order 2 mapping $u_6$ to $\infty$ generate $6$ isometries of $S_\text{reg}^\infty$.
As a by-product, we can calculate the volume by means of \eqref{eq:vol336} and obtain
\begin{equation*}
  \vol\left(\mathbb H^3/\Gamma\right) = \frac{1}{6}\,\vol\left(S_\text{reg}^\infty\right) = \frac{1}{2}\,\loba(\frac{\pi}{3}).
\end{equation*}

In the case $e>1$, there is a similar notion to a $(\frac{1}{d})$-ball in the case of one equivalence class.
More precisely, the rotations of Lemma \ref{lem:hororotation} map full-sized horoballs at distance $e$ from a full-sized horoball to horoballs of diameter $\frac{1}{e^2}$ at distance $\frac{1}{e}$ from the full-sized image of $B_{\infty}$.
Accordingly, we call these balls {\it $(\frac{1}{e})$-balls}.
The general situation is depicted in Figure \ref{fig:2equi-ebig}.

\begin{figure}[h]
  \centering
  \def\svgwidth{0.5\textwidth}
\begingroup%
  \makeatletter%
  \providecommand\color[2][]{%
    \errmessage{(Inkscape) Color is used for the text in Inkscape, but the package 'color.sty' is not loaded}%
    \renewcommand\color[2][]{}%
  }%
  \providecommand\transparent[1]{%
    \errmessage{(Inkscape) Transparency is used (non-zero) for the text in Inkscape, but the package 'transparent.sty' is not loaded}%
    \renewcommand\transparent[1]{}%
  }%
  \providecommand\rotatebox[2]{#2}%
  \newcommand*\fsize{\dimexpr\f@size pt\relax}%
  \newcommand*\lineheight[1]{\fontsize{\fsize}{#1\fsize}\selectfont}%
  \ifx\svgwidth\undefined%
    \setlength{\unitlength}{495.15604864bp}%
    \ifx\svgscale\undefined%
      \relax%
    \else%
      \setlength{\unitlength}{\unitlength * \real{\svgscale}}%
    \fi%
  \else%
    \setlength{\unitlength}{\svgwidth}%
  \fi%
  \global\let\svgwidth\undefined%
  \global\let\svgscale\undefined%
  \makeatother%
  \begin{picture}(1,0.90036771)%
    \lineheight{1}%
    \setlength\tabcolsep{0pt}%
    \put(0,0){\includegraphics[width=\unitlength,page=1]{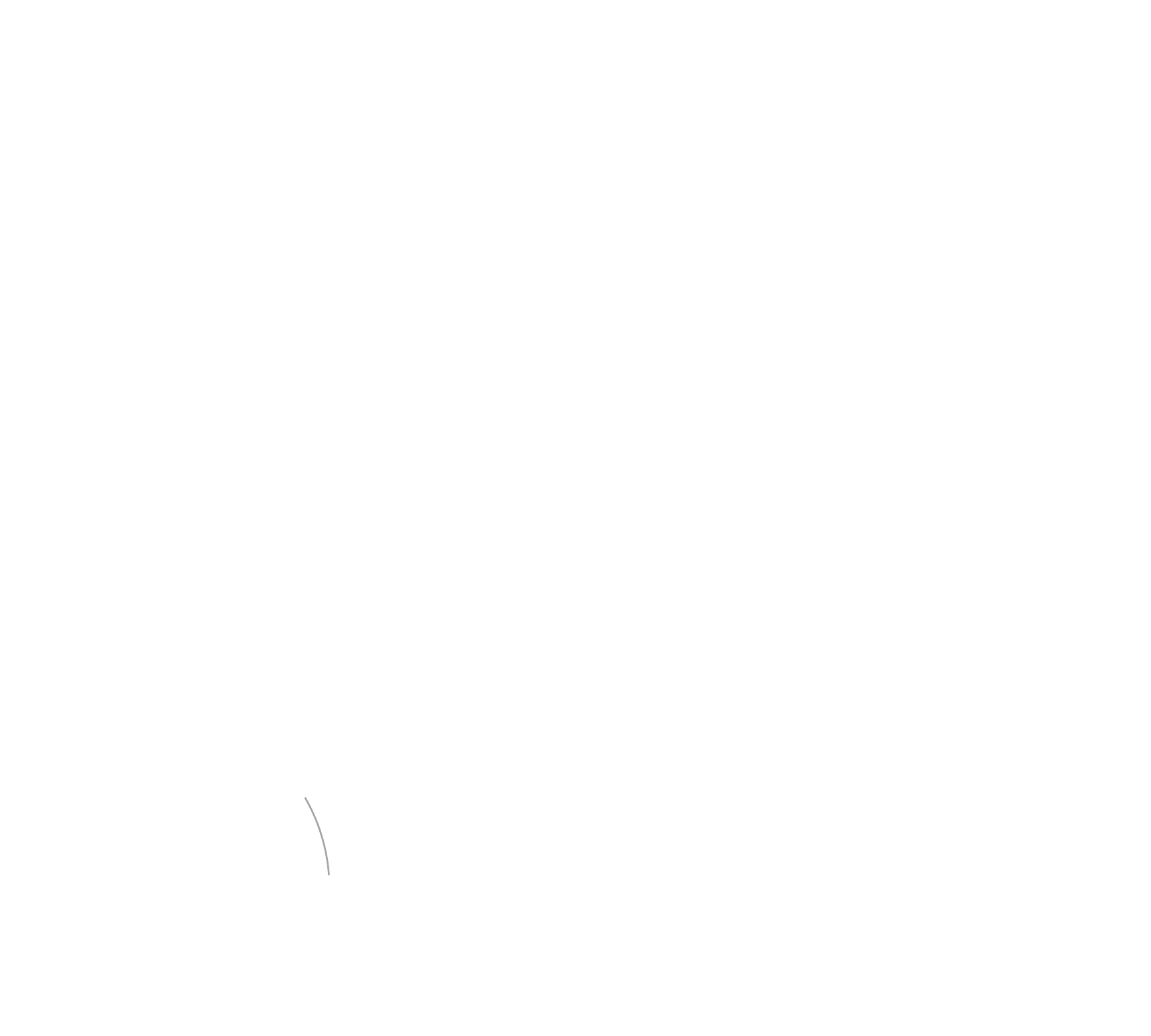}}%
    \put(0.285,0.157){\color[rgb]{0,0,0}\makebox(0,0)[lt]{\lineheight{1.25}\smash{\begin{tabular}[t]{l}$\alpha$\end{tabular}}}}%
    \put(0,0){\includegraphics[width=\unitlength,page=2]{236-2equiclasses-general.pdf}}%
    \put(0.26734577,0.22951759){\color[rgb]{0,0,0}\makebox(0,0)[lt]{\lineheight{1.25}\smash{\begin{tabular}[t]{l}$e$\end{tabular}}}}%
    \put(0.16773948,0.06886219){\color[rgb]{0,0,0}\makebox(0,0)[lt]{\lineheight{1.25}\smash{\begin{tabular}[t]{l}$\frac{1}{e}$\end{tabular}}}}%
  \end{picture}%
\endgroup%

  \caption{Inequivalent full-sized horoballs at distance $e>1$}
  \label{fig:2equi-ebig}
\end{figure}

The minimal distance $e$ is obtained by decreasing $e$ to the point where the $(\frac{1}{e})$-balls touch simultaneously both inequivalent full-sized horoballs; see Figure \ref{fig:2equi-touching}.
For a given angle $\alpha\in[0,\frac{\pi}{6}]$ as defined in Figure \ref{fig:2equi-ebig}, one can calculate the minimal possible $e$ by using the isosceles triangle with base $e$,  opposite angle $\alpha$ and sides $\frac{1}{e}$.
This implies that $\cos\alpha=\frac{e^2}{2}$ and that
\begin{equation}
  e = \sqrt{2\cos\alpha}\ge\sqrt{2\cos\frac{\pi}{6}}=\sqrt[4]{3}.
  \label{eq:1de236}
\end{equation}

\begin{figure}[h]
  \centering
  \def\svgwidth{0.5\textwidth}
\begingroup%
  \makeatletter%
  \providecommand\color[2][]{%
    \errmessage{(Inkscape) Color is used for the text in Inkscape, but the package 'color.sty' is not loaded}%
    \renewcommand\color[2][]{}%
  }%
  \providecommand\transparent[1]{%
    \errmessage{(Inkscape) Transparency is used (non-zero) for the text in Inkscape, but the package 'transparent.sty' is not loaded}%
    \renewcommand\transparent[1]{}%
  }%
  \providecommand\rotatebox[2]{#2}%
  \newcommand*\fsize{\dimexpr\f@size pt\relax}%
  \newcommand*\lineheight[1]{\fontsize{\fsize}{#1\fsize}\selectfont}%
  \ifx\svgwidth\undefined%
    \setlength{\unitlength}{526.88730003bp}%
    \ifx\svgscale\undefined%
      \relax%
    \else%
      \setlength{\unitlength}{\unitlength * \real{\svgscale}}%
    \fi%
  \else%
    \setlength{\unitlength}{\svgwidth}%
  \fi%
  \global\let\svgwidth\undefined%
  \global\let\svgscale\undefined%
  \makeatother%
  \begin{picture}(1,0.90636788)%
    \lineheight{1}%
    \setlength\tabcolsep{0pt}%
    \put(0,0){\includegraphics[width=\unitlength,page=1]{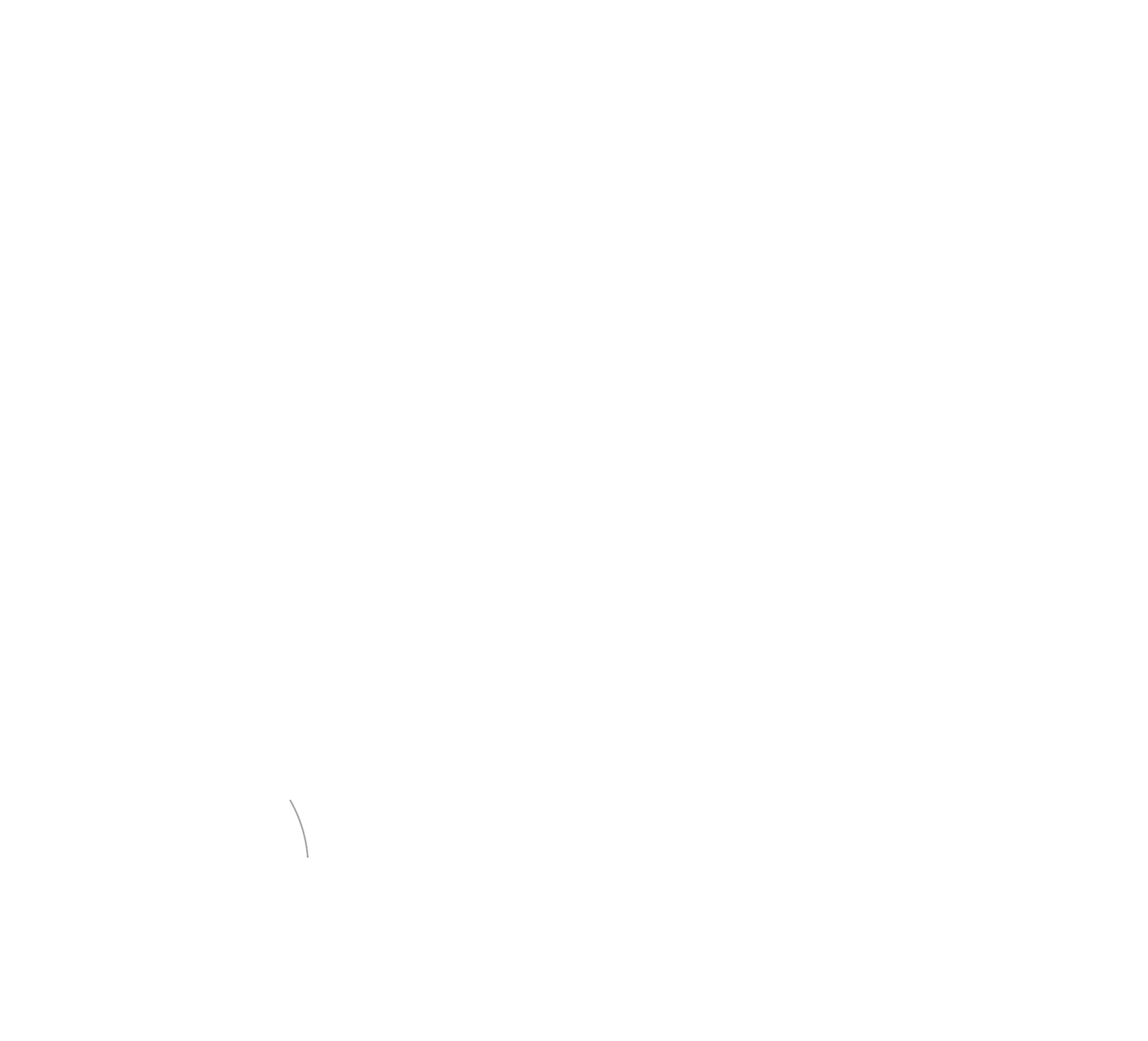}}%
    \put(0.22,0.16604947){\color[rgb]{0,0,0}\makebox(0,0)[lt]{\lineheight{1.25}\smash{\begin{tabular}[t]{l}$\alpha$\end{tabular}}}}%
    \put(0,0){\includegraphics[width=\unitlength,page=2]{236-2equiclasses-touching.pdf}}%
    \put(0.38495963,0.318){\color[rgb]{0,0,0}\makebox(0,0)[lt]{\lineheight{1.25}\smash{\begin{tabular}[t]{l}$e$\end{tabular}}}}%
    \put(0.33374682,0.198){\color[rgb]{0,0,0}\makebox(0,0)[lt]{\lineheight{1.25}\smash{\begin{tabular}[t]{l}$\frac{1}{e}$\end{tabular}}}}%
  \end{picture}%
\endgroup%

  \caption{A $(\frac{1}{e})$-ball touching two inequivalent full-sized horoballs}
  \label{fig:2equi-touching}
\end{figure}

The lower bound on $e$ yields a lower bound for the (oriented) cusp volume. It follows that
\begin{equation*}
  \vol(C) = \frac{\tau e}{8} = \frac{\sqrt{3}\,e^2}{8}\ge\frac{3}{8}>2\,v_*\,\,,
\end{equation*}
which is enough to exclude this case this from our considerations.

{\it Centres at $a_6$ and $a_2$.}\quad
Assume that the full-sized horoballs are centred in singular points $a_6$ and $a_2$, of order $6$ and $2$, respectively.
Denote by $e\ge1$ the shortest distance between $a_6$ and $a_2$ in the diagram $D$.
The next result shows that $e>1$.
\begin{proposition}\label{prop:e-big}
  Suppose that a 1-cusped oriented hyperbolic $3$-orbifold has a cusp of type $\{2,3,6\}$ with precisely two equivalence classes of full-sized horoballs centred in the singular points $a_6$ and $a_2$ of the horoball diagram $D$. Then,  $d_0(a_6,a_2)>1$, and the full-sized horoballs cannot touch each other.
\end{proposition}

  \begin{proof}
   On the contrary, assume there is such an orbifold with touching full-sized horoballs giving rise to a horoball configuration as in Figure \ref{fig:236-2and6}.
    By Lemma \ref{lem:hororotation}, there is a rotation of order 2 exchanging the horoball $B_1$ centred in the singular point $b_1=(u_1,1)$ of order 6 and the horoball $B_\infty$ at infinity.
    This rotation sends the touching full-sized horoballs around $B_1$ onto each other.
    After possibly using a rotation around the axis $l_1=(u_1,\infty)$, we can assume that there is such a rotation $r$ exchanging $B_1$ and $B_\infty$ while fixing $B_2$ and $B_2'$ centred in $b_2=(u_2,1)$ and $b_2'=(u_2',1)$, respectively.
    This implies that the image geodesic $(u_1,u_2)$
    of the order 2 axis $l_2=(u_2,\infty)$ is also a rotation axis of order 2.

    Using Lemma \ref{lem:hororotation} again, there has to be a rotation $r'$ around a common tangent of $B_2$ and $B_\infty$.
    Denote by $s$ the rotation of order 2 around the axis $l_2$.
    Then $sr'$ is also a rotation through a common tangent of $B_2$ and $B_\infty$.
    The axis of $sr'$ is perpendicular to the rotation axis of $r'$.
    Both rotations $r'$ and $sr'$ have to map full-sized horoballs touching $B_2$ to full-sized horoballs touching $B_2$.
    Hence one of these rotations has to fix two of the touching full-sized horoballs.
    (If $r'$ does not fix two full-sized horoballs touching $B_2$, then it has to map two neighbouring full-sized horoballs touching $B_2$ onto each other.
      The rotational axis of $r'$ then has to pass between those horoballs.
      Since the rotational axis of $sr'$ is orthogonal, it has to pass through the centres of two full-sized horoballs touching $B_2$.)
    Due to symmetry it is enough to consider the cases where $B_1$ or $B_3$ is fixed.
    Assuming that $r'$ fixes $B_1$, then the axis $(u_1, u_2)$ has to have order 6 because of its image $(u_1,\infty)$.
    However, we already noted that it has order 2.
    Assuming that $r'$ fixes $B_3$ centred at $b_3=(u_3,1)$, then the axis $(u_2,u_3)$ has to have order 2 -- the same as its image $(u_3,\infty)$.
    The geodesic $(u_2, u_4)$ related to $b_2$ and $b_4=(u_4,1)$, has to have order 6 as its image $(u_1, \infty)$.
    However, they have to have the same order as they get mapped to each other by rotations around the singular point $a_3$ representing the centre of the cusp diagram $D$.
  \end{proof}
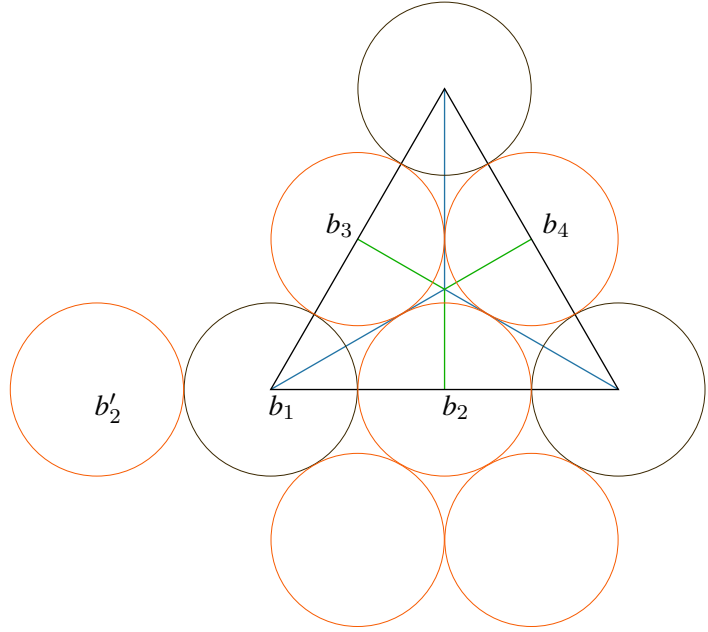
\begin{figure}[h]
  \centering
  \def\svgwidth{0.7\textwidth}
\begingroup%
  \makeatletter%
  \providecommand\color[2][]{%
    \errmessage{(Inkscape) Color is used for the text in Inkscape, but the package 'color.sty' is not loaded}%
    \renewcommand\color[2][]{}%
  }%
  \providecommand\transparent[1]{%
    \errmessage{(Inkscape) Transparency is used (non-zero) for the text in Inkscape, but the package 'transparent.sty' is not loaded}%
    \renewcommand\transparent[1]{}%
  }%
  \providecommand\rotatebox[2]{#2}%
  \newcommand*\fsize{\dimexpr\f@size pt\relax}%
  \newcommand*\lineheight[1]{\fontsize{\fsize}{#1\fsize}\selectfont}%
  \ifx\svgwidth\undefined%
    \setlength{\unitlength}{736.64939268bp}%
    \ifx\svgscale\undefined%
      \relax%
    \else%
      \setlength{\unitlength}{\unitlength * \real{\svgscale}}%
    \fi%
  \else%
    \setlength{\unitlength}{\svgwidth}%
  \fi%
  \global\let\svgwidth\undefined%
  \global\let\svgscale\undefined%
  \makeatother%
  \begin{picture}(1,0.89954583)%
    \lineheight{1}%
    \setlength\tabcolsep{0pt}%
    \put(0,0){\includegraphics[width=\unitlength,page=1]{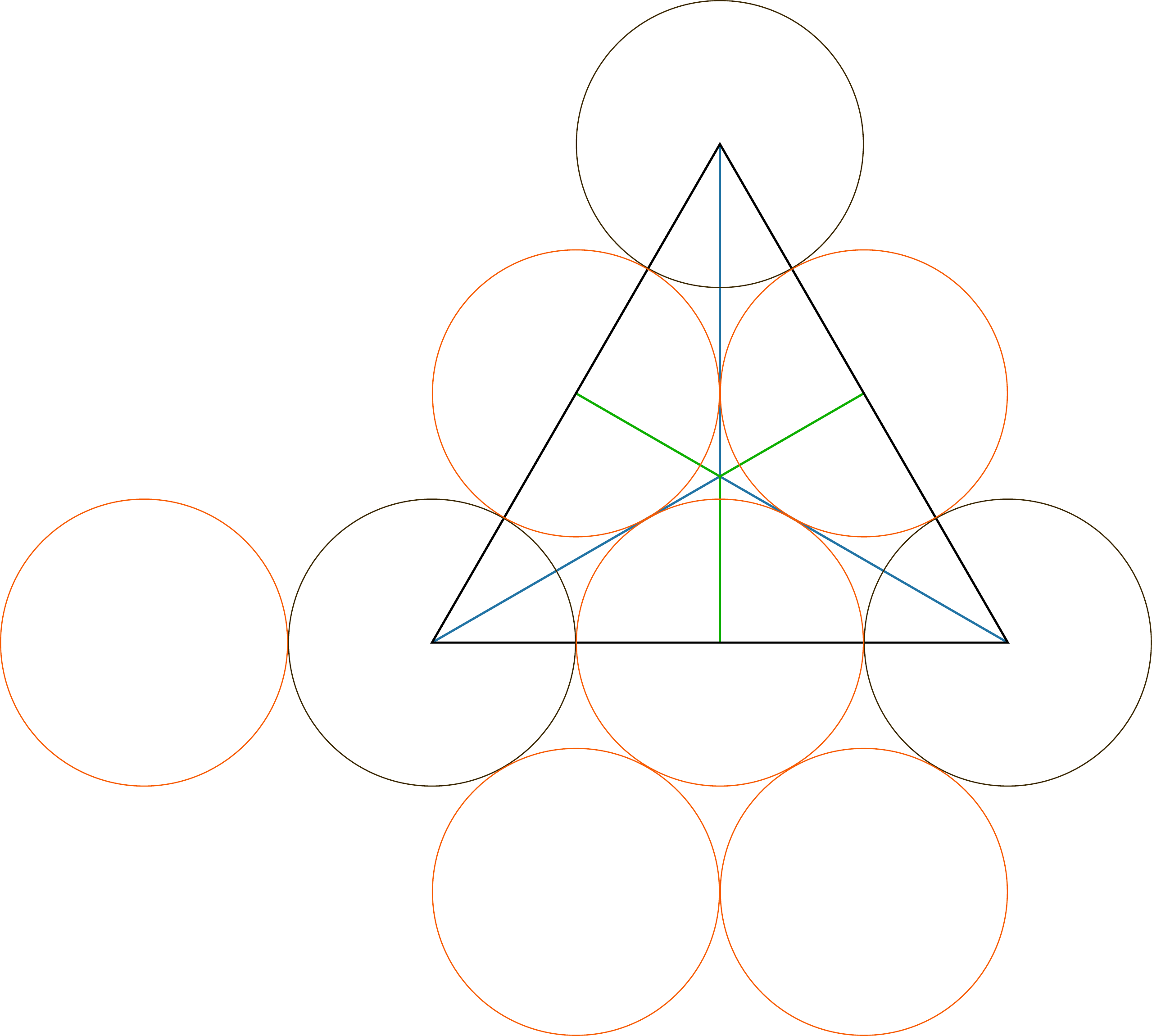}}%
    \put(0.36902938,0.30466622){\color[rgb]{0,0,0}\makebox(0,0)[lt]{\lineheight{1.25}\smash{\begin{tabular}[t]{l}$b_1$\end{tabular}}}}%
    \put(0.61896286,0.30466622){\color[rgb]{0,0,0}\makebox(0,0)[lt]{\lineheight{1.25}\smash{\begin{tabular}[t]{l}$b_2$\end{tabular}}}}%
    \put(0.11909267,0.30466622){\color[rgb]{0,0,0}\makebox(0,0)[lt]{\lineheight{1.25}\smash{\begin{tabular}[t]{l}$b_2'$\end{tabular}}}}%
    \put(0.45182041,0.56743775){\color[rgb]{0,0,0}\makebox(0,0)[lt]{\lineheight{1.25}\smash{\begin{tabular}[t]{l}$b_3$\end{tabular}}}}%
    \put(0.76336199,0.56743775){\color[rgb]{0,0,0}\makebox(0,0)[lt]{\lineheight{1.25}\smash{\begin{tabular}[t]{l}$b_4$\end{tabular}}}}%
  \end{picture}%
\endgroup%

  \caption{Two equivalence classes of full-sized horoballs centred in $a_6$ and $a_2$}
  \label{fig:236-2and6}
\end{figure}

Proposition \ref{prop:e-big} implies that the distance between two inequivalent full-sized horoballs satisfies $e>1$.
Again, consider the associated $(\frac{1}{e})$-balls. The minimal distance $e$ is achieved if a $(\frac{1}{e})$-ball lies on the angle bisector in $D$ and touches three full-sized horoballs as in Figure \ref{fig:236-2and6-touching}.
\begin{figure}[h]
  \centering
  \includegraphics[width=0.5\textwidth]{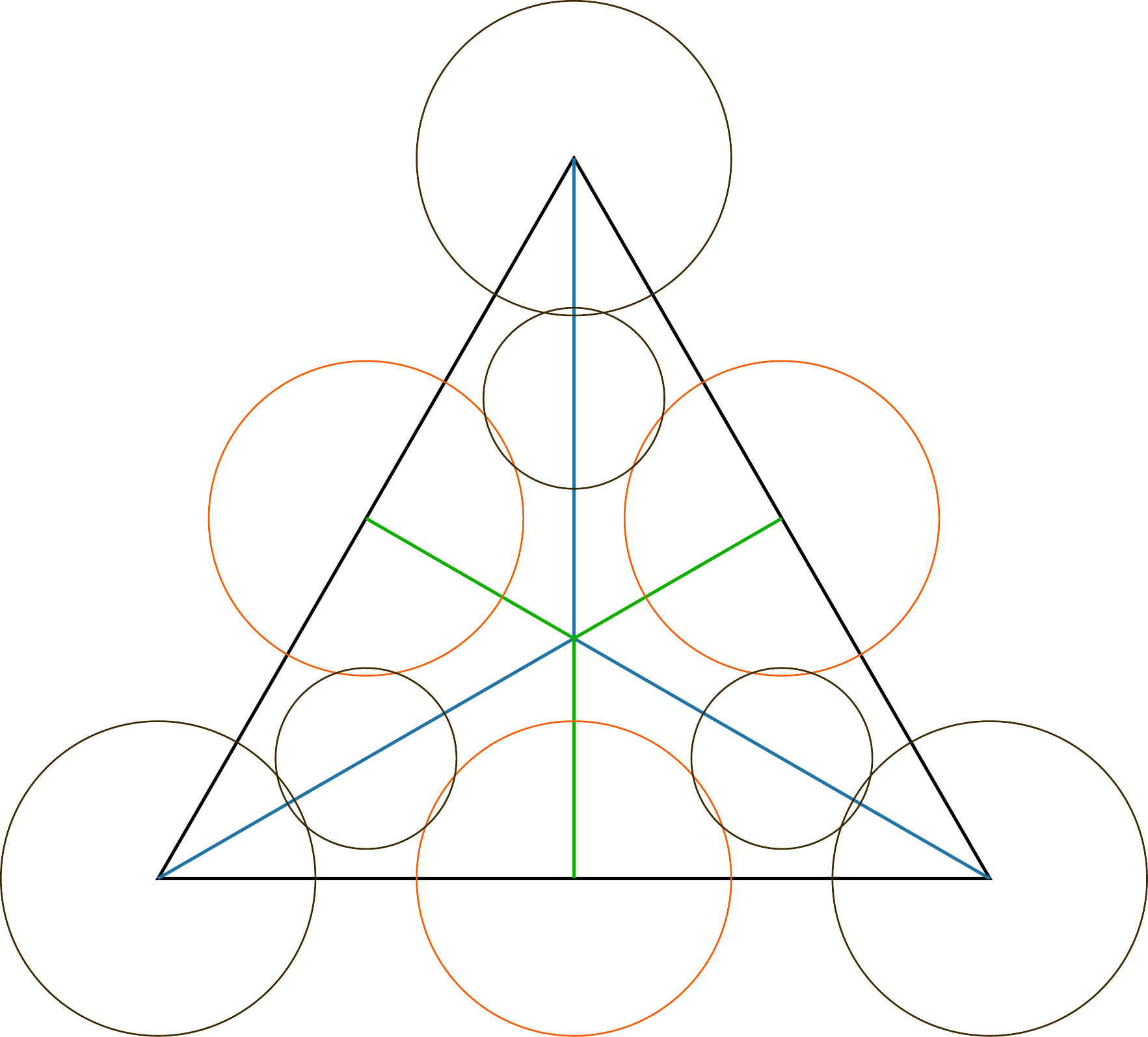}
  \caption{A $(\frac{1}{e})$-ball centred on the angle bisector in $D$}
  \label{fig:236-2and6-touching}
\end{figure}
This gives the lower bound $e\ge \sqrt[4]{3}$. As a consequence, the (oriented) cusp volume yields the estimate
\begin{equation*}
  \vol(C)= \frac{e^2}{2\sqrt{3}}\ge\frac{1}{2}>2\,v_*\,\,,
\end{equation*}
so that the corresponding orbifold volume becomes too big.

\subsubsection{The case $\{2,4,4\}$}\label{subsubsec:2}

Assume that there is a $\{2,4,4\}$-cusp with at least two equivalence classes of full-sized horoballs.
As in Section \ref{subsubsec:1}, we can exclude the case of  three equivalence classes, since the smallest volume would arise if all the full-sized horoballs were centred in the singular points
yielding a (oriented) cusp volume $\vol(C)=\frac{1}{2}>2\,v_*$.

Furthermore, given two equivalence classes of full-sized horoballs, these horoballs have to be centred in the singular points because otherwise the cusp volume would be too big.
If there are full-sized horoballs in $4$- and $2$-fold singular points, the shortest translation length satisfies $\tau\ge 2$ which gives a cusp volume $\vol(C)=\frac{\tau^2}{8}\ge\frac{1}{2}>2\,v_*$.
Thus we only have to consider the case where the full-sized horoballs are centred in the $4$-fold singularities.

\begin{figure}[h]
  \centering
  \def\svgwidth{0.5\textwidth}
\begingroup%
  \makeatletter%
  \providecommand\color[2][]{%
    \errmessage{(Inkscape) Color is used for the text in Inkscape, but the package 'color.sty' is not loaded}%
    \renewcommand\color[2][]{}%
  }%
  \providecommand\transparent[1]{%
    \errmessage{(Inkscape) Transparency is used (non-zero) for the text in Inkscape, but the package 'transparent.sty' is not loaded}%
    \renewcommand\transparent[1]{}%
  }%
  \providecommand\rotatebox[2]{#2}%
  \newcommand*\fsize{\dimexpr\f@size pt\relax}%
  \newcommand*\lineheight[1]{\fontsize{\fsize}{#1\fsize}\selectfont}%
  \ifx\svgwidth\undefined%
    \setlength{\unitlength}{205.22839294bp}%
    \ifx\svgscale\undefined%
      \relax%
    \else%
      \setlength{\unitlength}{\unitlength * \real{\svgscale}}%
    \fi%
  \else%
    \setlength{\unitlength}{\svgwidth}%
  \fi%
  \global\let\svgwidth\undefined%
  \global\let\svgscale\undefined%
  \makeatother%
  \begin{picture}(1,1.29343372)%
    \lineheight{1}%
    \setlength\tabcolsep{0pt}%
    \put(0,0){\includegraphics[width=\unitlength,page=1]{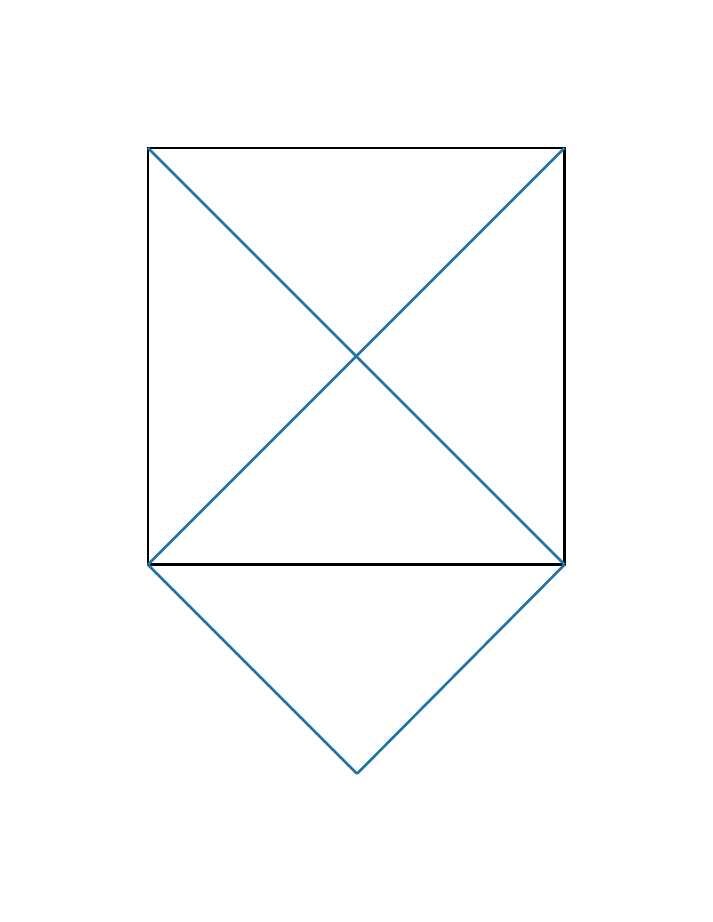}}%
    \put(0.46616775,0.45314404){\color[rgb]{0,0,0}\makebox(0,0)[lt]{\lineheight{1.25}\smash{\begin{tabular}[t]{l}$\tau$\end{tabular}}}}%
    \put(0.36770659,0.69211138){\color[rgb]{0,0,0}\makebox(0,0)[lt]{\lineheight{1.25}\smash{\begin{tabular}[t]{l}$e$\end{tabular}}}}%
    \put(0.16953156,0.45314404){\color[rgb]{0,0,0}\makebox(0,0)[lt]{\lineheight{1.25}\smash{\begin{tabular}[t]{l}$b_2$\end{tabular}}}}%
    \put(0.52752008,0.77551824){\color[rgb]{0,0,0}\makebox(0,0)[lt]{\lineheight{1.25}\smash{\begin{tabular}[t]{l}$b_1$\end{tabular}}}}%
    \put(0.77617327,0.45314404){\color[rgb]{0,0,0}\makebox(0,0)[lt]{\lineheight{1.25}\smash{\begin{tabular}[t]{l}$b_2'$\end{tabular}}}}%
    \put(0.50849977,0.16148669){\color[rgb]{0,0,0}\makebox(0,0)[lt]{\lineheight{1.25}\smash{\begin{tabular}[t]{l}$b_1'$\end{tabular}}}}%
    \put(0,0){\includegraphics[width=\unitlength,page=2]{442-2equiclasses.pdf}}%
  \end{picture}%
\endgroup%

  \caption{Two equivalence classes of full-sized horoballs in
  the diagram of a cusp of type $\{2,4,4\}$}
  \label{fig:244twoEqui}
\end{figure}

However, for a cusp of type $\{2,4,4\}$, there is no equivalent of Lemma \ref{lem:hororotation} since there are
non-conjugate rotations of the same order in $\Gamma_\infty$.
This is the reason why there are two cases to distinguish if one assumes that the full-sized horoballs, representing different equivalence classes, touch each other. We analyse this situation first.

$\bullet\quad$ Suppose that $e=1$.
There is an (orientation preserving) isometry $r\in\Gamma$ mapping the full-sized horoball $B_1$ centred at $b_1=(u_1,1)$ to $B_\infty$.
After possibly using an isometry in $\Gamma_\infty$, we can assume that $r$ maps $B_\infty$ to $B_1$ or to $B_2$ where $B_2$ is an inequivalent full-sized horoball centred in $b_2=(u_2,1)$ in $D$, respectively; see Figure \ref{fig:244twoEqui}.

{\it Case 1.} \quad
If $r$ sends $\infty$ to $u_1$, then $r$ is a rotation of order 2 and has to map the 4 touching full-sized horoballs onto each other.
There are 4 such rotations, and any one of them helps us.
In fact, the order 2 rotations around the midpoints $a_2$ of the sides of the square diagram $D$ with side lengths $\tau$ and their conjugates by $r$ give a side pairing of the ideal regular octahedron $O_\text{reg}^\infty$ formed by the ideal points defined by the 4 vertices and the in-centre of the square diagram $D$ and $\infty$.
Since this side pairing generates a finite index subgroup of $\Gamma$ as well as a finite index subgroup of the arithmetic reflection group $[4,4,4]$, they are all commensurable. In particular, $\Gamma$ has to be arithmetic.

{\it Case 2.} \quad 
If $r$ sends $\infty$ to $u_2$,
we can assume that $r(u_2)=u_1$.
That implies that $r$ is a rotation of order 3, which permutes $\infty$, $u_1$ and $u_2$.
It also permutes the centres  $b_1'$, $b_2'$ of neighbouring full-sized horoballs and a $(\frac{1}{e})$-ball touching all of them as depicted in Figure \ref{fig:244twoEqui}.
Denote by $s$ the rotation of order 2 around the vertical axis defined by the centre of this $(\frac{1}{e})$-ball in $\mathbb H^3$. Then, $rsr^{-1}$ is a rotation of order 2 exchanging $B_\infty$ and the $(\frac{1}{e})$-ball.
It follows that the two rotations of order 4 around the axes associated to $b_1$ and $b_1'$, together with their conjugates by $rsr^{-1}$, give a side-pairing of the ideal regular octahedron $O_\text{reg}^\infty$.
With the same reasoning as above, $\Gamma$ turns out to be arithmetic.

$\bullet\quad$ Suppose that $e>1$. We can use the ideas from the proofs of Lemma \ref{lem:hororotation} and Adams' Lemma \ref{elliptic} as follows.
A full-sized horoball $B_u$ centred at $u\in\partial\HH^3\setminus\{\infty\}$ can be mapped to $B_\infty$, and $B_\infty$ gets mapped to a horoball $B_v$ centred at $v$.
By hypothesis, there is a full-sized horoball at distance $e$ from $B_u$ which gets mapped to a horoball touching $B_v$ and of diameter $\frac{1}{e^2}$.
This is the equivalent of the $(\frac{1}{e})$-ball as above.

\begin{figure}[h]
  \centering
  \def\svgwidth{0.5\textwidth}
\begingroup%
  \makeatletter%
  \providecommand\color[2][]{%
    \errmessage{(Inkscape) Color is used for the text in Inkscape, but the package 'color.sty' is not loaded}%
    \renewcommand\color[2][]{}%
  }%
  \providecommand\transparent[1]{%
    \errmessage{(Inkscape) Transparency is used (non-zero) for the text in Inkscape, but the package 'transparent.sty' is not loaded}%
    \renewcommand\transparent[1]{}%
  }%
  \providecommand\rotatebox[2]{#2}%
  \newcommand*\fsize{\dimexpr\f@size pt\relax}%
  \newcommand*\lineheight[1]{\fontsize{\fsize}{#1\fsize}\selectfont}%
  \ifx\svgwidth\undefined%
    \setlength{\unitlength}{184.67595487bp}%
    \ifx\svgscale\undefined%
      \relax%
    \else%
      \setlength{\unitlength}{\unitlength * \real{\svgscale}}%
    \fi%
  \else%
    \setlength{\unitlength}{\svgwidth}%
  \fi%
  \global\let\svgwidth\undefined%
  \global\let\svgscale\undefined%
  \makeatother%
  \begin{picture}(1,1.00000015)%
    \lineheight{1}%
    \setlength\tabcolsep{0pt}%
    \put(0,0){\includegraphics[width=\unitlength,page=1]{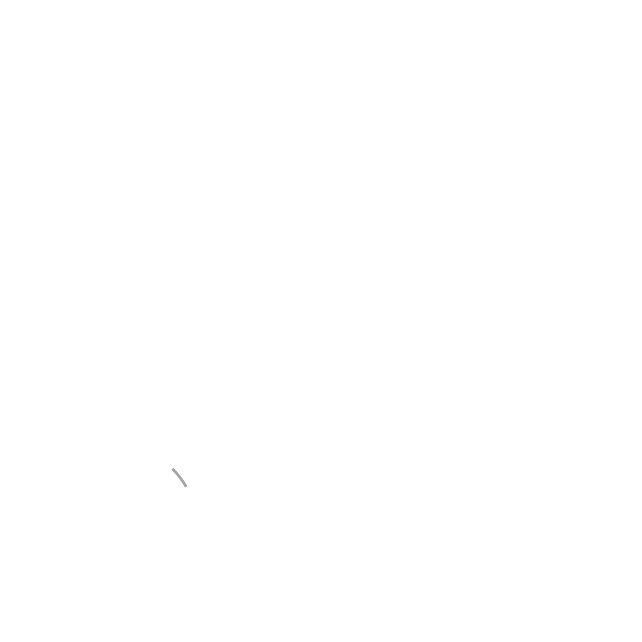}}%
    \put(0.2469702,0.23){\color[rgb]{0,0,0}\makebox(0,0)[lt]{\lineheight{1.25}\smash{\begin{tabular}[t]{l}$\alpha$\end{tabular}}}}%
    \put(0,0){\includegraphics[width=\unitlength,page=2]{442-2equiclasses-setting.pdf}}%
    \put(0.47,0.11){\color[rgb]{0,0,0}\makebox(0,0)[lt]{\lineheight{1.25}\smash{\begin{tabular}[t]{l}$\tau$\end{tabular}}}}%
    \put(0.34835457,0.22663965){\color[rgb]{0,0,0}\makebox(0,0)[lt]{\lineheight{1.25}\smash{\begin{tabular}[t]{l}$\frac{1}{e}$\end{tabular}}}}%
    \put(0.40067019,0.43313353){\color[rgb]{0,0,0}\makebox(0,0)[lt]{\lineheight{1.25}\smash{\begin{tabular}[t]{l}$e$\end{tabular}}}}%
  \end{picture}%
\endgroup%

  \caption{Two inequivalent full-sized horoballs at distance $e$ with $(\frac{1}{e})$-ball in the diagram of a cusp of type $\{2,4,4\}$}
  \label{fig:244twoEqui1de}
\end{figure}

The minimal possible distance $e$ is obtained if the $(\frac{1}{e})$-ball touches inequivalent full-sized horoballs as illustrated in
Figure \ref{fig:244twoEqui1de}.
We can perform the same calculation which led to equation \eqref{eq:1de236} by using the isosceles triangle with angle $\alpha$ satisfying $\alpha\in[0,\frac{\pi}{4}]$.
We obtain the lower bound according to
\begin{equation*}
  e = \sqrt{2\cos\alpha}\ge\sqrt{2\cos\frac{\pi}{4}}=\sqrt[4]{2}\,\,.
\end{equation*}
Hence, the (oriented) cusp volume yields the estimate
\begin{equation*}
  \vol(C) = \frac{e^2}{4}\ge\frac{\sqrt{2}}{4}>0.35>2\,v_*.
\end{equation*}

As a summary of the investigations in Sections \ref{subsubsec:1} and \ref{subsubsec:2}, we can formulate the following result.

\begin{proposition}\label{prop:2equivclasses}
A non-arithmetic cusped hyperbolic $3$-orbifold $\mathbb H^3/\Gamma$ of minimal volume with cusp of the form $C=B_{\infty}/\Gamma_{\infty}$ has the property that
$\Gamma_{\infty}$ permutes all full-sized horoballs.
\end{proposition}

In combination with Proposition \ref{prop:case1}, Proposition \ref{prop:case2}, Proposition
\ref{prop:case236} and Proposition \ref{prop:case244}, Proposition
\ref{prop:2equivclasses} is the final step in the proof of the Theorem 
as stated in Section \ref{sec:intro} and in Section \ref{sec:chapter3}.

\section{Final remarks}\label{sec:chapter4}
The theory and investigations exploited in this work are  valid for cusped hyperbolic $n$-orbifolds of dimensions $n\ge2$. While for $n\ge 4$
the adaption becomes very cumbersome, one can derive in a fairly easy way the following corresponding result for $n=2$ mentioned in  Section \ref{sec:intro} (for details, see \cite{Drewitz}).

\begin{proposition}\label{prop:min2orbi}
Among all non-arithmetic cusped hyperbolic 2-orbifolds, the 1-cusped quotient space of $\HH^2$ by the triangle Coxeter group $[5,\infty]$ has minimal area. As such the orbifold is unique, and its area is given
by $\frac{3\pi}{10}$.
\end{proposition}

By Selberg's Lemma, any cusped
hyperbolic orbifold has a finite sheeted non-compact cover manifold.
Consider the non-arithmetic cusped 3-orbifold $V_*=\HH^3/[5,3,6]$
of smallest volume in this class.
The minimal index of a torsion-free subgroup of $[5,3,6]$ is given by the l.c.m. of the orders of finite subgroups of $[5,3,6]$ and equals 120. In \cite{Everitt}, it is shown that there are 10 such subgroups which are non-conjugate in $[5,3,6]$, yielding non-isometric non-arithmetic hyperbolic 3-manifolds, all orientable with one or two cusps, covering the orbifold $V_*$. Each of these manifolds is of volume $120\cdot \vol([5,3,6])\approx 20.580199$.

Compare $V_*$
with the non-arithmetic 2-cusped orbifold $V_{\circ}$ of volume
\[
\frac{5}{8}\,\loba(\frac{\pi}{3})+\frac{1}{3}\,\loba(\frac{\pi}{4})\approx 0.364107\,\,
\]
whose fundamental group is given by the tetrahedral Coxeter group $[(3^3,6)]$ (see Example \ref{ex-nonarithm} and \eqref{eq:volcycle}). There is a torsion-free subgroup $\Lambda$ of $[(3^3,6)]$ of minimal possible index equal to $24$
so that the resulting manifold $M=\HH^3/\Lambda$ is of volume
$\approx 8.738570$. Since the groups $[5,3,6]$ and $[(3^3,6)]$
are incommensurable (see Remark \ref{incomm}), the manifold $M$
does not have a common covering manifold with $V_*$.
As a consequence, one gets the following result stated in Section \ref{sec:intro} (for details, see \cite{Drewitz}).
\begin{proposition}\label{prop:min3mani}
The fundamental group of a non-arithmetic cusped hyperbolic 3-manifold of minimal volume is incommensurable to the Coxeter group $[5,3,6]$; its volume is smaller than or equal to $24\cdot \hbox{\rm vol}([(3^3,6)])\approx 8.738570$.
\end{proposition}

\appendix
\section*{Appendix}
\renewcommand{\thesubsection}{\Alph{subsection}}

\subsection{The case $\{2,3,6\}$}\label{subsec:appendix1}
Consider a discrete group $\Gamma\subset\hbox{Isom}^+\mathbb H^3$
of orientation preserving isometries whose quotient space $O$ has exactly one cusp of the form $C=B_{\infty}/\Gamma_{\infty}$ giving rise to full-sized horoballs of diameter $1$ in the upper half space $U^3$.

Suppose that $C$ is of type $\{2,3,6\}$ and that the group $\Gamma_{\infty}$ induces only one equivalence class of full-sized horoballs, all centred at singular points of order $6$ in the cusp diagram $D$.

In the sequel, we consider the special case when the full-sized horoballs are at distance $d>1$ from one another and when the associated $(\frac{1}{d})$-balls do not touch one another. In particular, these $(\frac{1}{d})$-balls give rise to $(\frac{1}{w})$-balls.
Suppose that the $(\frac{1}{w})$-balls of the three $(\frac{1}{d})$-balls in $D$ coincide with center at the singular point $p=a_3$ of order 3 in $D$; see Figure \ref{fig:appendix-236}.
\begin{figure}
  \centering
  \def\svgwidth{0.5\textwidth}
  \footnotesize
  \input{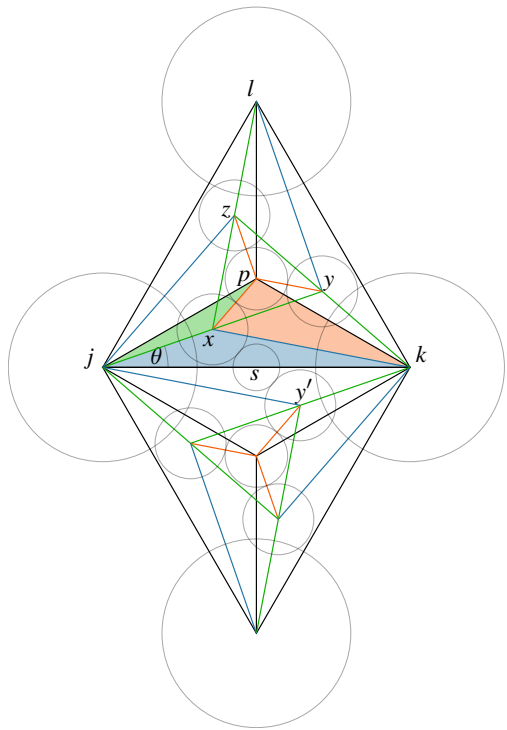}
  \caption{A single $(\frac{1}{w})$-ball in a $\{2,3,6\}$-cusp triangle $D$}
  \label{fig:appendix-236}
\end{figure}
As explained in Section \ref{step3a}, one obtains $w=\frac{\sqrt{3}}{d}$,  $\cos\theta=\frac{5}{2\sqrt{7}}$ and $d=\sqrt[\leftroot{-2}\uproot{2}4]{7}$. In particular, $\vol(C)=\frac{\sqrt{21}}{24}$.

In the following, we correct Adams' determination
of a fundamental polyhedron for $\Gamma$ and the volume computation for $O$ as given in \cite[p. 10]{Adams1}. Furthermore, we show that $O$ is arithmetic (see Remark \ref{21-24} in Section \ref{step3a}).

Consider the full-sized horoballs $B_j$, $B_k$ and $B_l$, as well as the
$(\frac{1}{d})$-balls centred at $x$, $y$ and $z$ as depicted in Figure \ref{fig:appendix-236}. Identify all horoball centres at height $t=1$ in $D$ with the corresponding points in the plane $\{\,t=0\}$.
Recall that the centres $j$, $x$, and $y$ are located on one line.
Denote by $r$ the order 2 rotation which interchanges the horoballs $B_\infty$ and $B_j$ as well as $B_k$ and $B_x$.

The $(\frac{1}{w})$-ball $B_s$ is actually the image of $B_y$ under this rotation $r$ and centred in $a_2$.
It is fixed by a rotation $s\in\Gamma_\infty$ of order 2 around a singular point $a_2$.
\newline
The image of the axis connecting $B_\infty$ and $B_p$ under the rotation $r$ is the geodesic connecting $B_j$ and $B_y$.
Hence the conjugate rotation $s':=rsr$ swaps $B_x$ and $B_\infty$ while fixing $B_j$ and $B_y$.

It is easy enough to find a fundamental polyhedron for a subgroup of $\Gamma$ described by the cusp diagram.
Consider the ideal tetrahedron with vertices $\infty$, $j$, $k$, and $p$.
Each of its ideal vertex triangles has angles $\frac{2\pi}{6}$, $\frac{\pi}{6}$ and
$\frac{\pi}{3}$.
Denote by $\rho$ the clockwise rotation of order 3 around the singular axis connecting the centres of $B_\infty$ and $B_p$.
The symmetries $s$, $s'$, and $\rho$ give face identifications which are visualised in Figure \ref{fig:faceids236}.
The faces $A$ and $B$ are identified with themselves by the appropriate rotation around the given axes.
The face $C$ gets mapped to $D$ by $\rho$.
This gives a {\it subproper side pairing} according to \cite[Section 13.4]{Rat-book} implying that $P$ is a fundamental polyhedron for a subgroup of finite index in $\Gamma$.
Denote this subgroup by $\Theta$.

\begin{figure}
  \centering
  \def\svgwidth{0.5\textwidth}
  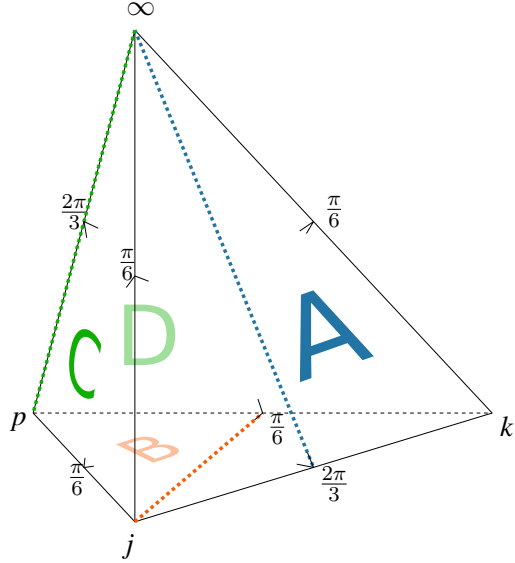
  \caption{Side pairings of $P$ for a subgroup of $\Gamma$}
  \label{fig:faceids236}
\end{figure}

From this we can deduce the arithmeticity of $\Gamma$ as follows.
A representation of $\Theta$ in $\PSL_2(\CC)$ is given by
\begin{align*}
  s&\mapsto \begin{pmatrix}
      i & -i \\
      0 & -i \\
    \end{pmatrix}, &
  \rho & \mapsto \begin{pmatrix}
      -\omega & \omega^2\\
      0 & \omega^2
    \end{pmatrix}, &
  s' & \mapsto \begin{pmatrix}
      i & 0\\
      i \,d^2 \,\ee^{-i\,\beta} & -i
    \end{pmatrix}
    =\begin{pmatrix}
        i & 0\\
        -\frac{\sqrt{3}}{2}+i\frac{5}{2} & -i
      \end{pmatrix},
\end{align*}
where $\omega=\frac{1}{2}+i\frac{\sqrt{3}}{2}$ is a primitive sixth root of unity.
The trace field is $k\Gamma=\QQ(\omega)$.
This can be seen by utilising \cite[Theorem 3.5.9]{MR1} and by taking $\rho^{-1}$, $\rho^{-1} s$, and $s\,rsr$ as generators for the subgroup $\Theta$.
By means of \cite[Theorem 8.3.2]{MR1}, it follows that $\Gamma$ is arithmetic, and the result \cite[Theorem 8.4.1]{MR1} in conjunction with \cite[Theorem 3.3.8]{MR1} allows us to deduce that $\Gamma$ is even commensurable to the $[3,3,6]$.

In order to obtain $\Gamma$, one needs to add $r$ as a generator to $\Theta$.
Replacing the generator $\rho$ by $s\rho$ (a rotation around the axis connecting $\infty$ and $j$) makes it easy to show that $\Gamma = \Theta\rtimes \langle r\rangle$ is a semidirect product.
The subgroup $\Theta$ has index $2$ in $\Gamma$, so half of $P$ is a fundamental polyhedron for $\Gamma$.
This allows us to calculate the volume of the fundamental polyhedron $P$ and of the orbifold $O$ as follows (see Section \ref{volume}).
\begin{equation*}
  \vol(O)=\text{vol}(P) = \frac{1}{2}\,\left\lbrace\loba\left( \frac{2\pi}{3} \right)+2\,\loba\left( \frac{\pi}{6}\right) \right\rbrace
   = \loba\left( \frac{\pi}{3} \right) \approx 0.338314\,\,.
\end{equation*}


\begin{remark}\label{loba-ident1}
Based on Adams' original explanation \cite[p. 10]{Adams1}, there is a dissection of the polyhedron $P$ and its image under $rsr$ into three ideal tetrahedra $T(\infty, j, k, x)$, $T(\infty, j, p, x)$, and $T(\infty, k, p, x)$.
  They are represented by the coloured triangles in Figure \ref{fig:appendix-236}.
  The blue and green tetrahedra have angles $\theta, \frac{\pi}{6}-\theta$ and $\frac{5\pi}{6}$, where $\cos\theta=\frac{5}{2\sqrt{7}}$ as above.
  The red tetrahedron has angles $\theta$, $\frac{\pi}{3}-\theta$, and $\frac{2\pi}{3}$.
  This is useful in order to obtain the following new identity\footnote{
    To the best of the authors knowledge, this identity is new.
  } for the Lobachevsky function evaluated at $\frac{\pi}{3}$.
  \begin{align*}
    \loba\left( \frac{\pi}{3} \right) & =
    \frac{1}{4}\left(
      2\left( \loba\left( \theta \right)+\loba\left( \frac{\pi}{6}-\theta \right)+\loba\left( \frac{5\pi}{6} \right) \right)
      +\loba\left( \theta\right)+\loba\left( \frac{2\pi}{3}-\theta \right)+\loba\left( \frac{\pi}{6} \right) \right)\\
       & = \frac{3}{4}\,\loba\left( \theta \right)+\frac{1}{2}\,\loba\left( \frac{\pi}{6}-\theta \right)+\frac{1}{4}\,\loba\left( \frac{2\pi}{3}-\theta \right)-\frac{1}{4}\,\loba\left( \frac{\pi}{6} \right)\quad\hbox{\rm with}\quad \cos\theta=\frac{5}{2\sqrt{7}}\,\,.
  \end{align*}
\end{remark}


\subsection{The case $\{2,4,4\}$}\label{subsec:appendix2}
As in Appendix \ref{subsec:appendix1}, consider a discrete group $\Gamma\subset\hbox{Isom}^+\mathbb H^3$
of orientation preserving isometries whose quotient space $O$ has exactly one cusp of the form $C=B_{\infty}/\Gamma_{\infty}$ giving rise to full-sized horoballs of diameter $1$ in the upper half space $U^3$.

Suppose that $C$ is of type $\{2,4,4\}$ and that the group $\Gamma_{\infty}$ has only one equivalence class of full-sized horoballs, all centred at singular points of order $4$ in the cusp diagram $D$, and at distance $d>1$ from one another.
Assume that the associated $(\frac{1}{d})$-balls do not touch one another so that each $(\frac{1}{d})$-ball yields $(\frac{1}{w})$-balls.

Suppose that the four $(\frac{1}{w})$-balls associated to the four
$(\frac{1}{d})$-balls in $D$ coincide and, hence, have  centre equal to the centre of the square $D$. Furthermore, the circumradius of $D$ coincides with $\frac{1}{w}$ so that $w=\frac{\sqrt{2}}{d}$.

In a completely analogous manner as in Appendix \ref{subsec:appendix1}, and by means of \eqref{eq:3}, one can show that $d=\sqrt[\leftroot{-2}\uproot{2}4]{5}$ and $\cos\theta=\frac{2}{\sqrt{5}}$. Since the cusp diagram $D$ has no mirror symmetry, the cusp volume yields the large value $\vol(C)=\frac{\sqrt{5}}{8}>v_*$; see Figure \ref{fig:einsdurchw244}.

Furthermore, by performing similar computations as in Appendix \ref{subsec:appendix1}, one can deduce that this configuration gives rise to a unique oriented orbifold $O=\mathbb H^3/\Gamma$ that is arithmetic and in fact commensurable to the Coxeter group $[3,4,4]$.

\begin{figure}
  \centering
  \def\svgwidth{0.5\textwidth}
  \footnotesize
  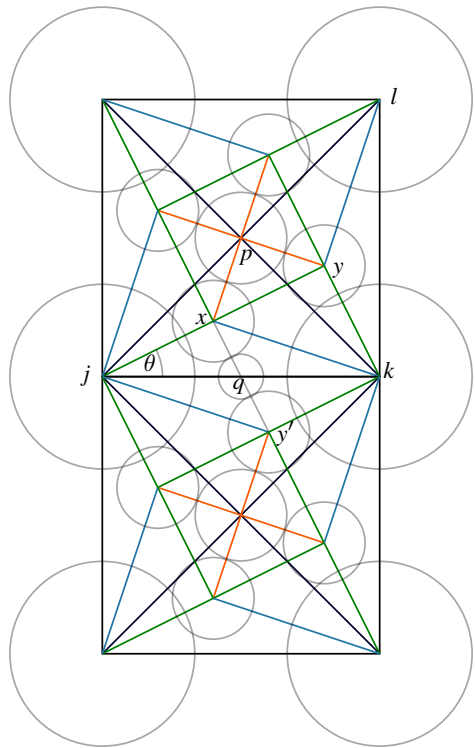
  \caption{The cusp diagram $D$ for the cusp of type $\{2,4,4\}$ with a single $(\frac{1}{w})$-ball at its centre}
  \label{fig:einsdurchw244}
\end{figure}

\begin{remark}\label{loba-ident2}
By a suitable decomposition of a fundamental polyhedron for $\Gamma$ as above, a similar identity for the value $\loba\left(\frac{\pi}{4} \right)$
of the Lobachevsky function in terms of the angle $0<\theta<\frac{\pi}{4}$ satisfying
$\cos\theta=\frac{2}{\sqrt{5}}$ is obtained. This value also gives the volume of the orbifold $O$ as follows.
\begin{align*}
  \loba\left( \frac{\pi}{4} \right)
  & = \frac{1}{4}\left\{
    2\left( \loba\left( \theta \right)+\loba\left( \frac{\pi}{4}-\theta \right)+\loba\left( \frac{3\pi}{4} \right) \right)
  +\loba\left( \theta\right)+\loba\left( \frac{\pi}{2}-\theta \right)+\loba\left( \frac{\pi}{2} \right) \right\rbrace\\
  & = \frac{3}{4}\,\loba\left( \theta \right)
        +\frac{1}{2}\,\loba\left( \frac{\pi}{4}-\theta \right)
        +\frac{1}{4}\,\loba\left( \frac{\pi}{2}-\theta \right)
        -\frac{1}{4}\,\loba\left( \frac{\pi}{4} \right)\\
        & \approx 0.45983\quad \hbox{\rm with}\quad\cos\theta=\frac{2}{\sqrt{5}}\,\,.
\end{align*}
\end{remark}
\bibliography{536}

\providecommand{\bysame}{\leavevmode\hbox to3em{\hrulefill}\thinspace}
\providecommand{\MR}{\relax\ifhmode\unskip\space\fi MR }
\providecommand{\MRhref}[2]{%
  \href{http://www.ams.org/mathscinet-getitem?mr=#1}{#2}
}
\providecommand{\href}[2]{#2}
\begin{thebibliography}{10}

\bibitem{Adams4}
C.~Adams, \emph{The noncompact hyperbolic {$3$}-manifold of minimal volume},
  Proc. Amer. Math. Soc. \textbf{100} (1987), 601--606.

\bibitem{Adams2}
\bysame, \emph{Limit volumes of hyperbolic three-orbifolds}, J. Differential
  Geom. \textbf{34} (1991), 115--141.

\bibitem{Adams1}
\bysame, \emph{Noncompact hyperbolic {$3$}-orbifolds of small volume}, Topology
  '90 ({C}olumbus, {OH}, 1990), Ohio State Univ. Math. Res. Inst. Publ.,
  vol.~1, de Gruyter, Berlin, 1992, pp.~1--15.

\bibitem{Adams3}
\bysame, \emph{Volumes of hyperbolic {$3$}-orbifolds with multiple cusps},
  Indiana Univ. Math. J. \textbf{41} (1992), 149--172.

\bibitem{Andreev1}
E.~Andreev, \emph{Convex polyhedra in {L}oba\v{c}evski\u{\i} spaces}, Mat. Sb.
  (N.S.) \textbf{81 (123)} (1970), 445--478.

\bibitem{Andreev2}
\bysame, \emph{Convex polyhedra of finite volume in {L}oba\v{c}evski\u{\i}
  space}, Mat. Sb. (N.S.) \textbf{83 (125)} (1970), 256--260.

\bibitem{Boer}
K.~B\"{o}r\"{o}czky, \emph{Packing of spheres in spaces of constant curvature},
  Acta Math. Acad. Sci. Hungar. \textbf{32} (1978), 243--261.

\bibitem{Cox2}
H.~Coxeter, \emph{Arrangements of equal spheres in non-{E}uclidean spaces},
  Acta Math. Acad. Sci. Hungar. \textbf{5} (1954), 263--274.

\bibitem{Drewitz}
S.~T. Drewitz, \emph{New contributions to groups of hyperbolic isometries},
  Ph.D. thesis, University of Fribourg, 2021.

\bibitem{Everitt}
B.~Everitt, \emph{3-manifolds from {P}latonic solids}, Topology Appl.
  \textbf{138} (2004), 253--263.

\bibitem{Stover}
D.~Fisher, J.-L. Lafont, N.~Miller, and M.~Stover, \emph{Finiteness of maximal
  geodesic submanifolds in hyperbolic hybrids}, arXiv:math/1802.04619 (2018),
  28 pp.

\bibitem{Gugliel}
R.~Guglielmetti, \emph{Cox{I}ter - {C}omputing invariants of hyperbolic
  {C}oxeter groups}, LMS J. Comput. Math. \textbf{18} (2015), 754--773.

\bibitem{Hild2}
T.~Hild, \emph{Cusped hyperbolic orbifolds of minimal volume in dimensions less
  than 11}, Ph.D. thesis, no. 1567, University of Fribourg, 2007.

\bibitem{Hild1}
\bysame, \emph{The cusped hyperbolic orbifolds of minimal volume in dimensions
  less than ten}, J. Algebra \textbf{313} (2007), 208--222.

\bibitem{HK}
T.~Hild and R.~Kellerhals, \emph{The {FCC} lattice and the cusped hyperbolic
  4-orbifold of minimal volume}, J. Lond. Math. Soc. (2) \textbf{75} (2007),
  677--689.

\bibitem{JKRT1}
N.~Johnson, R.~Kellerhals, J.~Ratcliffe, and S.~Tschantz, \emph{The size of a
  hyperbolic {C}oxeter simplex}, Transform. Groups \textbf{4} (1999), 329--353.

\bibitem{JKRT2}
\bysame, \emph{Commensurability classes of hyperbolic {C}oxeter groups}, Linear
  Algebra Appl. \textbf{345} (2002), 119--147.

\bibitem{K0}
R.~Kellerhals, \emph{On the volume of hyperbolic polyhedra}, Math. Ann.
  \textbf{285} (1989), 541--569.

\bibitem{K3}
\bysame, \emph{Volumes of cusped hyperbolic manifolds}, Topology \textbf{37}
  (1998), 719--734.

\bibitem{K2}
\bysame, \emph{Scissors congruence, the golden ratio and volumes in hyperbolic
  5-space}, Discrete Comput. Geom. \textbf{47} (2012), 629--658.

\bibitem{MR1}
C.~Maclachlan and A.~Reid, \emph{The arithmetic of hyperbolic 3-manifolds},
  Graduate Texts in Mathematics, vol. 219, Springer-Verlag, New York, 2003.

\bibitem{Mey2}
R.~Meyerhoff, \emph{The cusped hyperbolic {$3$}-orbifold of minimum volume},
  Bull. Amer. Math. Soc. (N.S.) \textbf{13} (1985), 154--156.

\bibitem{Mey1}
\bysame, \emph{Sphere-packing and volume in hyperbolic {$3$}-space}, Comment.
  Math. Helv. \textbf{61} (1986), 271--278.

\bibitem{NR}
W.~Neumann and A.~Reid, \emph{Arithmetic of hyperbolic manifolds}, Topology '90
  ({C}olumbus, {OH}, 1990), Ohio State Univ. Math. Res. Inst. Publ., vol.~1, de
  Gruyter, Berlin, 1992, pp.~273--310.

\bibitem{Rat-book}
J.~Ratcliffe, \emph{Foundations of hyperbolic manifolds}, Graduate Texts in
  Mathematics, vol. 149, Springer-Verlag, New York, 1994.

\bibitem{Roeder}
R.~Roeder, J.~Hubbard, and W.~Dunbar, \emph{Andreev's theorem on hyperbolic
  polyhedra}, Ann. Inst. Fourier (Grenoble) \textbf{57} (2007), 825--882.

\bibitem{V1}
\`{E}. Vinberg, \emph{Hyperbolic groups of reflections}, Uspekhi Mat. Nauk
  \textbf{40} (1985), 29--66, 255.

\bibitem{V2}
\`{E}. Vinberg and O.~Shvartsman, \emph{Discrete groups of motions of spaces of
  constant curvature}, Geometry, {II}, Encyclopaedia Math. Sci., vol.~29,
  Springer, Berlin, 1993, pp.~139--248.

\end{thebibliography}
\bibliographystyle{amsplain}

\end{document}